\documentclass[11pt, oneside]{amsart}
\usepackage[utf8]{inputenc}
\usepackage[american]{babel}
\usepackage{amssymb}
\usepackage{hyperref}
\usepackage[a4paper]{geometry}
\usepackage{dsfont}
\usepackage{fullpage}
\usepackage{amsmath}
\usepackage[all]{xy}
\usepackage{enumitem}
\usepackage[dvipsnames]{xcolor}
\usepackage{graphicx}
\usepackage{stmaryrd}
\usepackage{amsfonts,amstext,amssymb,amsmath,amsthm,pspicture,bigints}
\usepackage{moreverb}
\usepackage{pifont}
\usepackage{esint}
\usepackage{fourier-orns}
\usepackage{mathrsfs}
\usepackage{array}

\usepackage{hyperref}
\newcommand{\ii }{{\rm i} } 
\newcolumntype{C}[1]{>{\centering\arraybackslash}b{#1}}
\newcolumntype{R}[1]{>{\raggedleft\arraybackslash}b{#1}}
\newcolumntype{L}[1]{>{\raggedright\arraybackslash}b{#1}}
\newcolumntype{M}[1]{>{\centering}m{#1}}
\newtheorem{thm}{Theorem}[section]

\newtheorem{lem}{Lemma}[section]
\newtheorem{prop}{Proposition}[section]

\theoremstyle{remark}
\newtheorem{rem}{Remark}[section]

\newtheorem{cor}{Corollary}[section]
\usepackage{bbm}
\numberwithin{equation}{section}

\renewcommand*{\div}{\operatorname{div}}
\newcommand*{\rad}{\text{rad}}
\newcommand*{\curl}{\operatorname{curl}}
\newcommand*{\id}{\operatorname{Id}}
\newcommand*{\supp}{\operatorname{supp}}
\newcommand*{\loc}{\mathrm{loc}}

\usepackage[textsize=tiny,textwidth=27mm]{todonotes}

\usepackage{graphicx,pifont}

\newcommand*{\norm}[1]{\left\Vert #1\right\Vert}

\newcommand*{\bydef}{\overset{\rm def}{=}}

\date{\today}
\author[1]{Taoufik Hmidi} 

\author[2]{Haroune Houamed}
\address{New York University Abu Dhabi \\
Abu Dhabi \\
United Arab Emirates}
\email{th2644@nyu.edu, haroune.houamed@nyu.edu} 

\author[3]{Emeric Roulley} 
\address{SISSA International School for Advanced Studies, Via Bonomea 265, 34136 \\
	Trieste \\
	Italy}
\email{eroulley@sissa.it} 

\author[4]{Mohamed Zerguine} 
\address{LEDPA, Université de Batna 2,   département des mathématiques, 05000\\
	BATNA\\
	ALGERIE}
\email{m.zerguine@univ-batna2.dz} 
\title[V-states for the lake equation]{Uniformly rotating vortices for the lake equation}
\begin{document}  

\begin{abstract}

	We investigate the existence of time-periodic vortex patch solutions, in both  simply and doubly-connected cases,  for the two-dimensional lake equation  where the depth function of the lake is assumed to be non-degenerate and radial. 
The proofs employ   bifurcation techniques, where the most challenging steps are related to  the regularity study of some nonlinear functionals and the spectral analysis of their linearized operators around Rankine type vortices. 
The main difficulties stem from   the roughness and the implicit form of the Green function connecting  the fluid vorticity and its stream function.  We handle in part these issues  by exploring the asymptotic structure of the solutions  to the associated elliptic problem. As to the distribution of the spectrum, it is tackled by a fixed-point argument through a perturbative approach.	 
\end{abstract} 
 
\maketitle
\tableofcontents
 
\section{Introduction and main results}
 
\subsection{Aspects of interest}
In this paper, we are interested in the study of time-periodic solutions of the lake equations, given by
\begin{equation}\label{LE} 
\begin{cases}
\partial_tv+v \cdot \nabla v +\nabla p=0,\\
\div(bv)=0,
\end{cases}
\end{equation}
supplemented with an initial datum. 
In this set of equations, the space variable, denoted by $x$ hereafter, takes values in the whole plane $\mathbb{R}^2$. The time-independent function $b:\mathbb{R}^2\rightarrow\mathbb{R}^+$ describes the bottom topography (or depth function) of the lake, whereas the vector field $v$ refers to the velocity field of the fluid.  Finally, the scalar function $p$ designates the pressure. Further assumptions on the coefficient  $b$ will be specified later on. The  equations \eqref{LE}, describing  balanced  flow without gravity waves, are basically recovered from the shallow-water equations in an appropriate time-space scale limit \cite {G80}, see also \cite{DK11}.   

One of the most important features of the  equations  \eqref{LE} is their equivalent vorticity formulation, which we recall next
\begin{equation}\label{Lake1}
\begin{cases}
\partial_t\left(\frac{\omega}{b}\right)+v\cdot\nabla\left(\frac{\omega}{b}\right)=0,\\
\div(bv)=0,\\
\omega=\curl v.
\end{cases}
\end{equation} 
Note that the divergence free condition in \eqref{Lake1} allows us to  express $bv$ as
\begin{equation}\label{def psib}
	bv=-\nabla^{\perp}\psi_{b},\qquad\nabla^{\perp}\bydef\begin{pmatrix}
		-\partial_{x_2}\\
		\partial_{x_1}
	\end{pmatrix},
\end{equation}
where the potential $\psi_b$ solves the elliptic problem 
\begin{equation}\label{elliptic:intro} 
		-\div\left(\frac{\nabla\psi_b}{b}\right)=\omega.
	\end{equation}

The lake equation, formulated as \eqref{Lake1}, has garnered considerable attention over the past few decades, resulting in numerous significant contributions from both analytical and numerical perspectives. Interested readers can find more comprehensive series of results on the well-posedness of \eqref{LE} in \cite{AL21,BM06,HLM21,LNP14,LOT96} and the references therein.   Most of these results are established in bounded domains, with or without island, under various assumptions on the depth function $b$.  For additional insights and progress on specific types of solutions, such as point vortices and steady states, we refer to  \cite{CZZ22, DJ19, DJ22, HLM22, M23, Y21}.

It is worth noting that the vorticity reformulation \eqref{Lake1} shares structural similarities with the classical Euler equations, governing the motion of an ideal incompressible fluid. This similarity is observed by setting $b\equiv 1$ in the equations above. Additionally, under reasonable conditions on the depth function $b$, a Yudovich theory for the lake equation \eqref{Lake1} can be established, akin to its known statement for the Euler equations \cite{Y63}. As a consequence of such a theory, the structure of an initial patch
$$ \omega|_{t=0}=b \mathds{1}_{D_0},$$
 for any  bounded domain  $D_0$ in $\mathbb{R}^2$, remains unchanged by the flow. More precisely, the solution associated with the preceding initial datum is given by 
\begin{equation}\label{omg vp}
	\omega(t,\cdot)=b\mathds{1}_{D_t}, \qquad D_t=\Psi(t,D_0),
\end{equation}
for all positive times $t\geq 0$, where $\Psi(t,\cdot) $ is the flow of the log-Lipschitz velocity field $v$,     defined as the unique solution of the integro-differential equation
$$ \Psi(t,x)=x+\int_0^t v\big(\tau,\Psi(\tau,x)\big)d\tau, \quad x\in \mathbb{R}^2.$$ 

In this paper, our focus extends to a specific subclass of patches where the transported domain $D_t$ is exactly a  rotation, by  some angular velocity $\Omega\in \mathbb{R}$, of the initial domain, referred to  as V-states. In other words, we are seeking solutions of the form \eqref{omg vp}, where
\begin{equation}\label{Dt:V-state}
	D_t = D_0 e^{\ii t\Omega},
\end{equation}
for all times  $t\geq 0$. It is important to note that, in general and conversely to the Eulerian case, the area of the transported patch   may not be conserved. However, as a by-product of the specific ansatz \eqref{Dt:V-state}, the area of the transported domain in that case  remains unchanged by the flow.

 Notably, in the case where the depth function is radially symmetric, one can prescribe a huge class of stationary solutions with radial shapes. Among these steady states, it is readily seen that the functions
\begin{equation*}
	\omega(t,\cdot )=b(|\cdot |)\mathds{1}_{\{ x\in \mathbb{R}^2:\; |x| < a \}}, \quad a>0,
\end{equation*}
and
\begin{equation*}
	\omega(t,\cdot )=b(|\cdot |)\mathds{1}_{\{ x\in \mathbb{R}^2:\;  a_2 < |x| < a_1 \}}, \quad 0< a_2 < a_1,
\end{equation*}
are   solutions of \eqref{Lake1} for all $t\geq0$.  

 Therefore, it is  natural to seek  time-periodic solutions in the vicinity of these steady states.  The justification for the existence of such solutions is the central theme of this work, and further details on this will be provided later on.

Note that  time-periodic (and time-quasi-periodic) solutions have previously been extensively studied and established for various fluid models. The initial discovery  in this direction for the Euler equations, specifically for ellipses, dates back to Kirchhoff \cite{K76}. Subsequently, Deem and Zabusky \cite{DZ78} identified cases with higher symmetries through numerical simulations. Burbea \cite{B82} then proposed an analytical approach to rigorously justify their existence, based on local bifurcation theory. Specifically, for any $m$-fold symmetry, $m\geq 1$, there exists only one   branch of V-states  bifurcating from the Ranking vortices at the angular velocity $ \Omega_m = \frac{m-1}{2m}$. 

Similar structures for the doubly-connected case have been explored in  \cite{HHMV16}. In addition to that, several other properties of the V-states have been investigated for   Euler equations, including questions about   the regularity of the boundary \cite{CCG16,HMV13}, the global bifurcation \cite{HMW20} and  the existence of multipole vortex patches \cite{G20,G21,HH21,HW22,HM17}. Further other aspects have also been addressed, such as the implementation of Nash--Moser scheme and KAM techniques to construct time-quasi-periodic solutions; see \cite{BHM23,  HHR23, HR22}.

Similar inquiries have   been discussed about other active-scalar equations such as the generalized surface quasi-geostrophic equation (SQG) \cite{HHH16,JG24, HH15,HHM21,HXX22,R17}, the quasi-geostrophic shallow-water equations  \cite{DHR19,HR21,R21} or the Euler-$\alpha$ model \cite{R22}. A general approach generalizing the previous results in the simply-connected case has been obtained in \cite{HXX23}. More recently, bifurcation theory has been implemented for more complex coupled Hamiltonian systems, see \cite{HZHH23} for  an interconnected layers model arising in geophysical flows  and   \cite{R23}  for the  Vlasov--Poisson equation. At last, we also refer the interested reader to another set of results in a three-dimensional setting \cite{GHR23,GHM22,GHM22-1}. 

\subsection{Main results}
Before we get into the details of the main contribution  of the paper, allow us  to setup the basis of our assumptions on the bottom topography $b$ as well as some notations to be constantly used thereafter. 

\subsubsection*{Overall sufficient conditions on $b$}
The depth function $b$ will be assumed to enjoy certain conditions which are  consistent with the vortex-patch solutions of our interest throughout this paper, though they might not be optimal in some sense. The sufficient conditions on $b$ are now listed below for the sake of clarity and convenience, even though they will not be required as all of them in every single result of the paper. More details of  that will be discussed subsequently. 
\begin{enumerate}[label=\textbullet]
	\item (Regularity) We assume that the function $b$ enjoys the regularity  	\begin{equation}\label{Hb1:eq}\tag{Hb1}
		b\in C^2(\mathbb{R}^2,\mathbb{R}).
	\end{equation}
	\item (Positivity)  We assume that $b$ is positive, i.e.,  \begin{equation}\label{Hb2:eq}\tag{Hb2}
		b(x)>0 ,\quad \text{for all} \quad x\in\mathbb{R}^2. 
	\end{equation}
	\item (Constancy  at infinity) There exists  $R_\infty>0$ and $b_{\infty}>0$ such that 
	\begin{equation}\label{Hb3:eq}\tag{Hb3}
		  b(x)=b_{\infty}, \quad \text{for all} \quad x \in   \mathbb{R}^2\setminus B(0,R_\infty)  .
	\end{equation}
	\item (Radial symmetry) There exists $\mathtt{b}:[0,\infty) \to (0,\infty)$ such that
	\begin{equation}\label{Hb4:eq}\tag{Hb4}
		  b(x)=\mathtt{b}(|x|),\quad \text{for all} \quad x\in\mathbb{R}^2  .
	\end{equation}
	For a mere of simplicity, with an abuse of notation, we will keep using  $b$ instead of $\mathtt{b}$.
\end{enumerate}
At last, we introduce the sequence \begin{equation}\label{def Lambdan}
	\Lambda_{n}(\alpha,\beta) \bydef\int_{0}^{2\pi}K_b\big(\alpha,\beta e^{\ii \eta}\big)\cos(n\eta)d\eta,
\end{equation} 
for any given $\alpha,\beta>0$ and $n\in\mathbb{N}^*,$
where $K_b$ refers to the kernel associated with the elliptic equation \eqref{elliptic:intro}, which will be defined and discussed with details in Section \ref{sec:kernel}.

The first main result of the paper concerns the existence of time-periodic solutions of \eqref{Lake1} bifurcating from a   given disc, which we state next.
\begin{thm}\label{thm:1}   
Let $a>0$ and $b$ be  a function satisfying \eqref{Hb1:eq}, \eqref{Hb2:eq}, \eqref{Hb3:eq} and \eqref{Hb4:eq}. Then, there is a positive constant $M(a,b)>0$  such that, for any $m\in\mathbb{N}^*$ with
	 \begin{equation*}
	 	m> M(a,b),
	 \end{equation*}	
	    there exists a curve of $m$-fold symmetric simply-connected V-states for  \eqref{Lake1} bifurcating from the stationary solution
	    $$\omega = b\mathds{1}_{\{x\in\mathbb{R}^2:\; |x|<  a\}},$$
	     at the angular velocity 
	 $$\Omega_m =\frac{1}{a^2} \int_0^{a} \tau b(\tau) d\tau -\Lambda_{m}(a,a).$$  
	 \end{thm}
\begin{rem}    A precise expression of the constant $M(a,b)$ is given in   Proposition \ref{lemma:eigenfunction}, below. In particular, this constant is identically zero if $b$ is a constant function. Moreover, it holds that 
		\begin{equation*}
			\lim_{\norm{b''}_{L^\infty}  \to 0} M(a,b)=0 
		\end{equation*}  
		and  when $b\equiv 1$, that
		$$M(a,b) =0 \qquad \text{and} \qquad \Omega_{m}  =\frac{m-1}{2m},$$
		for any integer  $m\geq 1$. These values of the angular velocities correspond to the   same values of Burbea's result, previously established  for Euler equations \cite{B82}.
		Accordingly,   Theorem \ref{thm:1} above establishes ``sharp'' existence of $m$-fold symmetric $V$-states for any $m\geq  1$ (resp. large symmetries $m\gg 1$) if $\norm{b''}_{L^\infty}$ is sufficiently small (resp. if $\norm{b''}_{L^\infty}$ is arbitrary large).  
\end{rem}
The second main result of the paper establishes the existence of doubly-connected time-periodic solutions of \eqref{Lake1} bifurcating from an annulus, which  is the content of the next theorem.
\begin{thm}\label{thm:2}
	Let $0<a_2<a_1$ and $b$ be a function satisfying \eqref{Hb1:eq}, \eqref{Hb2:eq}, \eqref{Hb3:eq} and \eqref{Hb4:eq}. Then, there is a positive constant $N(a_1,a_2,b)>0$  such that, for any $m\in\mathbb{N}^*$ with 	
	$$ m> N(a_1,a_2,b),$$
	 there exist two curves of $m$-fold symmetric doubly-connected V-states bifurcating from the stationary solution of \eqref{Lake1}
	  $$\omega = b\mathds{1}_{\{x\in\mathbb{R}^2:\; a_2<|x|<a_1\}},$$   at the angular velocities  \begin{align*}
		\Omega_m^{\pm}& \bydef   \frac{1}{2 a_1^2} \int_{a_2}^{a_1} \tau b(\tau) d\tau+\frac{\Lambda_m(a_2,a_2)-\Lambda_m(a_1,a_1)}{2}\\
		& \quad \pm\frac{1}{2}\sqrt{\left(\frac{1}{  a_1^2} \int_{a_2}^{a_1} \tau b(\tau) d\tau-\Lambda_m(a_1,a_1)-\Lambda_m(a_2,a_2)\right)^2-4\Lambda_m^2(a_1,a_2)}.
	\end{align*}
\end{thm} 
\begin{rem}
By setting $b\equiv 1$ in Theorem \ref{thm:2} we recover the   same values of the angular velocities for the Euler equations established in \cite{HHMV16}. Moreover, in that case,  the  condition  on $m$ can be shown to align with the constraints on the  same symmetry parameter  derived from the results on the Euler equations \cite{HHMV16}.  
\end{rem} 

\subsubsection*{Comments about the assumptions on the depth function $b$} It is apparent through the statements of our theorems above that   the assumptions \eqref{Hb1:eq}--\eqref{Hb4:eq} are sufficient  for the construction of time-periodic solutions of \eqref{Lake1}. However, we do not make any claim   about their necessity and optimality. Though, let us briefly shed   light on the relevance of these conditions to the mechanism  underlying  our methods and proofs of Theorems \ref{thm:1} and \ref{thm:2}. Further discussion, specifically focused on the elliptic problem \eqref{elliptic:intro}, will be provided in Section \ref{Sec:elliptic}.
  
\begin{itemize}
	\item Condition \eqref{Hb1:eq} is closely tied to defining the kernel associated with the elliptic equation \eqref{elliptic:intro}. While it  has  been shown to be  useful and has been employed  in other papers on the lake equation \cite{AL21,DS20,M23}, it is    not strictly  required in many parts in the study of   the elliptic problem \eqref{elliptic:intro}. More precisely, it is particularly related to the decomposition proposed in \cite{DS20} of the kernel $K_b$, see \eqref{EL:P2}. Additionally, it is notably employed in the regularity analysis in Section \ref{section:spectral analysis}. 
	\item Condition \eqref{Hb2:eq} was commonly imposed in various other papers \cite{AL21,DS20,M23}. Our proofs heavily rely on it and our approach possibly breaks-down otherwise. Allowing the depth function to degenerate, as  in \cite{BM06,HLM21}, for example, would necessitate non-trivial changes in all the arguments presented in this paper. Therefore, it is currently unclear whether time-periodic solutions could be constructed if $b$  vanishes on a non-empty set.
	\item Condition \eqref{Hb3:eq} was conjectured in \cite{DS20} as sufficient to extend all the results therein from bounded domains to the whole plane $\mathbb{R}^2$ (see \cite[open problem 2, page 1493]{DS20}). Additionally, our motivation for imposing this condition is justified by the objective of studying rotating solutions, whose construction is primarily local-in-space. In other words, roughly speaking, solutions of the form \eqref{omg vp} are barely    impacted by the behavior of the depth function at infinity.  Nevertheless, condition \eqref{Hb3:eq} can possibly be relaxed. Indeed, we believe that most of the results in Section \ref{Sec:elliptic} about the elliptic problem \eqref{elliptic:intro} can potentially be established under appropriate assumptions on  the decay at infinity of derivatives of the depth function $b$, as is for instance discussed  in \cite{M23}. Even though, we emphasize that,   with condition \eqref{Hb3:eq} in hand, several properties of solutions of \eqref{elliptic:intro} can be proven through  direct computations via constructive proofs, which will be detailed later on in Section \ref{Sec:elliptic}.  
	\item It  will be noticed that the condition \eqref{Hb4:eq} appears to be  necessary in the construction of the V-states. Indeed, this is used to generate Rankine vortices from which we can bifurcate under suitable assumptions.  Thus, for the sake of simplicity, we have decided to discuss the elliptic problem \eqref{elliptic pb intro} under the radial symmetry assumption on the depth function $b$, that is \eqref{Hb4:eq}.  Nevertheless, most of the results in Section \ref{Sec:elliptic} could be extended to cover more general cases of the depth function $b$.   
 \end{itemize}

 \nobreak
 
\subsection{Challenges and roadmap of the proofs}  The proofs of Theorems \ref{thm:1} and \ref{thm:2} will be built upon the scheme laid out by the generalized version of implicit-functions theorem, that is, Crandall-Rabinowitz's Theorem \ref{Crandall-Rabinowitz theorem}. To that end, in the simply connected case, the starting point consists in observing that time-periodic solutions are patches, by definition, which  obey the following contour dynamics equation 
\begin{equation}\label{CD:EQ}
	\partial_{t}z(t,\theta)\cdot\mathbf{n}\big(t,z(t,\theta)\big)+\frac{1}{b\big(z(t,\theta)\big)}\partial_{\theta}\Big(\psi_b\big(t,z(t,\theta)\big)\Big)=0, 
\end{equation}
for all $ t\geq 0$, and $\theta\in \mathbb{T}$, where the unknown function $z(t,\cdot)\subset \partial D_t$ refers to a parametrization of the boundary of the patch at time $t\geq 0$, whereas $ \mathbf{n}\big(t,z(t,\theta )\big)$ denotes the outward normal vector to the boundary at the  point $z(t,\theta)$ and $\psi_b$ is the stream function that solves the elliptic problem \eqref{elliptic:intro}. In addition to that, time-periodic solutions will be characterized by a particular form 
\begin{equation*}
	z(t,\theta) = R(\theta) e^{\ii (\theta+ t \Omega)},
\end{equation*}
for  some   function $R >0$ and some angular velocity $\Omega\in \mathbb{R}$. At last, as is aforementioned, it is natural to look for solutions   close to the stationary state  
\begin{equation*}
	\omega = b\mathds{1}_{\{x\in\mathbb{R}^2:\; |x|<a\}},
\end{equation*} 
for a given radius $a>0$. Therefore, this leads us to the ansatz
\begin{equation*}
	z(t,\theta) = \sqrt{ a^2 + 2 r(\theta)} e^{\ii (\theta+ t \Omega)},
\end{equation*}   
where $r:\mathbb{T} \rightarrow \mathbb{R}$. Afterwards, by incorporating the preceding ansatz in the contour dynamics equation, one sees that the time-variable can be discarded from the equations, whereby we arrive at the stationary problem 
\begin{equation*}
	\mathcal{F}(\Omega,r)=0,
\end{equation*}
for some functional $\mathcal{F}$. The details of that are postponed and will be discussed in Section \ref{sec:controu}. Likewise in the case of doubly-connected V-states near a given annulus, we emphasize that a similar analysis leads to a two-dimensional functional system  
\begin{equation*}
	\mathcal{G}(\Omega,r_1,r_2)\bydef\Big(\mathcal{G}_1(\Omega,r_1,r_2), \mathcal{G}_2(\Omega,r_1,r_2) \Big)=0,
\end{equation*}
for some functionals $\mathcal{G}_1$ and $\mathcal{G}_2$ which will be precised and discussed in Section \ref{sec DC}, as well.

The strategy is then to show that the functionals $\mathcal {F}$ and  $\mathcal {G}$  satisfy all the prerequisites in the statement of Crandall-Rabinowitz's Theorem \ref{Crandall-Rabinowitz theorem}. This is mainly related to the regularity study of these nonlinear functionals and the spectral analysis of their linearized operators around the stationary solutions. In contrast to \cite{B82, CCG16-2, HH15,  HMV13}, the kernels involved in these functionals are not explicit and with a quite rough regularity, making their analysis much more delicate. In particular, our analysis cannot be included in the general work \cite{HXX23}. To give an insight onto the encountered difficulties throughout our proofs, let us now focus on the main aspects on to the reconstruction of the stream functions related to the elliptic problem 
\begin{equation}\label{elliptic pb intro}
		-\div(\tfrac{1}{b}\nabla\psi_b)=f,
	\end{equation}
	with data $b$, $f$ and unknown $\psi_b.$  We prove in Proposition \ref{prop.kernel} a Dekeyzer--Schaftingen-type decomposition \cite{DS20} in the form
\begin{equation}\label{DS dec intro}
	\psi_b(x)=\int_{\mathbb{R}^2}K_b(x,y)f(y)dy,\qquad K_b(x,y)=-\tfrac{1}{2\pi}\log|x-y|\sqrt{b(x)b(y)}+S_b(x,y),
\end{equation}
with 
$S_b$ being symmetric and of regularity $W^{2,p}$, for all $p\in (1,\infty)$. Note that a similar decomposition was obtained recently in \cite{M23}, using different arguments. Such a decomposition reveals to be useful to describe the discrete  spectrum of the linearized operator at the equilibrium state, laid out in Lemma \ref{lemma:linearized-OP} and Proposition  \ref{prop regular2}. More specifically, although  no explicit form of the     eigenvalues associated with   this linearized operator is available, we are able to develop a suitable asymptotic expansion of the spectrum through a contraction argument. This is very crucial and is discussed in Proposition \ref{lemma:eigenfunction}.

As for the regularity of the functionals $\mathcal{F}$ and $\mathcal{G}$,  we rely on  a new decomposition obtained in  Lemma \ref{lemma:expansion1} taking the form, in the particular case $k=1$ from that lemma,   
$$\psi_b(x)=-\tfrac{1}{2\pi}\int_{\mathbb{R}^2}\log|x-y|b(y)f(y)dy+\phi_f(x) ,$$
with $\nabla\phi_f\in W^{2,p}(\mathbb{R}^2)$, for any $p>2$, provided that $f$ is compactly supported and belongs to the Lebesgue space $L^p$. We believe that   the generic statement of Lemma \ref{lemma:expansion1}   could be used to explore more important questions on related topics.

 \bigskip 

In the sequel, we are going to use the standard notations for Lebesgue, Sobolev and H\"older spaces: $L^p(\mathbb{R}^2)$, $H^s(\mathbb{R}^2)$, $\dot{H}^s(\mathbb{R}^2)$, $C^{\alpha}(\Omega)$, $\dot C^{\alpha}(\Omega)$ \dots, where $p\in [1,\infty]$,  $s\in \mathbb{R}$, $\alpha\in (0,1)$ and $\Omega = \mathbb{R}^2$ or $\mathbb{T}^2$.  The preceding spaces are then equipped with their standard (semi-)norms. Note that a particular attention will be paid    to homogeneous Sobolev spaces $\dot H^s(\mathbb{R}^2)$ in Section \ref{sec:Sobolev}.

For a mere of simplicity, we  employ ``$\lesssim $'' instead of ``$\leq C$'' where the dependance on the constant $C>0$ is not crucial. Moreover, we occasionally use the symbol ``$\lesssim_\delta  $'' to emphasize the dependance of the preceding inequality on some parameter or function $\delta$. 

\section{On an elliptic problem with variable coefficients}\label{Sec:elliptic}  
This section is dedicated to the analysis of the elliptic problem 
\begin{equation}\label{EL:P1} 
		\displaystyle-\div\left(\frac{1}{b}\nabla\psi_b\right)=f,  
\end{equation} 
where $f=f(x)$ and $b=b(x)$ are two given real functions defined on the entire  domain $\mathbb{R}^2$.  Further precise assumptions on both $f$ and $b$ will be   explicitly stated, later on. Notice that     \eqref{EL:P1} is relevant to the lake equation \eqref{LE} as the stream function $\psi_b$   is recovered from the vorticity $\omega$ by solving \eqref{EL:P1} with the source term  $f=  \omega$.

Our primary focus here is to discuss some qualitative and quantitative properties of $\psi_b$ as the solution of \eqref{EL:P1} with a general source term $f$. The main novelty in this section concerns the new decomposition of solutions given by Lemma \ref{lemma:expansion1}. Additionally,  many other results are collected here, though  they may not be brand-new. For the sake of completeness, we supply this section  with a self-contained appendix (Appendix \ref{appendix:elliptic}), which provides a collection of some complementary  results and proofs that we present in a different (simple and self-contained) fashion compared to some other references.

\subsection{A priori bounds and decomposition of solutions} To begin with, let us assume that, under reasonable conditions on the source term $f$ to be specified subsequently, the elliptic problem \eqref{EL:P1} admits at least a solution $\psi_b$. This solution can in fact be shown to be unique in appropriate functional spaces. We defer the discussion about the existence and uniqueness of solutions of \eqref{EL:P1} to Appendix \ref{section:A:existence}. Moreover,  the solution of \eqref{EL:P1} enjoys the classical   elliptic regularities as will be stated  in the next proposition. For the sake of completeness,   a full justification of this proposition will be detailed in Appendix \ref{section:A:regularity}.
\begin{prop}\label{prop:Elliptic:00}  
	Let $b$ be a $C^1$ function satisfying assumptions \eqref{Hb2:eq}, \eqref{Hb3:eq}  and \eqref{Hb4:eq}. Further consider a source term $f$ with the properties that  
	\begin{equation*}
		f\in L^1\big((1+|x|)dx;\mathbb{R}^2\big)\cap L^p(\mathbb{R}^2),
	\end{equation*}
	for some  $p\in [2,\infty)$. Then, any solution of \eqref{EL:P1}   enjoys the bound
	\begin{equation}\label{Bound:ELLIP}
		\norm{\nabla ^2 \psi_b}_{ L^m(\mathbb{R}^2)} \lesssim_b \norm f_{L^m(\mathbb{R}^2)}, 
	\end{equation} 
	for all $m\in (1,p]$. 
	\end{prop}
	\begin{rem}
		By Sobolev embedding, it is readily seen that the preceding proposition entails the control
\begin{equation*}
	\nabla \psi_b\in  L^n(\mathbb{R}^2),
\end{equation*}
for all $n\in (2,\infty)$. The endpoint case $n=2$ is not attained in general. However, the case $2= \infty$ is reachable in the case $p>2$. This is crucial and, once again, for a mere sake of completeness,  it will be further discussed in  Appendix \ref{section:A:regularity}.
	\end{rem}

\begin{rem}\label{RMK:k:der}
	Note that the results of the preceding proposition can be extended  to cover similar bounds on derivatives of the solution as soon as the  source term enjoys extra Sobolev regularities. More precisely, it is straightforward to show that, if the source term $f$ is compactly supported\footnote{Here, the assumption on the compactness of the support of the source term is made for the sake of simplicity. This can be relaxed under adequate assumption on the growth of the function $b$.} and belongs to $W^{k,p}(\mathbb{R}^2)$, for some $k\in \mathbb{N}^*$, then it holds that  
	\begin{equation*}
		\norm{\nabla ^2 \psi_b}_{ \dot W^{j,m} (\mathbb{R}^2)} \lesssim \norm f_{ W^{j,m}(\mathbb{R}^2)}, 
	\end{equation*} 
	for all $m\in (1,p]$ and $j\in \llbracket1,k\rrbracket$. Indeed, to see that for $j=1$, we just need to expand \eqref{EL:P1} after applying a directional derivative $\partial_i$, for $i\in \{1,2\}$, to write that 
	\begin{equation*} 
		\displaystyle-\div\left(\frac{1}{b(x)}\nabla (\partial_i \psi_b)(x)\right)= \partial_i f(x) - \Delta \psi_b \frac{\partial_i b}{b^2}  - \nabla \left( \frac{\partial_i b}{b^2}\right) \cdot \nabla \psi_b \bydef F.
\end{equation*}
This means that $\partial_i \psi_b$ satisfies the elliptic equation \eqref{EL:P1}   with the new source term  $F$.
The justification of our claim above for $j=1$ amount then to showing that the new source term $F$    belongs to $ L^p(\mathbb{R}^2)$. This is done by virtue of a direct computation  using the results of Proposition \ref{prop:Elliptic:00} and the assumptions on $b$ and $f$.  Moreover, we emphasize that this procedure can  straightforwardly be iterated to show the desired control for any $j \in  \llbracket1,k\rrbracket$.
\end{rem}

In the next simple, but essential, lemma, we establish a new decomposition of  the solution of \eqref{EL:P1}. In particular, this new decomposition allows us to regard \eqref{EL:P1} as a regular perturbation of the classical Poisson's problem with a suitable source term where the size of its $L^p$-norm  is comparable to the same norm of the original source in \eqref{EL:P1}.

A crucial observation here, which will come in handy  later on, is that the remainder in the new expansion of $\psi_b$ below can be set to be regular as much as we need, provided that $b$ is sufficiently regular. We believe that the features of the new decomposition given in the next lemma go beyond the scope of the present paper as it can serve in other different situations.

\begin{lem}\label{lemma:expansion1}   Let  $b$ be a function that satisfies assumptions \eqref{Hb2:eq}, \eqref{Hb3:eq} and \eqref{Hb4:eq}  and belongs to $C^{k+1}(\mathbb{R}^2)$, for some integer $k\geq 1$. Further consider a compactly supported function $f$ in $L^p(\mathbb{R}^2)$, for some $p\in[2,\infty)$.   
 Then, there are two functions  $\varphi_k\in L^p(\mathbb{R}^2)$ and   $\phi_k \in \dot{W}^{k+2,m} (\mathbb{R}^2)$, for all $m\in (1,p]$, such that any solution of \eqref{EL:P1} with source term  $f$  can be decomposed as 
$$ \psi_b =-\frac{1}{2\pi}\int_{\mathbb{R}^2}\log|\cdot -y|\varphi_k(y)dy+\phi_k .$$
Furthermore, $\varphi_k$ is compactly supported and satisfies the bound 
\begin{equation*}
	\norm {\varphi_k}_{L^q  (\mathbb{R}^2)} \lesssim_b  \norm {f}_{L^q  (\mathbb{R}^2)},
\end{equation*}
for all $q\in [1,p]$, and $\phi_k$ enjoys the control
\begin{equation*}
\norm {\phi_k}_{\dot{W}^{i+2,m}  (\mathbb{R}^2)} \lesssim_b \norm {f}_{L^m  (\mathbb{R}^2)}, 
\end{equation*}
for all $m\in (1,p]$ and any $i\in \llbracket0,k\rrbracket$, as well as 
\begin{equation*}
	 \norm {\nabla \phi_k}_{L^n  (\mathbb{R}^2)} \lesssim_b \norm {f}_{L^n  (\mathbb{R}^2)},
\end{equation*}
for every $n\in (2,p]$.
\end{lem}

\begin{rem}
	We emphasize that the assumption on the support of the source term can be traded with  weaker assumptions on its growth in Lebesgue spaces along with suitable assumptions on the growth of the function $b$ in the absence of the stronger condition \eqref{Hb3:eq}. For the sake of simplicity, we chose to restrict our selves to the case where the condition \eqref{Hb3:eq} and the assumption one the compactness of the source  are in action, which are    reasonable hypothesis in the context of vortex-patch solutions of \eqref{Lake1}.
\end{rem}

\begin{rem}
	Later on, we will only need to apply the preceding lemma with  $k=1$. This in particular allows us  to write that
	\begin{equation*}
		\psi_b(x)=-\frac{1}{2\pi}\int_{\mathbb{R}^2}\log|x-y|b(y)f(y)dy+\phi_f(x), 
	\end{equation*}
	 for all $x\in \mathbb{R}^2$, where the remainder term solves the equation
	 \begin{equation}\label{Remainder_remark}
	 	-\div\left(\frac{1}{b}\nabla\phi_f\right)=\frac{\nabla b}{b^2}\cdot\nabla (\Delta)^{-1}(bf)
	 \end{equation}
	 and enjoys the bound 
\begin{equation}\label{rem:reg:remainder}
	\nabla \phi _f \in W_{}^{2,p}(\mathbb{R}^2).
\end{equation}
Notably, this remainder enjoys a better   regularity compared to the decomposition of the solution proposed  in \cite{DS20}. 
\end{rem}

\begin{proof} Observe first that the elliptic equation   \eqref{EL:P1} can be recast as
$$-\Delta\psi_b=bf-\frac{\nabla b}{b}\cdot\nabla\psi_b.$$
This  inspires introducing  the following expansion of $\psi_b$ 
$$\psi_b=\sum_{j=1}^k\widetilde{\phi}_j+\phi_k,$$
where $\widetilde{\phi}_j$ is defined by the   induction scheme  
$$-\Delta\widetilde{\phi}_1\bydef bf
\qquad\text{and}\qquad 
-\Delta\widetilde{\phi}_j\bydef-\frac{\nabla b}{b}\cdot\nabla\widetilde{\phi}_{j-1},
 \quad \text{for all} \quad 
   j\in \llbracket2,k\rrbracket,$$ 
whereas the remainder term $\phi_k $ in the expansion above is   governed by the equation 
\begin{equation}\label{remainder:eq}
	-\div\left(\frac{1}{b}\nabla\phi_k\right)=-\frac{\nabla b}{b^2}\cdot\nabla\widetilde{\phi}_{k}.
\end{equation}
Notice that the  assumptions on the supports of $f$ and $b$ imply   that the right-hand sides in   the previous elliptic equations are compactly supported. Therefore, classical elliptic regularity for Poisson's problem (or simply consider Proposition \ref{prop:Elliptic:00} with $b\equiv 1$)  ensure the following bounds 
\begin{equation}\label{Lk:ES:00}
	\norm {\nabla ^2 \widetilde \phi_j}_{L^m(\mathbb{R}^2)}   \lesssim_b \norm { f}_{L^m(\mathbb{R}^2)},
\end{equation}
and 
\begin{equation}\label{Lk:ES:0011}
	 \norm {\nabla   \widetilde \phi_j}_{L^n(\mathbb{R}^2)} \lesssim_b \norm { f}_{L^n(\mathbb{R}^2)},
\end{equation}
for all $m\in (1,p]$,  $n\in (2,p]$ and $j \in \llbracket 1,k\rrbracket$. Likewise for the remainder term, Proposition \ref{prop:Elliptic:00} yields the same bounds
  \begin{equation*}
	\norm {\nabla ^2   \phi_k}_{L^m(\mathbb{R}^2)}  \lesssim_b \norm { f}_{L^m(\mathbb{R}^2)},
\end{equation*}
and 
\begin{equation*}
	  \norm {\nabla   \phi_k}_{L^n(\mathbb{R}^2)} \lesssim_b \norm { f}_{L^n(\mathbb{R}^2)}.
\end{equation*} 
 We are now ready to justify the claims of the lemma. To that end, we introduce the functions
$$\varrho_1\bydef bf
\qquad\text{and}\qquad 
 \varrho_j\bydef-\frac{\nabla b}{b}\cdot\nabla\widetilde{\phi}_{j-1}, \quad \text{for all} \quad  j\in\llbracket2,k\rrbracket.$$
Accordingly, we define $\varphi_k$ as
\begin{equation*}
	\varphi_k\bydef\sum_{j=1}^k\varrho_j,
\end{equation*}
whereby obtaining that 
\begin{equation*}
	\sum_{j=1}^k\widetilde{\phi}_j=(-\Delta)^{-1}\varphi_k=-\frac{1}{2\pi}\int_{\mathbb{R}^2}\log|\cdot-y|\varphi_k(y)dy.
\end{equation*}
Moreover, due to the bounds established in the first part of the proof together with the support assumption on $\nabla b$ deduced from \eqref{Hb3:eq},  it is readily seen that 
\begin{equation*}
	\norm{\varphi_k}_{ L^q(\mathbb{R}^2)} \lesssim_b \norm {f}_{L^q(\mathbb{R}^2)},
\end{equation*}
for all $q\in [1,p]$. Let us now be more precise about the regularity of the summands $ \widetilde {\phi}_j$, which will allow us to show that the remainder $\phi_k$ enjoys a better regularity. To see that, we proceed by induction to show that 
\begin{equation*}
	 \widetilde{\phi}_j\in \dot{W}^{j+1,n}(\mathbb{R}^2),
\end{equation*}
for any $j\in \llbracket 1,k\rrbracket$ and $n\in (1,p]$, with 
\begin{equation*}
	\norm {\nabla ^{j+1} \widetilde{\phi}_j}_{L^n(\mathbb{R}^2)} \lesssim_b \norm {f}_{L^n(\mathbb{R}^2)}.
\end{equation*}
Indeed, the previous claim is already proven for $j=1$, whereas its proof for $j\in \llbracket 2,k\rrbracket$ only relies on the observation that any solution of Poisson's equation with a source term belonging to $\dot{W}^{n,q}(\mathbb{R}^2)$, for some $n\in \mathbb{N}$ and $q\in (1,\infty)$, belongs itself to $\dot{W}^{2+n,q}(\mathbb{R}^2)$. Thus, repeating this procedure and utilizing the $C^{1+k}(\mathbb{R}^2)$ regularity of $b$, altogether with assumption \eqref{Hb3:eq}, leads to the desired regularity for $ \widetilde{\phi}_j$.  
At the end, we particularly have shown that 
 \begin{equation*}
	\nabla  \widetilde{\phi}_k\in \dot{W}^{k ,p}(\mathbb{R}^2),
\end{equation*} 
 with 
\begin{equation*}
	\norm {\nabla   \widetilde{\phi}_k}_{\dot W^{k,p} (\mathbb{R}^2)} \lesssim_b \norm {f}_{L^p(\mathbb{R}^2)}.
\end{equation*}
Therefore, it follows that the source term in \eqref{remainder:eq} is compactly supported (due to assumption \eqref{Hb3:eq}) and belongs to $ W^{k ,p}(\mathbb{R}^2)$ (by virtue of the preceding bound  altogether with \eqref{Lk:ES:0011}).  Thus, in view of Remark \ref{RMK:k:der}, we deduce that  
\begin{equation*}
	\norm {\phi_k}_{\dot{W}^{k+2,m}  (\mathbb{R}^2)}   \lesssim_b \norm {f}_{L^m  (\mathbb{R}^2)},
\end{equation*}
for all  $m\in (1,p]$. At last, combining the previous bound with \eqref{Lk:ES:00}, we deduce by a direct interpolation inequality, that
\begin{equation*}
	  \norm {\phi_k}_{\dot{W}^{i+2,m}  (\mathbb{R}^2)} \lesssim_b \norm {f}_{L^m  (\mathbb{R}^2)}, 
\end{equation*}
for all $m\in (1,p]$ and $i\in \llbracket0,k\rrbracket$, thereby completing the proof of the lemma. 
\end{proof}

\subsection{Kernels and integral representation of solutions}\label{sec:kernel}
Here, we show the existence of a kernel associated with the elliptic operator $$-\div(b^{-1} \nabla \cdot).$$  Several properties of this kernel will also be established in the next proposition which will serve   in the upcoming sections of this article.   
    The solution of the following elliptic equation
 \begin{equation}\label{EL:P2} 
	\displaystyle-\div_x\left(\frac{1}{b(x)}\nabla_x S_b(x,y)\right)=\frac{1}{2\pi}\log|x-y|\sqrt{b(y)}\Delta_x\left(\frac{1}{\sqrt{b(x)}}\right)  ,
\end{equation}
for all $ x\in\mathbb{R}^2$ and any given $y\in \mathbb{R}^2$, is also relevant in a sense that will be made precise, now.

\begin{prop}\label{prop.kernel} 
Let $b$ be a function satisfying assumptions \eqref{Hb1:eq}, \eqref{Hb2:eq}, \eqref{Hb3:eq} and \eqref{Hb4:eq}. Then, for any given $y\in \mathbb{R}^2$, there is  a unique solution\footnote{Here, the uniqueness of the solution is understood in the sense of Theorem \ref{thm:EL:1}.} $S_b(\cdot ,y)$ of \eqref{EL:P2}  enjoying the bounds 
\begin{equation*}
	 \nabla ^2_x S_b(x ,y) \in  L^\infty_{\loc }\big((dy;\mathbb{R}^2);  L^m (dx;\mathbb{R}^2)\big),
\end{equation*}
for all $m\in (1,p]$, and 
\begin{equation*}
	 \nabla_x  S_b(x ,y) \in  L^\infty_{\loc }\big((dy;\mathbb{R}^2); L^n (dx;\mathbb{R}^2)\big),
\end{equation*}
for all $n\in (2,p]$, as well as
 \begin{equation*}
	  S_b  \in  L^\infty_{\loc } ( \mathbb{R}^2\times \mathbb{R}^2).
\end{equation*}

Moreover,  for any compactly supported function $f$  in $L^p(\mathbb{R}^2)$, for some $p\in [2,\infty)$, the unique solution of \eqref{EL:P1} can be represented as 
 \begin{equation}\label{integral rep psib}
		\psi _b(x)=\int_{\mathbb{R}^2}K_b(x,y)f(y)dy,
	\end{equation}
 for any $x\in \mathbb{R}^2$, where we set 
 \begin{equation}\label{symmetry-kernel:split} 
		 K_b(x,y)\bydef-\frac{1}{2\pi}\log|x-y|\sqrt{b(x)b(y)}+S_b(x,y) ,
		 	\end{equation}
	for any $x,y\in \mathbb{R}^2$ with $x\neq y$.
	
Furthermore,  $K_b(x,y)$ is symmetric for any $x,y\in \mathbb{R}^2$ with $x\neq y$, i.e.
\begin{equation}\label{symmetry-kernel-0}
	 K_b(x,y)=K_b(y,x),
\end{equation}
and it is stable by rotations and involution transformation in the sense that 
\begin{equation}\label{symmetry-kernel}
		K_b(\mathcal{R}_\theta\cdot x,\mathcal{R}_\theta\cdot y)=K_b(x,y),\qquad K_b(\mathcal{S}x, \mathcal{S} y)=K_b(x,y),
	\end{equation}
	for any $\theta \in [0,2\pi]$,  where $\mathcal{R}_\theta$ and $\mathcal{S}$ denote the rotation matrix with angle $\theta$ and the involution operator, respectively, i.e.
	$$\mathcal{R}_\theta\bydef\begin{pmatrix}
		\cos(\theta) & -\sin(\theta)\\
		\sin(\theta) & \cos(\theta)
	\end{pmatrix},$$
	and 
		$$\mathcal{S}x\bydef|x|e^{-i\eta},$$
		for any given $x=|x|e^{\ii \eta} \in \mathbb{R}^2$, for some $\eta\in[0,2\pi]$ 
	\end{prop} 
	
\begin{rem}\label{rem:ell:bounds}
    The solution  of the elliptic problem  \eqref{EL:P1} defined through the integral representation \eqref{integral rep psib} satisfies all the bounds on the solution of \eqref{EL:P1}, previously established in Proposition \ref{prop:Elliptic:00}. This will come in handy, later on.
\end{rem}

	For a mere of a clear presentation, we defer the proof of this proposition to Appendix \ref{section:A:kernel}. 
	Let us emphasize in passing that the definition of $K_b$ in terms of the solution $S_b$ of \eqref{EL:P2} was established at a first time in bounded domains in \cite{DS20}. Very recently, similar  results in the case of the whole plane were claimed in \cite{M23}. On the other hand, for the sake of completeness, the full justification of Proposition \ref{prop regular2} will be discussed with details in Appendix \ref{section:A:kernel}, using different arguments compared to \cite{M23}. Our proof tends to be more constructive under suitable assumptions on the function $b$.

We conclude this section by stating and proving a computational lemma which will be useful the next sections.

\begin{lem}\label{velocity_expl_comp}
	Let $\alpha, \beta \in (0,\infty)$ and $K_b$ be the kernel associated with the elliptic problem \eqref{EL:P1} as introduced in the preceding proposition. Then, it holds that 
	\begin{equation*}
		\frac{1}{\alpha}\int_0^{2\pi}\int_{0}^{\beta}\partial_{x_1}K_b\big(\alpha,\rho e^{\ii \eta}\big)b(\rho)\rho d\rho d\eta = - \frac{b(\alpha)}{\alpha^2} \int_0^{\min\{\alpha,\beta\}} \tau b(\tau) d\tau.
	\end{equation*}
\end{lem}

\begin{proof}
	The proof follows from a direct computation, employing the explicit expression of solutions of the elliptic equation \eqref{EL:P1}. To see that, we first write 
	\begin{equation*}
		\begin{aligned}
			\int_0^{2\pi}\int_{0}^{\beta}\partial_{x_1}K_b\big(\alpha,\rho e^{\ii \eta}\big)b(\rho)\rho d\rho d\eta 
		&= \left( \int_0^{2\pi}\int_{0}^{\beta}\partial_{x_1}K_b\big(x,\rho e^{\ii \eta}\big)b(\rho)\rho d\rho d\eta\right)\Big|_{x=(\alpha,0)}
		\\
		&= \left(\partial_{x_1} \int_0^{2\pi}\int_{0}^{\beta}K_b\big(x,\rho e^{\ii \eta}\big)b(\rho)\rho d\rho d\eta\right)\Big|_{x=(\alpha,0)}
		\\
		&= \left(\partial_{x_1} \int_{\mathbb{R}^2}K_b(x,y)b(|y|) \mathds{1}_{B_\beta} dy\right)\Big|_{x=(\alpha,0)},
		\end{aligned}
	\end{equation*}
	where $B_\beta$ denotes the ball  centered at the origin of radius $\beta$. Therefore, by virtue of Proposition \ref{prop.kernel}, we see that the function 
	\begin{equation*}
		x\mapsto\psi (x)\bydef \int_{\mathbb{R}^2}K_b(x,y)b(|y|) \mathds{1}_{B_\beta} dy
	\end{equation*}
	solves the equation 
	\begin{equation*}
		- \div\left( \frac{1}{b} \nabla \psi \right) = b  \mathds{1}_{B_\beta} ,
	\end{equation*}
	whence, since $b\mathds{1}_{B_\beta} $ is a radially symmetric function, we obtain the explicit formula  
	\begin{equation*}
		\psi(x)= - \int_0^{|x|} \frac{b(s)}{s} \int_0^{\min\{s,\beta\}} \tau b(\tau) d\tau ds + C, 
	\end{equation*}
	for some constant $C\in \mathbb{R}$ and all $x\in \mathbb{R}^2$. Deriving the preceding expression  of $\psi$ in the direction $x_1$, then, evaluating  the resulting formula at the point $x=(\alpha,0)$ and incorporating  it into the computations at the beginning of the proof yields the desired identity and concludes the proof of the lemma.
\end{proof}

\section{Simply-connected time-periodic solutions}\label{Sec:simply-connected}

This section is devoted to the proof of  Theorem \ref{thm:1}. We shall split it into several subsection for better readability. Some arguments from this section will also be useful in the doubly-connected case, which will be the subject of Section \ref{sec DC}, later on.

\subsubsection*{Notations} Before we get into the detailed analysis of the contour dynamics equation, allow us to set up a few notations that will be employed throughout the sequel of this paper.    

In the sequel, we will also   agree the identification $\mathbb{C}\approx\mathbb{R}^2$ and we recall that the scalar product of two complex numbers $z_1=x_1+\ii y_1$ and $z_2=x_2+\ii y_2$ is given by
\begin{equation*}
	z_1\cdot z_2=\text{Re}(z_1\overline{z_2})=x_1x_2+y_1y_2.  
\end{equation*}
We will often use the notation 
\begin{equation*}
	\int_{\mathbb{T}} = \frac{1}{2\pi} \int_0^{2\pi}
\end{equation*}
for a mere simplification of our proofs, where $\mathbb{T}\bydef \mathbb{R}/ 2\pi\mathbb{Z}$.

At last, for consistency of notation  in Section \ref{sec DC}, we now introduce the function 
\begin{equation}\label{def Q}
	Q(\alpha,\beta) \bydef \frac{1}{\alpha^2} \int_{\min\{\alpha,\beta\}}^{\max\{\alpha,\beta\}} \tau b(\tau) d\tau =  \int_{\min\{1,\frac{\beta}{\alpha}\}}^{\max\{1,\frac{\beta}{\alpha}\}} \tau b(\alpha \tau) d\tau ,
\end{equation} 
for any given $\alpha>0$ and $\beta\geq 0,$ which will serve in this section and the next one, too. 

\subsection{Contour dynamics}\label{sec:controu}
Here, we establish an equivalent reformulation of \eqref{Lake1} in the case of periodic uniformly rotating solutions. To that end, we begin with considering a parametrization of the boundary of the initial domain $D_0$ which is close to a disc of radius $a>0$, namely
\begin{equation}\label{R:def}
 \begin{aligned}
z:\mathbb{T} &\mapsto \partial D_0\\
 \theta &\mapsto R(\theta)e^{\ii \theta}\bydef\sqrt{a^2+2r(\theta)}e^{\ii \theta}.
\end{aligned}
\end{equation}

 \begin{lem}[Contour equation]\label{lemma:contour} Let $\Omega\in\mathbb{R}.$ The radial deformation defined through \eqref{R:def} gives rise to a uniformly rotating, at angular velocity $\Omega$, vortex patch solution to the lake equation \eqref{Lake1} if and only if it satisfies the nonlinear contour equation
 $$\mathcal{F}(\Omega,r)=0,$$ 
 where, for any $\theta\in\mathbb{T}$, we denote
 \begin{equation}\label{contour_dynamics}
 \mathcal{F}(\Omega,r)(\theta)\bydef\Omega r'(\theta)+\frac{1}{b\big(R(\theta)\big)}\partial_\theta\left(\int_{0}^{2\pi}\int_{0}^{R(\eta)}K_b\big(R(\theta)e^{\ii \theta},\rho e^{\ii \eta}\big)b(\rho)\rho d\rho d\eta\right),
 \end{equation}
 and $K_b$ is the kernel associated with the elliptic equation \eqref{EL:P1}.
 
 Moreover, the trivial deformation, corresponding to $r \equiv 0$, is a solution to the preceding contour equation for all $\Omega \in \mathbb{R}$.
 \end{lem} 
 \begin{proof} 
We first proceed to obtain the contour dynamics equation for a general parametrization of the patch; then we particularly look after the case of a rotating domain.

 Let us consider, for any time $t>0,$ a parametrization $z(t,\cdot):\mathbb{T}\rightarrow\partial D_t$ of the boundary. Then, according to \cite[page 174]{HMV13}, we have, for any $\theta \in \mathbb{T}$ and $t\geq 0$, that
$$ \partial_t z(t,\theta)\cdot\textbf{n}\big(t,z(t,\theta)\big)=v\big(t,z(t,\theta)\big)\cdot\textbf{n}\big(t,z(t,\theta)\big),$$
where we denoted $\textbf{n}\big(t,z(t,\theta)\big)$ the inward normal vector to the boundary $\partial D_t$ of the patch at the point $z(t,\theta)$, i.e., we agree that   
$$\textbf{n}\big(t,z(t,\theta)\big)\bydef \ii\partial_\theta z(t,\theta).$$
Now, recalling  that $\psi_b$ is the stream function corresponding to the divergence-free vector field $-bv$, defined through \eqref{def psib}, and using   the   notation 
$$\nabla=2\partial_{\overline{z}}, \quad \text{and} \quad \nabla^{\perp}=2\ii\partial_{\overline{z}},$$
 we then compute that 
\begin{align*}
	\partial_{\theta}\Big(\psi_b\big(t,z(t,\theta)\big)\Big)&=\nabla\psi_b\big(t,z(t,\theta)\big)\cdot\partial_{\theta}z(t,\theta)\\
	&=\textnormal{Re}\left(2\partial_{\overline{z}}\psi_{b}\big(t,z(t,\theta)\big)\overline{\partial_{\theta}z(t,\theta)}\right)\\
	&=\textnormal{Im}\left(2\ii\partial_{\overline{z}}\psi_{b}\big(t,z(t,\theta)\big)\overline{\partial_{\theta}z(t,\theta)}\right)\\
	&=-b\big(z(t,\theta)\big)\textnormal{Im}\left(v\big(t,z(t,\theta)\big)\overline{\partial_{\theta}z(t,\theta)}\right)\\
	&=-b\big(z(t,\theta)\big)v\big(t,z(t,\theta)\big)\cdot\textbf{n}\big(t,z(t,\theta)\big).
\end{align*}
Accordingly, we obtain the general equation of motion  
\begin{equation}\label{EQ:z}
	\partial_t z(t,\theta)\cdot\textbf{n}\big(t,z(t,\theta)\big)+\frac{1}{b\big(z(t,\theta)\big)}\partial_\theta\Big(\psi_b\big(t,z(t,\theta)\big)\Big)=0.
\end{equation}
Now, we particularize by looking for periodic solutions where $D_t$ is a rotation of the initial domain $D_0$, namely 
\begin{equation}\label{rot D0}
	D_t=e^{\ii \Omega t}D_0.
\end{equation}
This leads to consider the following parametrization for the boundary $\partial D_t$
$$z(t,\theta)\bydef e^{\ii \Omega t}z(\theta)=R(\theta)e^{\ii (\Omega t+\theta)},$$
where $z$ is the parametrization of $\partial D_0$ introduced in \eqref{R:def}. On the one hand,   it is readily seen that  
$$\partial_t z(t,\theta)=\ii\Omega e^{\ii \Omega t}R(\theta)e^{\ii \theta},$$
and 
$$\partial_\theta z(t,\theta)=e^{\ii \Omega t}\left(\ii R(\theta)+\tfrac{r'(\theta)}{R(\theta)}\right) e^{\ii \theta},$$
whereby, it follows that
\begin{align}
	\partial_t z(t,\theta)\cdot\textbf{n}\big(t,z(t,\theta)\big)&=\text{Im}\left(\partial_t z(t,\theta)\overline{\partial_\theta z(t,\theta)}\right)\nonumber\\
	&=\Omega r'(\theta).\label{left hs}
\end{align}
On the other hand, the radial symmetry property \eqref{Hb4:eq} of the bottom topography $b$ implies that 
\begin{equation}\label{brad}
	b\big(z(t,\theta)\big)=b\big(R(\theta)\big).
\end{equation}
In addition, in view of the representation \eqref{integral rep psib}  , we write that 
$$\psi_b\big(t,z(t,\theta)\big)=\int_{D_t}K_b\big(z(t,\theta),y\big)b(y)dy.$$
Therefore, we deduce by employing  \eqref{symmetry-kernel}, \eqref{rot D0}, making the change of variables $y\rightarrow y e^{\ii t\Omega}$ and utilizing \eqref{Hb4:eq} that
  \begin{align} 
   \psi_b\big(t,z(t,\theta)\big)&=\int_{D_0}K_b\big(R(\theta)e^{\ii (\Omega t+\theta)},ye^{\ii \Omega t}\big)b\big(ye^{\ii \Omega t}\big)dy\nonumber\\
   &=\int_{D_0}K_b\big(R(\theta)e^{\ii \theta},y\big)b(y)dy.\label{psi:equa}
  \end{align}
Hence, inserting \eqref{left hs}, \eqref{brad} and \eqref{psi:equa} into \eqref{EQ:z} infers that 
$$\Omega r'(\theta)+\frac{1}{b\big(R(\theta)\big)} \partial_\theta\left(\int_{D_0} K_b\big(z(\theta),y\big)b(y)dy\right)=0,$$
which implies in turn, via polar change of variables, that
$$\Omega r'(\theta)+\frac{1}{b\big(R(\theta)\big)}\partial_\theta\left(\int_{0}^{2\pi}\int_{0}^{R(\eta)}K_b\big(z(\theta),\rho e^{\ii \eta}\big)b(\rho)\rho d\rho d\eta\right)=0.$$
 This derives the contour dynamics as it is claimed in Lemma \ref{lemma:contour}, above.
 
 In addition to that,  by employing the identity \eqref{symmetry-kernel}, followed by a change of variables, we see, for all $\theta\in [0,2\pi]$, that 
 $$ \int_{0}^{2\pi}\int_{0}^{a}K_b\big(ae^{\ii \theta},\rho e^{\ii \eta}\big)b(\rho)\rho d\rho d\eta=\int_{0}^{2\pi}\int_{0}^{a}K_b\big(a,\rho e^{\ii \eta}\big)b(\rho)\rho d\rho d\eta,$$
 thereby yielding that 
\begin{equation}\label{diff0}
	\partial_\theta\left(\int_{0}^{2\pi}\int_{0}^{a}K_b\big(ae^{\ii \theta},\rho e^{\ii \eta}\big)b(\rho)\rho d\rho d\eta\right)=0.
\end{equation}
Subsequently, it follows, for all $\Omega\in\mathbb{R}$, that $(\Omega,0)$ is a zero of the functional $\mathcal{F}$. The proof of the lemma is now completed.
\end{proof}  

Observe, from the proof above, that the functional $\mathcal{F}(\Omega,r)$ can be recast, for any $\theta\in \mathbb{T}$, as 
\begin{equation}\label{contour:equation:2}
	\mathcal{F}(\Omega,r) (\theta) = \Omega r'(\theta) + \frac{1}{b\big(R(\theta)\big)} \partial_\theta\left(\psi_b\big(R(\theta) e^{\ii \theta}\big)\right),
\end{equation} 
where $\psi_b\bydef\psi_b(0,\cdot)$ is the unique solution of the elliptic problem \eqref{EL:P1} with source term $f= b\mathds{1}_{D_0}$. The preceding  representation  will come in handy in the regularity study of the contour equation, later on.

We conclude this section by a simple lemma where we establish important symmetry properties of the functional $\mathcal{F}$.  
\begin{lem}[Symmetry properties of the functional $\mathcal{F}$]\label{lemma:symmetry-lin}
Let $\Omega\in \mathbb{R}$,   $r\in C^{1+\alpha}(\mathbb{T}), $ for some $\alpha\in (0,1)$, and  $\mathcal{F}(\Omega,r)$ be given by \eqref{contour_dynamics}. If moreover $r$ is an even function, i.e., 
$$ 
r(-\theta) = r(\theta), 
\quad 
\text{for all} 
\quad 
\theta \in [0,2\pi],
$$
then, $\mathcal{F}(\Omega,r)$ is an odd function, i.e.
$$ 
\mathcal{F}(\Omega,r)(-\theta) = -\mathcal{F}(\Omega,r)(\theta), 
\quad 
\text{for all} 
\quad 
\theta \in [0,2\pi].
$$
Furthermore, if $r$ is a $m-$fold symmetric function, for some $m\in \mathbb{N}^*$, i.e.
 $$ 
r(\theta + \tfrac{2\pi}{m}) = r(\theta), 
\quad 
\text{for all} 
\quad 
\theta \in [0,2\pi],
$$
then, $\mathcal{F}(\Omega,r)$ is also a $m$-fold symmetric function.
\end{lem}
\begin{proof}
It is readily seen that, if $r$ is an even function, then $r'$ is odd. Likewise for the $m-$fold symmetry, i.e., if $r$ is a $m-$fold symmetric function, then so is $r'$. Thus, regarding the expression of $\mathcal{F}$ in \eqref{contour_dynamics}, we can restrict our focus on showing that the function
$$\theta\mapsto\int_{0}^{2\pi}\int_{0}^{R(\eta)}K_b\big(R(\theta)e^{\ii \theta},\rho e^{\ii \eta}\big)b(\rho)\rho d\rho d\eta$$
is even and $m$-fold symmetric, as soon as $r$ is an even $m$-fold symmetric function. The parity part follows by applying the second symmetry property of the kernel $K_b$ given by \eqref{symmetry-kernel}, followed by the change of variables $\eta\to-\eta$. As for the $m$-fold symmetry property, it is shown by employing the first symmetry property of the kernel $K_b$ given by \eqref{symmetry-kernel}, followed by the change of variables $\eta\to\eta+\frac{2\pi}{m}$. The proof of Lemma \ref{lemma:symmetry-lin} is now completed.
\end{proof}


 \subsection{Spectral analysis} \label{section:linearized:operator}
In this section,  we   study the linearized operator associated with the functional $\mathcal{F}$ at the equilibrium $r=0$ and  we show that it acts as a Fourier multiplier. We emphasize, however, that all the computations in this section will be completed with a detailed regularity analysis in Section \ref{section:spectral analysis}, below.

To begin, let us now compute the differential of the functional $\mathcal{F}$ at any point $r$. 
\begin{lem}\label{lemma:regularity}
Let $\mathcal{F}$ be given by  \eqref{contour_dynamics}. Let $m\in \mathbb{N}^*$, $\alpha\in (0,1)$ and $\varepsilon>0$ be a given small parameter in such a way that the function $\mathcal{F}$ is differentiable according to Proposition \ref{proposition regularity of the functional}. Then, the differential of $\mathcal{F}$ at $r\in \mathcal{B}_{m,\varepsilon}^\alpha$ in the direction $h\in X_{m}^{\alpha}$ is given by 
$$d_r\mathcal{F}(\Omega,r)[h]=\Omega h'+\frac{1}{b\circ R}\partial_\theta\Big(Q(r)h+P(r)[h]\Big)+L(r)h,$$
where we set
\begin{align*}
P(r)[h](\theta)&\bydef\int_0^{2\pi}h(\eta)K_b\big(R(\theta)e^{\ii \theta},R(\eta)e^{\ii \eta}\big)b\big(R(\eta)\big)d\eta,\\
Q(r)(\theta)&\bydef\frac{e^{\ii \theta}}{R(\theta)}   \cdot  \int_0^{2\pi}\int_{0}^{R(\eta)}\nabla _x K_b\big(R(\theta)e^{\ii \theta},\rho e^{\ii \eta}\big)b(\rho)\rho d\rho d\eta,\\
L(r)(\theta)&\bydef\frac{-b'\big(R(\theta)\big)}{R(\theta) b^2\big(R(\theta)\big)}\partial_\theta\left(\int_{0}^{2\pi}\int_{0}^{R(\eta)}K_b\big(R(\theta)e^{\ii \theta},\rho e^{\ii \eta}\big)b(\rho)\rho d\rho d\eta\right),
\end{align*}
for all $\theta\in\mathbb{T}$.
\end{lem}
\begin{proof} Recalling that $\mathcal{F}$ is given by \eqref{contour_dynamics},  we only focus on  differentiating
$$r\mapsto\frac{1}{b\big(R(\theta)\big)}\qquad \text{and}\qquad r\mapsto\int_{0}^{R(\eta)}K_b \big(R(\theta)e^{\ii \theta},\rho e^{\ii \eta}\big)b(\rho)\rho d\rho.$$
For the first function above, it is readily seen that 
$$d_r\left(\frac{1}{b\circ R}\right)[h]=-\frac{b'\circ R}{R\times(b^2\circ R)}h.$$
As for the second one,  by a direct computation,  we find that 
\begin{equation*}
	\begin{aligned}
		d_r\bigg(\int_{0}^{R(\eta)}&K_b\big(R(\theta)e^{\ii \theta},\rho e^{\ii \eta}\big)b(\rho)\rho d\rho\bigg)[h]\\
 &=K_b\big(R(\theta)e^{\ii \theta},R(\eta)e^{\ii \eta}\big)b\big(R(\eta)\big)h(\eta)\\
 &\quad\quad\quad+\left(\frac{e^{\ii \theta}}{R(\theta)}\cdot\int_{0}^{R(\eta)}\nabla_xK_b\big(R(\theta)e^{\ii \theta},\rho e^{\ii \eta}\big)b(\rho)\rho d\rho \right)h(\theta). 
	\end{aligned}
\end{equation*} 
At last, gathering the foregoing two identities completes the proof of the lemma.
\end{proof} 

Now, we show that the linearized operator at the equilibrium $r=0$ acts as a Fourier multiplier whose coefficients are precise in the next lemma. 

\begin{lem}\label{lemma:linearized-OP}
Let $m\in\mathbb{N}^*$, $a>0$, $\Omega\in \mathbb{R}$   and $\alpha\in(0,1).$ Then, for any $h\in X_{m}^{\alpha}$ taking the form $$  h(\theta)=\sum_{n=1}^{\infty}h_n\cos(nm\theta),$$
for all $ \theta\in\mathbb{T}$ and some $h_n\in\mathbb{R}$, it holds that
$$d_r\mathcal{F}(\Omega,0)[h](\theta)=-\sum_{n=1}^{\infty}nm\big(\Omega   - Q(a,0)+  \Lambda_{nm}(a,a)\big)h_n\sin(nm\theta),$$
where $Q(a,0)$ and $\big(\Lambda_{n}(a,a)\big)_{n\in\mathbb{N}^*}$ are defined by \eqref{def Q} and \eqref{def Lambdan}, respectively.
\end{lem}

\begin{proof}
First, observing that 
$$P(0)[h](\theta)=b(a)\sum_{n=1}^{\infty}h_n\int_0^{2\pi}K_b\big(ae^{\ii \theta},ae^{\ii \eta}\big)\cos(nm\eta)d\eta,$$
yields, in view of \eqref{symmetry-kernel}, that
$$P(0)[h](\theta)=b(a)\sum_{n=1}^{\infty}h_n\int_0^{2\pi}K_b\big(a,ae^{\ii (\eta-\theta)}\big)\cos(nm\eta)d\eta.$$
Therefore, by a change of variables, we infer that 
\begin{align*}
P(0)[h](\theta)&=b(a)\sum_{n=1}^{\infty}h_n\int_0^{2\pi}K_b\big(a,ae^{\ii \eta}\big)\cos\big(nm(\eta+\theta)\big)d\eta\\
&=b(a)\cos(nm\theta)\sum_{n=1}^{\infty}h_n\int_0^{2\pi}K_b\big(a,ae^{\ii \eta}\big)\cos(nm\eta)d\eta\\
&\quad-b(a)\sin(nm\theta)\sum_{n=1}^{\infty}h_n\int_0^{2\pi}K_b\big(a,ae^{\ii \eta}\big)\sin(nm\eta)d\eta.
\end{align*}
Now, the second identity in \eqref{symmetry-kernel} implies that the following function is even
$$\eta\mapsto K_b(a,ae^{\ii \eta}).$$
As a consequence, we obtain the oddness of the function
$$\eta\mapsto K_b(a,ae^{\ii \eta})\sin(nm\eta)$$ and by $2\pi$-periodicity, we deduce that 
$$\int_0^{2\pi}K_b\big(a,ae^{\ii \eta}\big)\sin(nm\eta)d\eta=0,$$
thereby establishing that 
$$P(0)[h](\theta)=b(a)\sum_{n=1}^{\infty}h_n\cos(nm\theta)\int_0^{2\pi}K_b\big(a,ae^{\ii \eta}\big)\cos(nm\eta)d\eta.$$
Consequently, it follows that 
\begin{equation*}
	\frac{1}{b\circ R(r)}\partial_\theta\Big( P(r)[h]\Big)\big|_{r=0} = - \sum_{n=1}^{\infty}h_n\sin(nm\theta) \Lambda _{nm}(a,a).
\end{equation*}
Next, we take care of $Q(0)$. To that end, observe first that, for any $\theta\in \mathbb{T}$,
\begin{align*}
Q(0)(\theta)&=\frac{e^{\ii \theta}}{a}\cdot\int_0^{2\pi}\int_{0}^{a}\nabla_xK_b\big(ae^{\ii \theta},\rho e^{\ii \eta}\big)b(\rho)\rho d\rho d\eta\\
&=\frac{1}{a}\int_0^{2\pi}\int_{0}^{a}\partial_{a}\Big(K_b\big(ae^{\ii \theta},\rho e^{\ii \eta}\big)\Big)b(\rho)\rho d\rho d\eta\\
&=\frac{1}{a}\int_0^{2\pi}\int_{0}^{a}\partial_{a}\Big(K_b\big(a,\rho e^{\ii (\eta-\theta)}\big)\Big)b(\rho)\rho d\rho d\eta.
\end{align*}
Thus, by performing a change of variables, again, we find that 
$$Q(0)(\theta)=\frac{1}{a}\int_0^{2\pi}\int_{0}^{a}\partial_{x_1}K_b\big(a,\rho e^{\ii \eta}\big)b(\rho)\rho d\rho d\eta.$$
Thus, by the explicit formula for the right-hand side given in Lemma \ref{velocity_expl_comp}, we deduce that 
 $$Q(0)(\theta)=   -b(a) Q(a,0),$$
 which in turn yields 
 \begin{equation*}
	\frac{1}{b\circ R(r)}\partial_\theta\Big( Q(0)[h]\Big)\big|_{r=0} = - Q(a,0)\sum_{n=1}^{\infty}h_n\sin(nm\theta) .
\end{equation*}
     In addition, we deduce from \eqref{diff0} that
\begin{equation}\label{L=0}
	L(0)\equiv 0. 
\end{equation}
All in all, gathering the preceding identities concludes the proof of the lemma.
\end{proof}
 
  In the next lemma, we perform a detailed  asymptotic analysis of the specific Fourier coefficients $\Lambda_{n}$ that appear in the expression of the linearized operator $d_r\mathcal{F}$ from the preceding lemma. The results of the next lemma will be also crucial in the spectral analysis of the doubly-connected V-states, which will be discussed in Section \ref{sec DC}, later on.

\begin{prop}\label{lemma:eigenfunction}     
Let $\alpha,\beta>0$ and the sequence $\big (\Lambda_n(\alpha,\beta)\big)_{n\in \mathbb{N}^*}$ be as defined in \eqref{def Lambdan}. Then, the following holds:
\begin{enumerate}
	\item The terms of the sequence $\big (\Lambda_n(\alpha,\beta)\big)_{n\in \mathbb{N}^*}$ expand as 
	\begin{equation}\label{exp Lambdan}
		\Lambda_n(\alpha,\beta)=\Lambda_{n}(\beta,\alpha)=\frac{\sqrt{b(\alpha)b(\beta)}}{2n}\mathtt{m}^n(\alpha,\beta)+f_n(\alpha,\beta),
	\end{equation}
	where 
	\begin{equation*}
		 \mathtt{m}(\alpha ,\beta )\bydef\min\left\{\frac{\alpha }{\beta},\frac{\beta}{\alpha }\right\} 
	\end{equation*}
	and   
	\begin{equation}\label{def fn}
		f_n(\alpha,\beta)\bydef\int_{0}^{2\pi}S_b\big(\alpha,\beta e^{\ii \eta}\big)\cos(n\eta)d\eta,
	\end{equation} 
	with $S_b$  being defined as the solution of \eqref{EL:P2}.
	
	\item Defining, on $[0,\infty)$, the function 
		\begin{align*}
		r\mapsto \Theta(r)&\bydef\frac{d}{dr}\left(r\frac{d}{dr}\left(\frac{1}{\sqrt{b(r)}}\right)\right)
	\end{align*} 
	 and setting 
	\begin{equation*}
		A  \bydef  \max\{\alpha,\beta, R_\infty \} ,
	\end{equation*}
	where $R_\infty>0$ being the radius of the ball defined in \eqref{Hb3:eq},
	 it   holds that  
	 	\begin{align}\label{bnd fn_prop} 
		|f_n(\alpha,\beta)|\leq 2A \sqrt{{b(\beta)}}  \frac{\mathtt{m}^n(\alpha,\beta)}{n^2} \left( \delta_{\alpha,\beta} \left( \frac{1}{n} -1   \right) + 1  \right)  \|\Theta\|_{L^{\infty}(\mathbb{R}^+)},
	\end{align}
	for any $n\in \mathbb{N}^*$ satisfying
  	$$ n\geq   \frac{A}{2} \norm{\Theta}_{L^{\infty}(\mathbb{R}^+) } ,$$ 
  	where $\delta_{\alpha,\beta}$ denotes the Kronecker-delta-function of entries $\alpha$ and $\beta$. 
\item The sequence $\big(\Lambda_n(\alpha,\beta )\big)_{n> M(b)}$ is positive,
where
\begin{equation}\label{def M ab}
		 M( b)\bydef   16 A \norm{\Theta}_{L^\infty(\mathbb{R}^+)}  \max\left\{1 ,\frac{1}{\sqrt{b(\alpha)}}\right\}. 
	\end{equation}
	\item Moreover, in the case $\alpha=\beta$,   the sequence $\big(\Lambda_n(\alpha,\alpha)\big)_{n>  M(b)}$ is  decreasing.\end{enumerate} 
 	 \end{prop}

\begin{proof}  
	  \textit{(1)} Fix $n\in\mathbb{N}^*$ and $\alpha,\beta>0$ and observe  that the symmetry property in \eqref{exp Lambdan} is obtained as a consequence of \eqref{symmetry-kernel} by implementing the change of variable $\eta\to-\eta.$ This leaves us with the proof of the second identity of \eqref{exp Lambdan}. To that end, we first write, by using the decomposition \eqref{symmetry-kernel:split} of the kernel $K_b$, that
	  \begin{equation}\label{split Lambdan}
	  	\Lambda_n(\alpha,\beta)=-\frac{\sqrt{b(\alpha)b(\beta)}}{2\pi}\int_{0}^{2\pi}\log|\alpha-\beta e^{\mathrm{i}\eta}|\cos(n\eta)d\eta +\int_{0}^{2\pi}S_b\big(\alpha,\beta e^{\ii \eta}\big)\cos(n\eta)d\eta.
	  \end{equation}
	Now, by virtue of \eqref{log:identity}, we find that
	\begin{equation*} 
		\begin{aligned}
			 -\frac{\sqrt{b(\alpha)b(\beta)}}{2\pi}\int_{0}^{2\pi}\log|\alpha-\beta e^{\mathrm{i}\eta}|\cos(n\eta)d\eta   =\frac{\sqrt{b(\alpha)b(\beta)}}{2n}\mathtt{m}^{n}(\alpha,\beta).
		\end{aligned}
	\end{equation*} 
Hence, identifying the second integral in  \eqref{split Lambdan} with the definition of $f_n$ given by \eqref{def fn} concludes the proof of   \eqref{exp Lambdan}.

 \textit{(2)} Since $S_b(\cdot,y)$ solves the elliptic problem \eqref{EL:P2}, then it can be represented, by applying Proposition \ref{prop.kernel}, as 
$$S_b(x,y)=-\frac{\sqrt{b(y)}}{2\pi}\int_{\mathbb{R}^2}K_b(x,z)\log|y-z|\Delta\left(\frac{1}{\sqrt{b(z)}}\right)dz.$$
Therefore, using the decomposition \eqref{symmetry-kernel:split} in the preceding identity leads to the following splitting
\begin{align*}
S_b(x,y)&=\frac{\sqrt{b(y) b(x) }}{2\pi}\int_{\mathbb{R}^2}  \log| x-z| \log|y-z|\Delta\left(\frac{1}{\sqrt{b(z)}}\right) \sqrt{b(z)}dz\\
&\quad+\frac{\sqrt{b(y)}}{2\pi}\int_{\mathbb{R}^2}S_b(x,z)\log|y-z|\Delta\left(\frac{1}{\sqrt{b(z)}}\right)dz\\
&\bydef\widetilde{S}_{b,1}(x,y)+\widetilde{S}_{b,2}(x,y).
\end{align*}
We now take care of each term separately. Before we get into the details of that, note first that   using the radial symmetry of the bottom topography $b$ and writing the Laplace operator in polar coordinates yields
$$\Delta\left(\frac{1}{\sqrt{b(|z|)}}\right)=\frac{1}{r}\frac{d}{dr}\left(r\frac{d}{dr}\left(\frac{1}{\sqrt{b(r)}}\right)\right)=\frac{1}{r}\Theta(r),$$
for all $z\in \mathbb{R}^2$ with $r=|z|$. Thus, for $\widetilde{S}_{b,1}$, we proceed by utilizing a polar change of variables and twice the identity \ref{log:identity} to find that 
\begin{equation*}
	\begin{aligned}
		\int_{0}^{2\pi}\widetilde{S}_{b,1} &\big(\alpha,\beta e^{\mathrm{i}\eta}\big)\cos(n\eta) d\eta \\
	&=\frac{\sqrt{b(\beta)}}{4\pi^2}\int_{0}^{\infty}\Theta(r)\int_{0}^{2\pi}\log|\alpha-re^{\mathrm{i}\theta}|\int_{0}^{2\pi}\log|\beta e^{\mathrm{i}\eta}-re^{\ii \theta}|\cos(n\eta)d\eta d\theta dr \\
	&=\frac{\sqrt{b(\beta)}}{2\pi}\int_{0}^{\infty}\Theta(r)\int_{0}^{2\pi}\log|\alpha-re^{\mathrm{i}\theta}|\left(-\frac{\mathtt{m}^{n}(r,\beta)}{2n}\right)\cos(n\theta)d\theta dr \\
	&=\frac{\sqrt{b(\beta)}}{4n^2}\int_{0}^{\infty}\Theta(r)\mathtt{m}^{n}(r,\alpha)\mathtt{m}^{n}(r,\beta)dr,
	\end{aligned}
\end{equation*}
whereby, we have shown that 
\begin{equation*}
	\int_{0}^{2\pi}\widetilde{S}_{b,1}\big(\alpha,\beta e^{\mathrm{i}\eta}\big)\cos(n\eta) d\eta = \frac{\sqrt{b(\beta)}}{4n^2}U_{n}(\alpha,\beta),
\end{equation*}
where we set 
\begin{equation*}
	U_{n}(\alpha,\beta) \bydef \int_{0}^{\infty}\Theta(r)\mathtt{m}^{n}(r,\alpha)\mathtt{m}^{n}(r,\beta)dr .
\end{equation*}
Then, we control $U_n$ by employing the elementary computation 
\begin{equation} \label{int:mn}
	\int_0^\infty  \mathtt{m}^n(r,\alpha)\mathtt{m}^n(r,\beta)dr =  \mathtt{m}^{n}(\alpha,\beta)\left(|\alpha-\beta|+\frac{\min(\alpha,\beta)}{2n+1}+\frac{\max(\alpha,\beta)}{2n-1}\right)
\end{equation}
to obtain that 
\begin{equation*}
	\begin{aligned}
		|U_n(\alpha,\beta)|&\leqslant\|\Theta\|_{L^\infty(\mathbb{R}^+)}\int_{0}^{\infty}\mathtt{m}^n(r,\alpha)\mathtt{m}^n(r,\beta)dr \\
	&=\|\Theta\|_{L^{\infty}(\mathbb{R}^+)}\mathtt{m}^{n}(\alpha,\beta)\left(|\alpha-\beta|+\frac{\min(\alpha,\beta)}{2n+1}+\frac{\max(\alpha,\beta)}{2n-1}\right), 
	\end{aligned}
\end{equation*} 
thereby arriving at the bound  
\begin{equation}\label{maj Un}
	|U_n(\alpha,\beta)|\leq C_n\,\mathtt{m}^{n}(\alpha,\beta),
\end{equation} 
where $C_n$ is defined as 
\begin{equation}\label{Cn:def}
	 C_n  \bydef  
	 \begin{cases}
	 	  4 A  \|\Theta\|_{L^{\infty}(\mathbb{R}^+)},& \text{if} \quad \alpha \neq  \beta, \vspace{0.2cm}\\ 
	 	 \displaystyle\frac{2A}{n}  \|\Theta\|_{L^{\infty}(\mathbb{R}^+)},  & \text{if} \quad \alpha = \beta ,
	 \end{cases}
	 \end{equation}
	 where $A$ is defined in the statement of the proposition.  Now, we  expand  the term involving $\widetilde{S}_{b,2}$ by proceeding in a similar way as before to write
\begin{equation*}
	\begin{aligned}
		\int_0^{2\pi}\widetilde{S}_{b,2}&(\alpha,\beta e^{\ii \eta})  \cos(n\eta)d\eta\\
	&=\frac{\sqrt{b(\beta)}}{2\pi}\int_{0}^{\infty}\Theta(r)\int_0^{2\pi}S_b\big(\alpha,re^{\ii \theta}\big)\left(\int_0^{2\pi}\log|\beta e^{\ii \eta}-re^{\ii \theta}|\cos(n\eta)d\eta\right)d\theta dr\\
	&=-\frac{\sqrt{b(\beta)}}{2n}\int_{0}^{\infty}\Theta(r)\mathtt{m}^n(r,\beta)\left(\int_0^{2\pi}S_b\big(\alpha,re^{\ii \theta}\big)\cos(n\theta)d\theta\right)dr\\
	&=-\frac{\sqrt{b(\beta)}}{2n}\int_{0}^{\infty}\Theta(r)\mathtt{m}^n(r,\beta)f_n(\alpha,r)dr.
	\end{aligned}
\end{equation*} 
Therefore, noticing that the hypothesis \eqref{Hb3:eq} implies that $\Theta$ is compactly supported in the ball $ B_{R_\infty} \subset B_A$, 
it follows from the previous computations that  
\begin{equation}\label{fn:identity}
	f_n(\alpha,\beta)=\frac{\sqrt{b(\beta)}}{4n^2}U_n(\alpha,\beta)-\frac{\sqrt{b(\beta)}}{2n}\int_{0}^{A}\Theta(r)\mathtt{m}^n(r,\beta)f_n(\alpha,r)dr.
\end{equation}
 Now, note that the sequence $f_n$ is uniformly bounded due to its definition \eqref{def fn} and the bounds on function $S_b$ from Proposition \ref{prop.kernel} which imply that 
$$\sup_{n\in\mathbb{N}^*}\norm{f_n}_{L^{\infty}\left([0,A]^2\right)}\leq2\pi\norm{S_b}_{L^{\infty}\left([0,A]^2 \right)}<\infty.$$ 
Thus,  for any $\varepsilon\in (0,1)$, defining 
\begin{equation*}
	 	g_{n,\varepsilon}(\alpha)\bydef \sup_{  \beta\in [0,A]}	\left| \frac{f_n(\alpha,\beta)}{\sqrt{b(\beta)} (\mathtt{m}^n(\alpha,\beta)+ \varepsilon) } \right|
\end{equation*}
it follows    
\begin{equation*}
   \begin{aligned}
     \frac{f_n(\alpha,\beta)}{\sqrt{b(\beta)} (\mathtt{m}^n(\alpha,\beta)+ \varepsilon) }  
       & \leq \frac{C_n}{4n^2}   +  g_{n,\varepsilon}(\alpha)\frac{ \norm {\Theta}_{L^\infty(\mathbb{R}^+)}}{2n} \int_0^A  \mathtt{m}^n(r,\beta) \Big (\mathtt{m}^n(r,\alpha)+ \varepsilon \Big  ) dr
		\\
		& \leq \frac{C_n}{4n^2}   +    g_{n,\varepsilon}(\alpha)\frac{A\norm{\Theta}_{L^\infty(\mathbb{R}^+)}}{n},
   \end{aligned}
\end{equation*} 
for the fact that 
$$\sup_{c_1,c_2\in (0,\infty)}\mathtt{m}^n(c_1,c_2)\leq 1.$$
Now, choosing $n\in \mathbb{N}^*$ in such a way that
$$n\geq\frac{A\norm{\Theta}_{L^{\infty}(\mathbb{R}_+)}}{2},$$
we find  
$$g_{n,\varepsilon}(\alpha)\leq\frac{C_n}{2n^2},$$
which, upon taking the limit $\varepsilon \to 0$, yields 
\begin{equation}\label{maj fn}
	|f_n(\alpha,\beta)|\leq\frac{C_n\sqrt{b(\beta)}}{2n^2}\mathtt{m}^n(\alpha,\beta),
\end{equation} 
which completes the justification of the bound on $f_n$ after noticing that 
\begin{equation*}
	C_n \leq 4 A \left( \delta_{\alpha,\beta} \left( \frac{1}{n} -1   \right) + 1  \right)  \|\Theta\|_{L^{\infty}(\mathbb{R}^+)}.
\end{equation*}

\textit{(3)}  We establish now the positivity of $\Lambda_n$ for $n$ sufficiently large. To that end, we write, in view of \eqref{Cn:def} and  \eqref{maj fn}, that  
\begin{equation*}
	\begin{aligned}
		\Lambda_{n}(\alpha,\beta) & \geq \frac{\sqrt{b(\alpha)b(\beta)} }{2n} \mathtt{m}^n(\alpha,\beta)- |f_n(\alpha,\beta)| \\
	&\geq \frac{\mathtt{m}^n(\alpha,\beta) }{2n^2}\left( n \sqrt{b(\alpha)b(\beta)} - 4 A  \sqrt{b(\beta )}  \|\Theta\|_{L^{\infty}\left(\mathbb{R}^+\right) }  \right)  .	\end{aligned}
\end{equation*}
Therefore, it follows that 
\begin{equation*}
		\Lambda_{n}(\alpha,\beta) > 0,
\end{equation*}
for all $n\in \mathbb{N}^*$ with 
\begin{equation*}
	n >      \frac{4A  }{\sqrt{b(\alpha)}} \norm{\Theta}_{L^\infty(\mathbb{R}^+)},
\end{equation*}
which yields the non-negativity of $\Lambda_n(\alpha,\beta)$    under the stronger condition stated in the proposition  
\begin{equation*}
	n > M(b),
\end{equation*} 
where $M(b)$ is introduced in \eqref{def M ab}.

\textit{(4)} Now, let us assume that $\alpha=\beta$ and we compute, for any $n\geq 1$, that 
\begin{align*}
	\Lambda_{n}(\alpha,\alpha)-\Lambda_{n+1}(\alpha,\alpha)&=\frac{b(\alpha)}{2n(n+1)}+f_n(\alpha,\alpha)-f_{n+1}(\alpha,\alpha)\\
	&\geq\frac{b(\alpha)}{2n(n+1)}-|f_n(\alpha,\alpha)|-|f_{n+1}(\alpha,\alpha)|.
\end{align*}
Hence, noticing that the bound \eqref{Cn:def} and  \eqref{maj fn}  imply
$$|f_n(\alpha,\alpha)|\leq\frac{2A\norm{\Theta}_{L^\infty(\mathbb{R}^+)}\sqrt{b(\alpha)}}{n^3} ,  $$
for all $n \geq N(\alpha,\alpha,b),$ it then follows that    
$$\Lambda_{n}(\alpha,\alpha)-\Lambda_{n+1}(\alpha,\alpha)\geq\frac{b(\alpha)}{4n^2}-\frac{4A\|\Theta\|_{L^{\infty}\left(\mathbb{R}^+\right)}\sqrt{  b(\alpha)}}{n^3}\cdot$$
Consequently, for a given $n_0\in\mathbb{N}^*$, we deduce that the sequence $\big(\Lambda_n(\alpha,\alpha)\big)_{n\geq n_0}$ is decreasing if  
$$n_0> A \norm{\Theta}_{L^\infty(\mathbb{R}^+)} \max\left\{\frac{1}{2} ,\frac{16 }{\sqrt{b(\alpha)}}\right\} ,$$
which is clearly implied by the simpler condition that $n_0>M(b)$ as claimed in the statement of the proposition and concludes its proof.
\end{proof}


\subsection{Regularity analysis}\label{section:spectral analysis}
The aim of this section is to provide a detailed analysis of the regularity properties of the functional $\mathcal{F}$. This analysis allows us in particular to  justify  the computations laid out in  Section \ref{section:linearized:operator}, above, which will also serve in the doubly connected case, later on.

In the sequel of the paper, for a fixed integer $m\in \mathbb{N}^*$, and real parameter $\alpha\in (0,1)$, the spaces $X_m^\alpha$ and $Y_m^\alpha$ are defined by
\begin{equation*}
	X_m^\alpha\bydef\left\{f\in C^{1+\alpha}(\mathbb{T}):f(\theta)=\sum_{n=1}^{\infty}f_n\cos(nm\theta),\ f_n\in\mathbb{R},\ \theta\in\mathbb{T}\right\}
\end{equation*}
and 
\begin{equation*}
	Y_m^\alpha\bydef\left\{g\in C^{\alpha}(\mathbb{T}): g(\theta)=\sum_{n=1}^{\infty}g_n\sin(nm\theta),\ g_n\in\mathbb{R},\ \theta\in\mathbb{T}\right\}
\end{equation*} 
and equipped with the usual $C^{1+\alpha}$ and $C^{\alpha}$ norms, respectively. Moreover, for a given $\varepsilon>0$, we set 
$$\mathcal{B}^\alpha_{m,\varepsilon}\bydef\big\{f\in X_m^\alpha: \norm{f}_{C^{1+\alpha}}<\varepsilon\big\}. $$

 In addition to that, allow us to clarify some vectorial notations that will    be  exclusively used in this section.  For any given two vectors $u=(u_1,u_2)\in\mathbb{R}^2$ and $v=(v_1,v_2)\in\mathbb{R}^2$, the bilinear application $$u\otimes v:\mathbb{R}^2\times\mathbb{R}^2\rightarrow\mathbb{R}$$ is defined, for any $x,y\in \mathbb{R}^2$, by
$$ (u\otimes v)(x,y)\bydef(u\cdot x)(v\cdot y).$$
Applying it to the canonical basis $(e_1,e_2)$ of $\mathbb{R}^2$, we obtain, for any $(i,j)\in\{1,2\}^2$, that 
$$ (u\otimes v)(e_i,e_j)=(u\cdot e_i)(v\cdot e_j)=u_iv_j,$$
which gives the following representation of the tensor-product
$$u\otimes v\bydef \begin{pmatrix}
	u_1v_1 & u_1v_2\\
	u_2v_1 & u_2v_2
\end{pmatrix}\in M_2(\mathbb{R}).$$
 Now, for given two matrices $A=(A_{ij})_{1\leq i,j\leq 2} $ and $B=(B_{ij})_{1\leq i,j\leq 2}$ in $M_2(\mathbb{R}),$ we define the double contraction of $A$ and $B$ by
$$A:B\bydef\sum_{1\leq i,j\leq2}A_{ij}B_{ij}.$$
In particular, one also sees that, for given vectors $u=(u_1,u_2)$, $u'=(u_1',u_2')$, $v=(v_1,v_2)$ and $v'=(v_1',v_2')$ in $\mathbb{R}^2$, we have
$$\big(u\otimes v\big):\big(u'\otimes v'\big)=\sum_{1\leq i,j\leq2}u_iv_ju_i'v_j'=\left(\sum_{i=1}^{2}u_iu_i'\right)\left(\sum_{j=1}^{2}v_jv_j'\right)=(u\cdot u')(v\cdot v').$$

\subsubsection{On the continuity of some singular operators}
For a later use, we now discuss some    fundamental   results which will allow us to show the continuity of operators with singular kernels. We first recall an important lemma whose proof can be found in \cite[Lemma 2.6]{HXX22} and \cite[Lemma 1]{HHMV16}.

\begin{lem}\label{lemma:contintuity:kernel} Let $\alpha\in [0,1)$ and consider a measurable complex-valued function $K$, defined on $\mathbb{T}\times \mathbb{T} \setminus \{ (\theta,\theta), \ \theta \in \mathbb{T}\}$. Assume that there exists a constant $C>0$ such that
	\begin{equation*}
		 |K(\theta,\eta)| \leq \frac{C}{\left|\sin\left( \frac{\theta-\eta}{2}\right) \right|^\alpha}, \quad \text{for all} \quad \theta\neq \eta \in \mathbb{T},
	\end{equation*}
	and that $\theta\mapsto K(\theta,\eta)$ is differentiable with
	\begin{equation*}
		 |\partial_\theta K(\theta,\eta)| \leq \frac{C}{\left|\sin\left( \frac{\theta-\eta}{2}\right) \right|^{1+\alpha}} ,  \quad \text{for all} \quad \theta\neq \eta \in \mathbb{T}.
	\end{equation*}
	Then, the complex-valued function defined on $\mathbb{T}$, for any $f\in L^\infty(\mathbb{T})$, by
	\begin{equation*}
		\theta\mapsto\mathcal{T}f(\theta)\bydef\int_0^{2\pi}K(\theta,\eta)f(\eta)d\eta,
	\end{equation*}
	belongs to $C^{\alpha}(\mathbb{T})$ and satisfies
	\begin{equation*}
		\norm{\mathcal{T}f}_{C^\alpha}\lesssim_\alpha C\norm{f}_{L^\infty}.
	\end{equation*}
\end{lem} 
Now, we establish three essential consequences of   Lemma \ref{lemma:contintuity:kernel} which will come in handy in the regularity analysis of contour dynamics, later on.

\begin{cor}\label{corollary:kernel:1}
	Let $a>0$ and $(\alpha,s)\in (0,1)\times [0,1]$. Consider a real-valued functions $r\in C^{1+\alpha}(\mathbb{T})$ with 
	\begin{equation}\label{smallness assumption}
		\norm{r}_{C^{1+\alpha}}\leq\varepsilon_0,
	\end{equation}
	for some small parameter $\varepsilon_0>0$. Accordingly, we define the function
	\begin{equation}\label{R:DEF}
		\theta\mapsto R(\theta)\bydef\sqrt{a^2+2r(\theta)}.
	\end{equation}
	Then, the function   
	\begin{equation*}
		\theta \in \mathbb{T} \mapsto\mathcal{T}_sf(\theta)\bydef\int_0^{2\pi}\log\left|R (\theta)e^{\ii \theta}-sR(\eta)e^{\ii \eta}\right|f(\eta)d\eta,
	\end{equation*}
	defined for any $f\in L^\infty(\mathbb{T})$, belongs to $C^{\alpha}(\mathbb{T})$ and enjoys the bound 
	\begin{equation*}
		\norm{\mathcal{T}_sf}_{C^\alpha}\leq C_\alpha  \norm{f}_{L^\infty},
	\end{equation*}
	for some constant $C_\alpha>0$ which is independent of $s$.
\end{cor} 

\begin{proof}
	The proof directly follows by applying Lemma \ref{lemma:contintuity:kernel} provided that we show that the kernel defined by
	\begin{equation*}
		(\theta,\eta)\mapsto K_s(\theta,\eta)\bydef\log\left|R(\theta)e^{\ii \theta}-sR(\eta)e^{\ii \eta}\right| 
	\end{equation*}
	satisfies the assumptions of Lemma \ref{lemma:contintuity:kernel} for a constant $C$ which is independent of $s$. To that end, observe first, by a continuity argument together with the smallness assumption \eqref{smallness assumption}, that 
	\begin{equation}\label{bnd inf}
		\inf_{\theta\in\mathbb{T}}R(\theta)\geq\frac{a}{2}\quad \text{and} \quad \sup_{\theta\in\mathbb{T}}R(\theta)\leq 2a.
	\end{equation}
	Besides, writing, for any $s\in[0,1]$ and any $\theta,\eta\in\mathbb{T}$, that
	\begin{equation}
		\begin{aligned}\label{mod diff}
			\left|R(\theta)e^{\ii \theta}-sR(\eta)e^{\ii \eta}\right|^2&=R^2(\theta)+s^2R^2(\eta)-2sR(\theta)R(\eta)\cos(\eta-\theta) \\
		&=\big(R(\theta)-sR(\eta)\big)^2+4sR(\theta)R(\eta)\sin^2\left(\tfrac{\theta-\eta}{2}\right),
		\end{aligned}
	\end{equation} 
it then follows, by continuity and  using \eqref{mod diff} and \eqref{bnd inf}, that there exists a small parameter $s_0\in(0,1)$ in such a way that
$$ \left|R(\theta)e^{\ii \theta}-sR(\eta)e^{\ii \eta}\right|\geq\frac{R(\theta)}{2}\geq\frac{a}{4}\left|\sin\left(\tfrac{\theta-\eta}{2}\right)\right|, \quad \text{for all} \quad  s\in[0,s_0].$$
On the other hand,  by utilizing \eqref{mod diff} and \eqref{bnd inf} again, we obtain that
$$ \left|R(\theta)e^{\ii \theta}-sR(\eta)e^{\ii \eta}\right|\geq 2\sqrt{s_0R(\theta)R(\eta)}\left|\sin\left(\tfrac{\theta-\eta}{2}\right)\right|\geq a\sqrt{s_0}\left|\sin\left(\tfrac{\theta-\eta}{2}\right)\right|, \quad \text{for all} \quad s\in[s_0,1] .$$
Putting together the last two estimates leads to the lower bound
\begin{equation}\label{lower-bound-1}
	 \inf_{s\in[0,1]}\left|R(\theta)e^{\ii \theta}-sR(\eta)e^{\ii \eta}\right|\geq\frac{a\sqrt{s_0}}{4}\left|\sin\left(\tfrac{\theta-\eta}{2}\right)\right|, \quad \text{for all} \quad \theta,\eta\in\mathbb{T}.
\end{equation}
Moreover, using the fact that
\begin{equation}\label{log:bound}
	  x\mapsto x^\gamma\big|\log x\big|\in L^\infty_{\loc}(\mathbb{R}^+),  
\end{equation} 
for all $\gamma>0$, we then deduce  that 
\begin{equation*}
	 |K_s(\theta,\eta)|\leq\frac{C}{\left|\sin\left(\frac{\theta-\eta}{2}\right)\right|^\gamma},
\end{equation*} 
 for all $\eta\neq\theta\in\mathbb{T}$ and  some constant $C>0$ which is independent of $s$.	
 
 Likewise, computing that 
\begin{equation*}
	\partial_\theta K_s(\theta,\eta)=\partial_\theta\left(R(\theta)e^{\ii \theta}\right)\cdot\frac{R(\theta)e^{\ii \theta}-sR(\eta)e^{\ii \eta}}{\left|R(\theta)e^{\ii \theta}-sR(\eta)e^{\ii \eta}\right|^2},
\end{equation*}
and utilizing the bounds \eqref{smallness assumption}, \eqref{bnd inf} and \eqref{lower-bound-1}, we infer that
\begin{equation*}
	 |\partial_\theta K_s(\theta,\eta)|\leq\frac{C}{\left|\sin\left(\frac{\theta-\eta}{2}\right)\right|}\leq\frac{C}{\left|\sin\left(\frac{\theta-\eta}{2}\right)\right|^{1+\alpha}},
\end{equation*}
for all $\eta\neq\theta\in\mathbb{T}$, up to suitably change the constant $C$, above. At last, Lemma \ref{lemma:contintuity:kernel} applies to conclude the proof of Corollary \ref{corollary:kernel:1}. 
\end{proof}

\begin{cor}\label{corollary:kernel:2}
	Let $\alpha,s$, $R_1$ and $R_2$ be  respectively defined in terms of  $r_1$ and $r_2$ as in Corollary \ref{corollary:kernel:1} by also assuming that $r_1$ and $r_2$   satisfy  the smallness assumption \eqref{smallness assumption}. Further define the following operator acting on $L^\infty(\mathbb{T})$ by
	\begin{equation*}
		f\mapsto \mathcal{\widetilde{T}}_sf\bydef \int_0^{2\pi}\widetilde{K}_s(\cdot,\eta) f(\eta)d\eta,
	\end{equation*}
	where we set, for any $\theta\neq\eta\in\mathbb{T}$, 
	\begin{equation*}
		\widetilde{K}_s(\theta,\eta)\bydef\log\left|R_1(\theta)e^{\ii \theta}-sR_1(\eta)e^{\ii \eta}\right|-\log\left|R_2(\theta)e^{\ii \theta}-sR_2(\eta)e^{\ii \eta}\right|.
	\end{equation*}
	Then, $\mathcal{\widetilde{T}}_sf \in C^{\alpha}(\mathbb{T})$ and it holds, for any $\beta\in [0,\alpha)$, that 
	\begin{equation*}
		\norm{\mathcal{\widetilde{T}}_sf}_{C^\alpha} \lesssim_{\alpha}\norm{r_1-r_2}_{C^1}^\beta \norm {f}_{L^\infty}.
	\end{equation*}
\end{cor}

\begin{proof} First remark that the mean value theorem implies, for all $x,y \in \mathbb{R}$, that 
	\begin{equation*}
	 \big|\log x-\log y\big|\leq\frac{|x-y|}{\min\{x,y\}}\cdot
	\end{equation*}
	Thus, combining this estimate with  \eqref{lower-bound-1} and employing the triangular inequality, yields
	\begin{equation*} 
		\begin{aligned}
			\big|\widetilde{K}_s(\theta,\eta)\big|&\lesssim\frac{\Big| \left|R_1(\theta)e^{\ii \theta}-sR_1(\eta)e^{\ii \eta}\right|-\left|R_2(\theta)e^{\ii \theta}-sR_2(\eta)e^{\ii \eta}\right|\Big|}{\left|\sin\left(\frac{\theta-\eta}{2}\right)\right|}\\
			&\lesssim\frac{\norm{R_1-R_2}_{L^\infty}}{\left|\sin\left(\frac{\theta-\eta}{2}\right)\right|}\cdot
		\end{aligned}
	\end{equation*}
Now, observing that  
\begin{equation}\label{alg R12}
	R_1-R_2=2\frac{r_1-r_2}{R_1+R_2},
\end{equation}
 and combining this identity with \eqref{bnd inf}, it is readily seen then that 
$$\norm{R_1-R_2}_{L^{\infty}}\lesssim\norm{r_1-r_2}_{L^{\infty}},$$
which implies in turn that
$$\big|\widetilde{K}_s(\theta,\eta)\big|\lesssim\frac{\norm{r_1-r_2}_{L^\infty}}{\left|\sin\left(\frac{\theta-\eta}{2}\right)\right|}\cdot$$
On the other hand, we deduce from \eqref{log:bound}  that
	\begin{equation*}
		|\widetilde{K}_s(\theta,\eta)|\leq\frac{C}{\left|\sin\left(\frac{\theta-\eta}{2}\right)\right|^\gamma},
	\end{equation*} 
	for any $\gamma>0$. Hence, in view of the preceding two bounds, we write, for any $\theta\neq\eta\in\mathbb{T}$ and $\beta\in(0,1)$ to be chosen in a moment, that 
	\begin{align*}
		\big|\widetilde{K}_s(\theta,\eta)\big|&=\big|\widetilde{K}(\theta,\eta)\big|^{\beta}\big|\widetilde{K}(\theta,\eta)\big|^{1-\beta}\\
		&\lesssim\frac{\norm{r_1-r_2}_{L^\infty}^\beta}{\left|\sin\left(\frac{\theta-\eta}{2}\right)\right|^{\beta+\gamma(1-\beta)}}\cdot
	\end{align*}
	Therefore, choosing $\beta>0$ small enough by setting, for instance, that 
	\begin{equation*}
		\gamma=\frac{\alpha-\beta}{1-\beta}\quad\text{and}\quad\beta\in(0,\alpha),
	\end{equation*}
	we end up with the desired control 
	\begin{equation}\label{bound:tilde-K}
		|\widetilde{K}_s(\theta,\eta)|\lesssim\frac{\norm{r_1-r_2}_{L^\infty}^\beta}{\left|\sin\left(\frac{\theta-\eta}{2}\right)\right|^\alpha}\cdot
	\end{equation} 
	Likewise, we first observe that 
	\begin{equation*}
		\begin{aligned}
			|\partial_\theta \widetilde{K}_s(\theta,\eta)|&\leq\frac{\left|\partial_\theta\left(R_1(\theta)e^{\ii \theta}\right)-\partial_\theta\left(R_2(\theta)e^{\ii \theta}\right)\right|}{\left|R_2(\theta)e^{\ii \theta}-sR_2(\eta)e^{\ii \eta}\right|}\\ 
			&\quad+\left|\partial_\theta\left(R_1(\theta)e^{\ii \theta}\right)\right|\left|\frac{R_1(\theta)e^{\ii \theta}-sR_1(\eta)e^{\ii \eta}}{\left|R_1(\theta)e^{\ii \theta}-sR_2(\eta)e^{\ii \eta}\right|^2}-\frac{R_2(\theta)e^{\ii \theta}-sR_2(\eta)e^{\ii \eta}}{\left|R_2(\theta)e^{\ii \theta}-sR_2(\eta)e^{\ii \eta}\right|^2}\right|.
		\end{aligned}
	\end{equation*}
	Then writing, in view of the elementary bounds
	\begin{equation*}
		\left|\frac{X_1}{|X_1|^2}-\frac{X_2}{|X_2|^2}\right|\leq3\frac{|X_1-X_2|}{\min\{|X_1|,|X_2|\}^2},
	\end{equation*}
	and 
	\begin{equation*}
		\left|\frac{X_1}{|X_1|^2}-\frac{X_2}{|X_2|^2}\right|\leq\frac{2}{\min\{|X_1|,|X_2|\}},
	\end{equation*}
	with the values
	\begin{equation*}
		X_k= X_k(\theta,\eta)=R_k(\theta)e^{\ii \theta}-sR_k(\eta)e^{\ii \eta},\qquad k\in\{1,2\},
	\end{equation*}
	that
	\begin{equation*}
		\begin{aligned}
			\left|\frac{X_1}{|X_1|^2}-\frac{X_2}{|X_2|^2}\right|&=\left|\frac{X_1}{|X_1|^2}-\frac{X_2}{|X_2|^2}\right|^\alpha\left|\frac{X_1}{|X_1|^2}-\frac{X_2}{|X_2|^2}\right|^{1-\alpha}\\
			&\lesssim_\alpha\frac{|X_1-X_2|^\alpha}{\min\{|X_1|,|X_2|\}^{1+\alpha}},
		\end{aligned}
	\end{equation*}
	it follows, by further utilizing the bound   \eqref{lower-bound-1},  that 
	\begin{equation*}
		\begin{aligned}
			|\partial_\theta\widetilde{K}_s(\theta,\eta)| 
			&\lesssim_\alpha\frac{\norm{R_1-R_2}_{C^1}}{\left|\sin\left(\frac{\theta-\eta}{2}\right)\right|}+\frac{\norm{R_1-R_2}_{L^\infty}^\alpha}{\left|\sin\left(\frac{\theta-\eta}{2}\right)\right|^{1+\alpha}}\cdot
		\end{aligned}
	\end{equation*}
Now, notice that deriving \eqref{alg R12} and estimating the resulting identity in $L^{\infty}$, yields
$$\norm{R_1-R_2}_{C^1}\lesssim\norm{r_1-r_2}_{C^1}.$$
	Hence, since $\beta<\alpha$ and $r_1,r_2$ satisfy the smallness condition \eqref{smallness assumption}, we arrive at the control
	\begin{equation*}
		\begin{aligned}
			|\partial_\theta\widetilde{K}_s(\theta,\eta)|&\lesssim\frac{\norm{r_1-r_2}_{C^1}^\alpha}{\left|\sin\left(\frac{\theta-\eta}{2}\right)\right|^{1+\alpha}}\\
			&\lesssim\frac{\norm{r_1-r_2}_{C^1}^\beta}{\left|\sin\left(\frac{\theta-\eta}{2}\right)\right|^{1+\alpha}}\cdot
		\end{aligned}
	\end{equation*} 
	At last, combining the preceding control with \eqref{bound:tilde-K} allows us to apply Lemma \ref{lemma:contintuity:kernel}, thereby completing the proof of Corollary \ref{corollary:kernel:2}.
\end{proof}

We conclude this section by stating the following corollary whose proof can be done along the same lines as shown for Corollaries \ref{corollary:kernel:1} and \ref{corollary:kernel:2}.

\begin{cor}\label{corollary:kernel:3}
	Under the assumptions of Corollaries \ref{corollary:kernel:1} and \ref{corollary:kernel:2}, further consider a complex valued function $h\in C^{1+\alpha}(\mathbb{T})$ and define the operators $\mathcal{T}_{h,R}$ and $\mathcal{\widetilde{T}}_{h}$ by
	\begin{equation*}
		f\mapsto\mathcal{T}_{h,R}f\bydef\int_0^{2\pi}K_{h,R}(\cdot,\eta)f(\eta)d\eta,
	\end{equation*}
	and 
	\begin{equation*}
		f\mapsto\mathcal{\widetilde{T}}_{h}f\bydef\int_0^{2\pi}\widetilde{K}_{h}(\cdot,\eta)f(\eta)d\eta,
	\end{equation*}
	where, we set, for any $\theta\neq\eta\in\mathbb{T}$, that 
	\begin{equation*}
		K_{h,R}(\theta,\eta)\bydef\frac{R(\theta)e^{\ii \theta}-R(\eta)e^{\ii \eta}}{|R(\theta)e^{\ii \theta}-R(\eta)e^{\ii \eta}|^2}\cdot\left(\frac{h(\theta)e^{\ii \theta}}{R(\theta)}-\frac{h(\eta)e^{\ii \eta}}{R(\eta)}\right) ,
	\end{equation*}
	and
	\begin{equation*}
		\widetilde{K}_{h}(\theta,\eta)\bydef K_{h,R_1}(\theta,\eta)-K_{h,R_2}(\theta,\eta).
	\end{equation*}
	Then, $\mathcal{T}_{h,R}f$ and $\mathcal{\widetilde{T}}_hf$ belong to $C^{\alpha}(\mathbb{T})$ with
	\begin{equation*}
		\norm{\mathcal{T}_{h,R} f}_{C^\alpha}\lesssim_\alpha\norm{h}_{C^{1+\alpha}}\norm {f}_{L^\infty},
	\end{equation*}
	and 
	\begin{equation*}
		\norm{\mathcal{\widetilde{T}}_{h}f}_{C^\alpha}\lesssim_{\alpha}\norm{r_1-r_2}_{C^1}^\beta\norm{h}_{C^{1+\alpha}}\norm{f}_{L^\infty},
	\end{equation*}
	for all $\beta\in (0,\alpha)$.
\end{cor}

The next two lemmas   can be seen as variations of Corollaries \ref{corollary:kernel:1} and \ref{corollary:kernel:2}. Even though their statements are similar, the proofs of these lemmas require a careful attention and they do not straitforwardly follow from the statements of the previous results of this section.
\begin{lem}\label{lemma:A:1}
	Let $f$ be an element of $L^\infty(\mathbb{T})$ and  $R$ be as in Corollary \ref{corollary:kernel:1}. Then, the function defined, for any $(\rho,\theta)\in \mathbb{R}^+\times \mathbb{T}$, by 
	\begin{equation*}
		A(\rho ,\theta) \bydef \int_0^{2\pi} \log\left| \rho e^{i\theta} - R(\eta)e^{i\eta}\right| f(\eta) d\eta
	\end{equation*}
	enjoys the bound
	\begin{equation*}
		\norm {\partial_\rho  A}_{ L^p \big ( (0,\delta )\times \mathbb{T}\big)} \lesssim \norm {f}_{L^\infty(\mathbb{T})},
	\end{equation*}
	for any  finite  $\delta>0$ and any $p\in [1,\infty)$. 
\end{lem} 
\begin{rem}
	A direct consequence of the preceding lemma is that the function 
	\begin{equation*}
		x\mapsto \int_0^{2\pi} \log\left| x - R(\eta)e^{i\eta}\right| f(\eta) d\eta
	\end{equation*}
	belongs to $W^{1,p}_{\loc}(\mathbb{R}^2)$ and, up to a multiplicative constant, obeys the same bound from the lemma in terms of the given function $f$.
\end{rem}

\begin{proof} To get started, noticing
\begin{equation*}
	\partial_\rho A(\rho,\theta) = \int_0^{2\pi } \frac{\rho - R(\eta) + 2R(\eta) \sin^2  (\frac{\eta - \theta}{2}) }{\big( \rho-R(\eta)\big)^2 + 4\rho R(\eta) \sin ^2 (\frac{\eta - \theta}{2})} f(\eta) d\eta,
\end{equation*}
 and recalling the bounds from \eqref{bnd inf}
\begin{equation}\label{bound:R:***}
	\frac{a}{2}\leq  R(\eta) \leq 2a, \quad \text{for all} \quad \eta \in \mathbb{T},
\end{equation}
yields that 
\begin{equation*}
\left(  \frac{\rho - R(\eta) + 2R(\eta) \sin^2 (\frac{\eta - \theta}{2}) }{\big( \rho-R(\eta)\big)^2 + 4\rho R(\eta) \sin ^2 (\frac{\eta - \theta}{2})} \right) \mathds {1}_{\{0<\rho<\frac{a}{4}\} \cup \{4a<\rho<\delta\}}= O(1),  
 \end{equation*}
	for all $ \theta,\eta \in \mathbb{T}$. Therefore, we deduce that  
	\begin{equation}\label{ES:1:dr-A}
		\norm {\partial_{\rho} A}_{ L^p \big ( (0, \frac{a}{4})\cup (4a,\delta )\times \mathbb{T}\big)} \lesssim \norm {f}_{L^\infty(\mathbb{T})}.
	\end{equation}
	It remains now to   establish the same control on the compact domain $(\frac{a}{4}, 4a)$, where the variable $\rho$ is away from the origin but   within   the singular set of the function
	  \begin{equation*}
  D(\rho, \theta, \eta)\bydef \frac{\rho - R(\eta) + 2R(\eta) \sin^2 (\frac{\eta - \theta}{2}) }{\big( \rho-R(\eta)\big)^2 + 4\rho R(\eta) \sin ^2 (\frac{\eta - \theta}{2})} \mathds {1}_{\{\frac{a}{4} <\rho< 4a \}}  .  
 \end{equation*}
To that end, letting $\varepsilon\in (0,1)$ to be chosen in a little while,  we first notice the bound 
\begin{equation}\label{D:ES:ref}
	\begin{aligned}
		|  D(\rho, \theta, \eta)|
		& \leq \left(  \frac{1}{\big(\rho R(\eta)\big)^\varepsilon}\frac{1}{\big |\rho-R(\eta)\big|^{1-2\varepsilon} \sin^{2\varepsilon}\left( \frac{\theta-\eta}{2} \right)}+ \frac{1}{\rho}   \right)  \mathds {1}_{\{\frac{a}{4} <\rho< 4a \}}\\
		& \lesssim  \left(    \frac{1}{\big |\rho-R(\eta)\big|^{1-2\varepsilon} \sin^{2\varepsilon}\left( \frac{\theta-\eta}{2} \right)} +1 \right)  \mathds {1}_{\{\frac{a}{4} <\rho< 4a \}},
	\end{aligned}
\end{equation}
which is obtained by decomposing the nominator in the definition of the function $D$ along with the elementary interpolation inequality 
\begin{equation}\label{interpolation_AA}
	c_1^2+ c_2^2 \gtrsim c_1^{2\varepsilon} c_2^{2(1-\varepsilon)}, 
\end{equation}
which is valid for any positive numbers $c_1$ and $c_2$ and any parameter $\varepsilon\in [0,1]$.  Thus, for any given $p\in (1,\infty)$, we obtain  that 
\begin{equation*}
	\norm {D(\cdot, \theta, \eta)}_{L^p(\mathbb{R}^+)} \lesssim 1 + \frac{1}{\sin^{2\varepsilon}\left( \frac{\theta-\eta}{2} \right)} \left(\int_{\frac{a}{4}}^{4a} \frac{d\rho}{\big |\rho-R(\eta)\big|^{p(1-2\varepsilon)} } \right)^{\frac{1}{p}},
\end{equation*}
whence, by making an adequate change of variables and  choosing $\varepsilon$ sufficiently close to $\frac{1}{2}$ in such a way that 
\begin{equation}\label{epsilon_choice}
	p(1-2\varepsilon) <1  \quad \text{and} \quad 0<\varepsilon<\frac{1}{2},
\end{equation} 
we infer that 
\begin{equation*}
	\begin{aligned}
		\norm {D(\cdot, \theta, \eta)}_{L^p(\mathbb{R}^+)} 
		&\lesssim 1 + \frac{1}{\sin^{2\varepsilon}\left( \frac{\theta-\eta}{2} \right)} \left(\int_{0}^{4a} \frac{d\sigma }{ \sigma ^{p(1-2\varepsilon)} } \right)^{\frac{1}{p}}
		\\
		& \lesssim 1 + \frac{1}{\sin^{2\varepsilon}\left( \frac{\theta-\eta}{2} \right)} ,
	\end{aligned}
\end{equation*}
for any $\theta,\eta\in \mathbb{T}$. Hence, by virtue of the choice $\varepsilon<\frac{1}{2}$, it follows that 
\begin{equation*}
	\int_0^{2\pi} \norm{ D(\cdot, \theta,\eta)}_{L^p(\mathbb{R}^+)} |f(\eta)| d\eta \lesssim \norm {f}_{L^\infty(\mathbb{T})},
\end{equation*}
where we have utilized the fact that 
\begin{equation*}
	\eta \mapsto \frac{1}{\sin^{2\varepsilon}\left( \frac{\theta-\eta}{2} \right)} \in L^1(\mathbb{T}),
\end{equation*}
uniformly in  $\theta \in \mathbb{T}$ as soon as $\varepsilon<\frac{1}{2}$.  Consequently, we deduce that
\begin{equation*}
		\norm {\partial_{\rho} A}_{ L^p \big ( ( \frac{a}{4}, 4a  )\times \mathbb{T}\big)} \lesssim \norm {f}_{L^\infty(\mathbb{T})}.
	\end{equation*}
All in all, combining this bound with \eqref{ES:1:dr-A} concludes the proof of the lemma.
\end{proof}

The next lemma in this section establishes  a continuity aspect with respect to the functional $R$ in the preceding lemma. 
\begin{lem}\label{lemma:A:2}
	Let $R_1$ and $R_2$ be  respectively defined in terms of  $r_1$ and $r_2$ as in Corollary \ref{corollary:kernel:1}. Furtehr define, for $i\in \{1,2\}$ and any given function $f\in L^\infty(\mathbb{T})$,
	\begin{equation*}
		A_i(\rho ,\theta) \bydef \int_0^{2\pi} \log\left| \rho e^{i\theta} - R_i(\eta)e^{i\eta}\right| f(\eta) d\eta,
	\end{equation*} 
	for any $\theta,\eta\in \mathbb{T}$. Then, it holds that
	\begin{equation*}
		\norm{\partial_\rho A_1 - \partial_\rho A_2}_{ L^p \big ( (0,\delta )\times \mathbb{T}\big)} \lesssim \norm{R_1-R_2}_{L^\infty(\mathbb{T})} ^{\beta}\norm {f}_{L^\infty(\mathbb{T})},
	\end{equation*}
	for any $\beta\in [0,\frac{1}{p})$ and   $\delta>0$.
\end{lem}
\begin{proof}
	We begin with noticing that 
	\begin{equation*}
	\begin{aligned}
		 &  \left|  \frac{\rho - R_1(\eta) + 2R_1(\eta) \sin^2 (\frac{\eta - \theta}{2}) }{\big( \rho-R_1(\eta)\big)^2 + 4\rho R_1(\eta) \sin ^2 (\frac{\eta - \theta}{2})} -  \frac{\rho - R_2(\eta) + 2R_2(\eta) \sin^2 (\frac{\eta - \theta}{2}) }{\big( \rho-R_2(\eta)\big)^2 + 4\rho R_2(\eta) \sin ^2 (\frac{\eta - \theta}{2})} \right| \mathds {1}_{\{0<\rho<\frac{a}{4}\} \cup \{4a<\rho<\delta\}}
		 \\
		 & \qquad = O\left (|R_1(\eta)-R_2(\eta)|\right),  
	\end{aligned}
	\end{equation*}
	for all $\theta,\eta \in \mathbb{T}$, whose justification solely relies on the bound \eqref{bound:R:***} and the  fact that  the variable $\rho$ is localized away from the potential values of $R(\eta)$. Thus, it follows that 
	\begin{equation}\label{A1-A2:1}
		\norm{\partial_\rho A_1 - \partial_\rho A_2}_{ L^p \big ( (0, \frac{a}{4})\cup (4a,\delta )\times \mathbb{T}\big)} \lesssim \norm{R_1-R_2}_{L^\infty(\mathbb{T})} \norm {f}_{L^\infty(\mathbb{T})}.
		\end{equation}
		Now, we focus our attention on obtaining a similar bound within the domain $ ( \frac{a}{4}, 4a  )\times \mathbb{T}$ where the function 
		\begin{equation*}
			\begin{aligned}
				&D_1(\rho,\theta,\eta) - D_2(\rho,\theta,\eta)
				\\
				& \quad \bydef \left(  \frac{\rho - R_1(\eta) + 2R_1(\eta) \sin^2 (\frac{\eta - \theta}{2}) }{\big( \rho-R_1(\eta)\big)^2 + 4\rho R_1(\eta) \sin ^2 (\frac{\eta - \theta}{2})} -  \frac{\rho - R_2(\eta) + 2R_2(\eta) \sin^2 (\frac{\eta - \theta}{2}) }{\big( \rho-R_2(\eta)\big)^2 + 4\rho R_2(\eta) \sin ^2 (\frac{\eta - \theta}{2})} \right) \mathds {1}_{\{ \frac{a}{4}< \rho<   4a  \} }
			\end{aligned}
		\end{equation*}
		is singular. Here, we need to perform additional splittings. We proceed by introducing the notation 
		\begin{equation*}
			\alpha= \alpha (\eta)\bydef |R_1(\eta)-R_2(\eta)| \leq 1
		\end{equation*}
		 and, accordingly, writing 
		\begin{equation*}
			\begin{aligned}
				|D_1(\rho,\theta,\eta) - D_2(\rho,\theta,\eta)| 
				&\leq \Big (  |D_1(\rho,\theta,\eta)| + |  D_2(\rho,\theta,\eta)|  \Big )\mathds{1}_{\{ |\rho- R_1(\eta)| \leq 2\alpha \} \cap \{ |\rho- R_2(\eta)| \leq  2\alpha \}}
				\\
				&\quad + \big| D_1(\rho,\theta,\eta) - D_2(\rho,\theta,\eta)\big|   \mathds{1}_{\{ |\rho- R_1(\eta)| > 2\alpha \} \cap \{   |\rho- R_2(\eta)| \geq   \alpha \}}
				\\
				&\quad + \big| D_1(\rho,\theta,\eta) - D_2(\rho,\theta,\eta)\big|   \mathds{1}_{\{ |\rho- R_2(\eta)| > 2\alpha \} \cap \{  |\rho- R_1(\eta)| \geq   \alpha \}}
				\\
				&\leq \Big (  |D_1(\rho,\theta,\eta)| + |  D_2(\rho,\theta,\eta)|  \Big )\mathds{1}_{\{ |\rho- R_1(\eta)| \leq 2\alpha \} \cap \{ |\rho- R_2(\eta)| \leq  2\alpha \}}
				\\
				&\quad + 2\big|D_1(\rho,\theta,\eta) - D_2(\rho,\theta,\eta)\big|  \mathds{1}_{\{ |\rho- R_1(\eta)| > \alpha \} \cap \{   |\rho- R_2(\eta)| \geq   \alpha \}}
				\\ 
				 &\bydef \mathcal{D}_1 (\rho,\theta,\eta) + \mathcal{D}_2 (\rho,\theta,\eta) .
			\end{aligned}
		\end{equation*}
		Notice that the validity of this decomposition is due to the partition 
		\begin{equation*}
			\mathbb{R}^+ \times \mathbb{T} = \left(\{ |\rho- R_1(\eta)| \leq 2\alpha \} \cap \{ |\rho- R_2(\eta)| \leq  2\alpha \} \right) \cup \{ |\rho- R_1(\eta)| > 2\alpha \}  \cup \{ |\rho- R_2(\eta)| > 2\alpha \} ,
		\end{equation*}
		and the triangular inequality 
		\begin{equation*}
			|\rho-R_1(\eta)| \geq |\rho-R_2(\eta)|- \alpha \geq \alpha,
		\end{equation*}
		in the case $ |\rho - R_2(\eta)| >2\alpha$ and \textit{vice versa}. Next, we take care of $\mathcal{D}_1$ and $\mathcal{D}_2$ separately. To that end, by means of the same arguments laid out in the proof of Lemma \ref{lemma:A:1} to estimate the function $D$ therein (see in particular \eqref{D:ES:ref}), we infer that 
		\begin{equation*}
			\begin{aligned}
				\norm{\mathcal{D}_1 (\cdot ,\theta,\eta)}_{L^p(\mathbb{R}^+)} 
				&\lesssim \sum_{i=1}^2\left( \int_{\frac{a}{4}}^{4a}  \mathds{1}_{\{ |\rho- R_i(\eta)| \leq 2\alpha \} }  d\rho \right)^\frac{1}{p} 
				\\
				&\quad + \frac{1}{\sin^{2\varepsilon}\left( \frac{\theta-\eta}{2} \right)} \sum_{i=1}^2 \left(\int_{\frac{a}{4}}^{4a} \frac{ \mathds{1}_{\{ |\rho- R_i(\eta)| \leq 2\alpha \} }}{\big |\rho-R_i(\eta)\big|^{p(1-2\varepsilon)} }  d\rho\right)^{\frac{1}{p}}
			\\
			& \lesssim  \alpha ^\frac{1}{p} +  \frac{1}{\sin^{2\varepsilon}\left( \frac{\theta-\eta}{2} \right)} \left(\int_{0}^{2\alpha} \frac{d\sigma }{ \sigma ^{p(1-2\varepsilon)} } \right)^{\frac{1}{p}}
			\\& \lesssim  \alpha ^\frac{1}{p} +  \frac{1}{\sin^{2\varepsilon}\left( \frac{\theta-\eta}{2} \right)} \alpha^{\frac{1}{p} - ( 1-2\varepsilon)} ,
			\end{aligned}
		\end{equation*}
		for all $\theta,\eta\in \mathbb{T}$, where the smallness condition \eqref{epsilon_choice} on the choice of $\varepsilon\in (0,\frac{1}{2})$ was employed in the last line. Hence,   it follows that 
\begin{equation}\label{D1:bound}
	\int_0^{2\pi} \norm{\mathcal{D}_1 (\cdot, \theta,\eta)}_{L^p(\mathbb{R}^+)} |f(\eta)| d\eta \lesssim \norm{R_1-R_2}_{L^\infty(\mathbb{T})}^{\beta} \norm {f}_{L^\infty(\mathbb{T})},
\end{equation}
uniformly in $\theta\in \mathbb{T}$, where we set 
\begin{equation}\label{beta:def}
	\beta \bydef \frac{1}{p}- (1-2\varepsilon)>0.
\end{equation}
We turn our attention now to the estimate of $\mathcal{D}_2$. To that end, we write that 
\begin{equation*}
	\begin{aligned}
		(D_1 & - D_2)(\rho,\theta,\eta)  =\frac{\big (R_2(\eta)-R_1(\eta)\big) \big ( 1-  2\sin^2 (\tfrac{\eta - \theta}{2}) \big) }{\big( \rho-R_2(\eta)\big)^2 + 4\rho R_2(\eta) \sin ^2 (\frac{\eta - \theta}{2})}  
		\\
		& \qquad \qquad \qquad +\Big(  \rho - R_1(\eta) + 2R_1(\eta) \sin^2 (\tfrac{\eta - \theta}{2}) \Big ) 
		\\
		& \quad \times \left(  \frac{1}{\big( \rho-R_1(\eta)\big)^2 + 4\rho R_1(\eta) \sin ^2 (\frac{\eta - \theta}{2})} 
		-  \frac{1}{\big( \rho-R_2(\eta)\big)^2 + 4\rho R_2(\eta) \sin ^2 (\frac{\eta - \theta}{2})} \right)
		\\ 
		&= \frac{\big (R_2(\eta)-R_1(\eta)\big)  \big ( 1-  2\sin^2 (\frac{\eta - \theta}{2}) \big) }{\big( \rho-R_2(\eta)\big)^2 + 4\rho R_2(\eta) \sin ^2 (\frac{\eta - \theta}{2})}  
		\\
		& \qquad \qquad \qquad +\Big(  \rho - R_1(\eta) + 2R_1(\eta) \sin^2 \left (\tfrac{\eta - \theta}{2}  \right)\Big)  \big (R_2(\eta)-R_1(\eta)\big)
		\\
		& \quad \times \left(  \frac{  (\rho-R_1(\eta)) + (\rho-R_2(\eta)) + 2\rho \sin^2 (\frac{\eta - \theta}{2}) }{\left(\big( \rho-R_1(\eta)\big)^2 + 4\rho R_1(\eta) \sin ^2 (\frac{\eta - \theta}{2})\right) \left(\big( \rho-R_2(\eta)\big)^2 + 4\rho R_2(\eta) \sin ^2 (\frac{\eta - \theta}{2})\right)} \right).
	\end{aligned}
\end{equation*} 
Therefore, employing \eqref{bound:R:***} once again together with the interpolation inequality \eqref{interpolation_AA} for suitable values of parameters leads to the control
\begin{equation*}
	\begin{aligned}
		 \big|( D_1 & - D_2)(\rho,\theta,\eta)\big|  
		 \\
		 & \lesssim \big |R_1(\eta)-R_2(\eta)\big| \left( 1+  \frac{ 1 }{|\rho-R_1(\eta)|^{2(1-\varepsilon)} \sin^{2\varepsilon}(\frac{\theta-\eta}{2})}  +  \frac{ 1 }{|\rho-R_2(\eta)|^{2(1-\varepsilon)} \sin^{2\varepsilon}(\frac{\theta-\eta}{2})} \right),
	\end{aligned}
\end{equation*}
for all $\rho \in (\frac{a}{4}, 4a)$ and any  $\varepsilon \in (0,\frac{1}{2})$. 
Thus, we find that 
\begin{equation*}
	\begin{aligned}
		\norm {\mathcal{D}_2(\cdot,\theta,\eta)}_{L^p(\mathbb{R}^+)} 
		&\lesssim \alpha \left( \left( \int_{\frac{a}{4}}^{4a}   d\rho \right)^\frac{1}{p} + \frac{1}{\sin^{2\varepsilon}(\frac{\theta-\eta}{2})} \sum_{i=1}^2\left( \int_{\frac{a}{4}}^{4a}   \frac{ \mathds{1}_{\{ |\rho- R_i(\eta)| \geq \alpha \} }  }{|\rho-R_i(\eta)|^{2p(1-\varepsilon)} }  d\rho \right)^\frac{1}{p}   \right)
		\\
		& \lesssim \alpha\left( 1 +  \frac{1}{\sin^{2\varepsilon}(\frac{\theta-\eta}{2})}  \left( \int_{\alpha}^{6a}   \frac{d\sigma}{\sigma ^{2p(1-\varepsilon)} }   \right)^\frac{1}{p}   \right)
		\\ 
		& \lesssim \alpha\left( 1 +  \frac{1}{\sin^{2\varepsilon}(\frac{\theta-\eta}{2})}  \alpha ^{\frac 1p-2(1-\varepsilon)} \right)
		\\
		& \lesssim \alpha ^{\beta} \left(1 +  \frac{1}{\sin^{2\varepsilon}(\frac{\theta-\eta}{2})} \right) ,
	\end{aligned}
\end{equation*}
where,  $\beta>0$ is introduced in \eqref{beta:def}. Thus,   it follows that 
\begin{equation*}
	\int_0^{2\pi} \norm{\mathcal{D}_2 (\cdot, \theta,\eta)}_{L^p(\mathbb{R}^+)} |f(\eta)| d\eta \lesssim \norm{R_1-R_2}_{L^\infty(\mathbb{T})}^{\beta} \norm {f}_{L^\infty(\mathbb{T})},
\end{equation*}
uniformly in $\theta\in \mathbb{T}$, whence, combined with \eqref{D1:bound}, we deduce that 
\begin{equation*} 
		\norm{\partial_\rho A_1 -\partial_\rho  A_2}_{ L^p \big ( ( \frac{a}{4},4a )\times \mathbb{T}\big)} \lesssim \norm{R_1-R_2}^{\beta}_{L^\infty(\mathbb{T})} \norm {f}_{L^\infty(\mathbb{T})}.
		\end{equation*}
		At last, putting this bound and \eqref{A1-A2:1} together concludes the proof of the lemma.
    \end{proof}

We conclude this section with the following simple lemma that establishes a continuity property of composed H\"older-functions. This will serve to simply the conclusion of the proof of Proposition \ref{proposition regularity of the functional}.
 
	\begin{lem}\label{lemma:composition}
		Let $m_1$ and $m_2$ be  two Lipschitz-continuous functions on $\mathbb{T}$. Further consider a function $\Phi \in C^{\kappa}_{\loc }(\mathbb{R}^2)$, for some $\kappa \in (0,1)$. Then, it holds that 
		\begin{equation*}
			\norm {\Phi (m_1) - \Phi(m_2)}_{C^\alpha(\mathbb{T})} \leq  C_{\Phi}  \norm{m_1-m_2}_{C^{1}(\mathbb{T})}^{\kappa-\alpha},
		\end{equation*}
		for any $\alpha \in [0,\kappa]$ and for some constant $C_\Phi>0$ depending on the $C^\kappa$-norm of $\Phi$.
	\end{lem}

\begin{proof}
	We only focus our attention on the estimate of the difference in $\dot C^{\kappa}$ as its control in $L^\infty$ easily follows from the fact that the supremum of continuous functions is attained on compact sets. Now,  we introduce the parameter
\begin{equation*}
	\delta\bydef\norm{m_1-m_2}_{C^{1+\alpha}} 
\end{equation*}
 and we write that
 \begin{align*}
 	\sup_{|\theta_1-\theta_2 |\leq1}&\frac{\left|\Big(\Phi\big(m_1(\theta_1)\big)-\Phi\big(m_2(\theta_1)\big)\Big)-\Big(\Phi\big(m_1(\theta_2)\big)-\Phi\big(m_2(\theta_2)\big)\Big)\right|}{|\theta_1-\theta_2|^\alpha}\\
 	&\quad\leq\sup_{|\theta_1-\theta_2|\leq\delta}\left(\frac{\left|\Phi\big(m_1(\theta_1)\big)-\Phi\big(m_1(\theta_2)\big)\right|}{|\theta_1-\theta_2|^\alpha}+\frac{\left|\Phi\big(m_2(\theta_1)\big)-\Phi\big(m_2(\theta_2)\big)\right|}{|\theta_1-\theta_2|^\alpha}\right)\\
 	&\qquad\qquad+\sup_{\delta\leq|\theta_1-\theta_2|\leq 1}\left(\frac{\left|\Phi\big(m_1(\theta_1)\big)-\Phi\big(m_2(\theta_1)\big)\right|}{|\theta_1-\theta_2|^\alpha}+\frac{\left|\Phi\big(m_1(\theta_2)\big)-\Phi\big(m_2(\theta_2)\big)\right|}{|\theta_1-\theta_2|^\alpha}\right).
 	\end{align*}
 Therefore, employing the assumption that $\Phi$ belongs to $C^{\kappa}(B_{\bar m})$, with 
 \begin{equation*}
 	\bar m \bydef \max \{  \norm {m_1}_{L^\infty(\mathbb{T})}, \norm {m_2}_{L^\infty(\mathbb{T})}\},
 \end{equation*}
  it follows that
   \begin{align*}
  	\sup_{|\theta_1-\theta_2|\leq 1}&\frac{\left|\Phi\big(m_1(\theta_1)\big)-\Phi\big(m_2(\theta_1)\big)-\Phi\big(m_1(\theta_2)\big)-\Phi\big(m_2(\theta_2)\big)\right|}{|\theta_1-\theta_2|^\alpha}\\
  	&\quad\leq \norm {\Phi}_{C^{\kappa}(B_{\bar m})}\sup_{|\theta_1-\theta_2|\leq\delta}\left(\frac{\left|m_1(\theta_1)-m_1(\theta_2)\right|^\kappa}{|\theta_1-\theta_2|^\alpha}+\frac{\left|m_2(\theta_1)-m_2(\theta_2)\right|^\kappa}{|\theta_1-\theta_2|^\alpha}\right)\\
  	&\qquad\qquad+\norm {\Phi}_{C^{\kappa}(B_{\bar m})} \sup_{\delta\leq|\theta_1-\theta_2|\leq1}\left(\frac{\norm{m_1-m_2}_{L^\infty(\mathbb{T})}^\kappa}{|\theta_1-\theta_2|^\alpha}\right).
 \end{align*}
 Hence, by  employing the Lipschitz-regularity assumption on the functions $m_1$ and $m_2$, we arrive at the bound
  \begin{equation*}
 	\begin{aligned}
 	       \sup_{|\theta_1-\theta_2|\leq1}&\frac{\left|\Big(\Phi\big(r_1(\theta_1)\big)-\Phi\big(r_2(\theta_1)\big)\Big)-\Big(\Phi\big(r_1(\theta_2)\big)-\Phi\big(r_2(\theta_2)\big)\Big)\right|}{|\theta_1-\theta_2|^\alpha} \leq 2 \norm {\Phi}_{C^{\kappa}(B_{\bar m})} \delta^{\kappa-\alpha},  
 	\end{aligned}
 \end{equation*}
 thereby completing the proof of the lemma.
\end{proof}

\subsubsection{Regularity of the functional  $\mathcal{F}$}
In Proposition \ref{proposition regularity of the functional} below,  we rigorously show that the functional $\mathcal{F}$ is differentiable in a well-chosen functional setting. This will allow us to compute the linearized operator and study its spectral properties, later on.

\begin{prop}\label{proposition regularity of the functional}
Let $\alpha\in(0,1)$ and $m\in\mathbb{N}^*.$ There exists $\varepsilon>0$ such that  the functional   	$$\mathcal{F}:\mathbb{R}\times \mathcal{B}^\alpha_{m,\varepsilon} \rightarrow Y_m^{\alpha}$$
	is well-defined and is of class  $C^{1}.$
	 Moreover, the partial derivative $\partial_{\Omega}d_r\mathcal{F}$ exists in the sense that
	$$\partial_{\Omega}d_r\mathcal{F}:\mathbb{R}\times \mathcal{B}^\alpha_{m,\varepsilon} \rightarrow\mathcal{L}(X_m^{\alpha},Y_m^{\alpha})$$
	 and is continuous. 
\end{prop}
 
 \begin{proof} 
 	The symmetry properties of the functional $\mathcal{F}$ are already established in Lemma \ref{lemma:symmetry-lin}, above. Hence, we only focus here on justifying the regularity and continuity of $\mathcal{F}$ as well as its derivatives. Recalling first the contour representation \eqref{contour:equation:2}, it is then readily seen that the claims of Proposition \ref{proposition regularity of the functional} above are reduced to showing that the mapping 
$$r\mapsto\frac{1}{b(R\big(\theta)\big)}\partial_\theta\left(\psi_b\big(R(\theta)e^{\ii \theta}\big)\right):\mathcal{B}^\alpha_{m,\varepsilon}\rightarrow C^{\alpha}(\mathbb{T})$$
is well-defined and is of class $C^1$, for some $\varepsilon>0$ and all $(m,\alpha)\in\mathbb{N}^*\times(0,1)$. Moreover, note, due to the assumptions on $b$ and $r$, that we can further restrict the justification of the preceding claims to the functional 
$$r\mapsto \partial_\theta\left(\psi_b\big(R(\theta)e^{\ii \theta}\big)\right):\mathcal{B}^\alpha_{m,\varepsilon}\rightarrow C^{\alpha}(\mathbb{T}).$$
  To that end, we first employ the decomposition from Lemma \ref{lemma:expansion1} to write that  
 \begin{equation*} 
	\partial_\theta\left(\psi_b\big(R(\theta)e^{\ii \theta}\big)\right)=-\frac{1}{2\pi}\underbrace{\partial_\theta\left(\int_{D_0}\log|R(\theta)e^{\ii \theta}-y|b^2(|y|)dy\right)}_{\bydef\mathcal{H}(r)(\theta)}+\underbrace{\partial_\theta\left(\phi_{R}\big(R(\theta)e^{\ii \theta}\big)\right)}_{\bydef\mathcal{K}(r)(\theta)},
	\end{equation*} 
 where the remainder $\phi_{R} \in W^{3,p}_{\loc}(\mathbb{R}^2)$ is   determined by \eqref{Remainder_remark}   and can be recast   in terms of the kernel $K_b$ as
 \begin{equation}\label{phi_D:expression}
 	\phi_{R}(x)= \frac{1}{2\pi } \int_{\mathbb{R}^2} K_b(x,z)\frac{\nabla b}{b^2}(z)\cdot\nabla_z \int_{D_0} \log|z-y| b(y)  dydz,
 \end{equation}
 for all $x\in\mathbb{R}^2.$ Next, we will  take care of each term separately. 
 
 \subsubsection*{Regularity of $\mathcal{H}$}
 By a straightforward computation, we find that 
 \begin{align*}
 	\mathcal{H}(r)(\theta)&=-\partial_\theta \left(R(\theta)e^{\ii \theta}\right)\cdot\int_{D_0}\nabla _y\Big(\log|R(\theta)e^{\ii \theta}-y|\Big)b^2(|y|)dy\\
 	&=\int_0^{2\pi}\log|R(\theta)e^{\ii \theta}-R(\eta)e^{\ii \eta}|b^2\big(R(\eta)\big)\partial_\eta\left(\ii R(\eta)e^{\ii \eta}\right)\cdot\partial_\theta\left(R(\theta)e^{\ii \theta}\right)d\eta\\
 	&\quad+2\partial_\theta\left(R(\theta)e^{\ii \theta}\right)\cdot\int_0^{2\pi}\int_0^{R(\eta)}\log|R(\theta)e^{\ii \theta}-\rho e^{\ii \eta}|(bb')(\rho)e^{\ii \eta}\rho d\rho d\eta,
 \end{align*}
where we have utilized Gauss--Green theorem in the second identity, above. Therefore, by further employing the identity
\begin{equation}\label{identity:SP}
	\partial_\eta\left(\ii f(\eta)e^{\ii \eta}\right)\cdot\partial_\theta\left(g(\theta)e^{\ii \theta}\right)=\partial_{\theta\eta}^2\big(f(\theta)g(\eta)\sin(\theta-\eta)\big),
\end{equation}
we arrive at the expansion
\begin{equation*}
	\mathcal{H}(r)(\theta)=\mathcal{H}_1(r)(\theta)+\mathcal{H}_2(r)(\theta),
 \end{equation*}
where 
\begin{align*}
	\mathcal{H}_1(r)(\theta)&\bydef\int_0^{2\pi}\log|R(\theta)e^{\ii \theta}-R(\eta)e^{\ii \eta}|b^2\big(R(\eta)\big)L(r)(\theta,\eta)d\eta,\\
	L(r)(\theta,\eta)&\bydef\partial_{\theta\eta}^2\big(R(\theta)R(\eta)\sin(\theta-\eta)\big)
\end{align*} 
and
\begin{equation*}
	  \mathcal{H}_2(r)(\theta)\bydef2\partial_\theta\left(R(\theta)e^{\ii \theta}\right)\cdot\int_0^{2\pi}\int_0^{R(\eta)}\log|R(\theta)e^{\ii \theta}-\rho e^{\ii \eta}|(bb')(\rho)e^{\ii \eta}\rho d\rho d\eta.
\end{equation*}
We now deal with each term in the preceding decompositions separately, and we will only prove their $C^1$ regularities, for the fact that the  functionals above are well defined can be easily justified by making use of the same arguments below. First, we compute that 
	\begin{align*}
		d_r\big(\mathcal{H}_1(r)\big)[h](\theta)&=\int_{0}^{2\pi}\nabla _x\log|R(\theta)e^{\ii \theta}-R(\eta)e^{\ii \eta}|\cdot\left(\frac{h(\theta)e^{\ii \theta}}{R(\theta)}-\frac{h(\eta)e^{\ii \eta}}{R(\eta)}\right)L(r)(\theta,\eta)b^2\big(R(\eta)\big)d\eta\\
		&\quad+\int_{0}^{2\pi}\log|R(\theta)e^{\ii \theta}-R(\eta)e^{\ii \eta}|\partial_{\theta\eta}^2\left(\frac{h(\eta)}{R(\eta)}R(\theta)\sin(\theta-\eta)\right)b^2\big(R(\eta)\big)d\eta\\
		&\quad+\partial_\theta\left(\frac{h(\theta)}{R(\theta)}e^{\ii \theta}\right)\cdot\int_{0}^{2\pi}\log|R(\theta)e^{\ii \theta}-R(\eta)e^{\ii \eta}|\partial_\eta\left(\ii R(\eta) e^{\ii \eta}\right)b^2\big(R(\eta)\big)d\eta\\
		&\quad+2\int_0^{2\pi}\log|R(\theta)e^{\ii \theta}-R(\eta)e^{\ii \eta}|L(r)(\theta,\eta)(bb')\big(R(\eta)\big)\frac{h(\eta)}{R(\eta)}d\eta,
	\end{align*}
 where we have used the identities
		\begin{equation*}
			d_r\big(L(r)\big)[h](\theta,\eta)=\partial_{\theta\eta}^2\left(\frac{h(\eta)}{R(\eta)}R(\theta)\sin(\theta-\eta)\right)+\partial_{\theta\eta}^2\left(\frac{h(\theta)}{R(\theta)}R(\eta)\sin(\theta-\eta)\right)
		\end{equation*}
 and, in view of \eqref{identity:SP}, that
 \begin{equation*}
 	\begin{aligned}
 		\partial_{\theta\eta}^2\left( \frac{h(\theta)}{R(\theta)}R(\eta)\sin(\theta-\eta)\right)
 		&=\partial_\eta\left(\ii R(\eta)e^{\ii \eta}\right)\cdot\partial_\theta\left(\frac{h(\theta)}{R(\theta)}e^{\ii \theta}\right).
 		 	\end{aligned}
 	\end{equation*}
 	Now, it is more convenient to utilize \eqref{identity:SP}, again, to write that 
 	\begin{equation*}
 		L(r)(\theta,\eta)=\partial_\eta\left(\ii R(\eta)e^{\ii \eta}\right)\cdot\partial_\theta\left(R(\theta)e^{\ii \theta}\right)
 	\end{equation*}
 	and then that
 	\begin{equation*}
 		d_r\big(\mathcal{H}_1(r)\big)[h](\theta)=g_{1,1}(\theta)\cdot\int_{0}^{2\pi}T_{1}(\theta,\eta)f_{1,1}(\eta)d\eta+\sum_{i=2}^4g_{1,i}(\theta)\cdot\int_{0}^{2\pi}T_{0}(\theta,\eta)f_{1,i}(\eta)d\eta,
 	\end{equation*} 
 	where, we set 
 	\begin{align*}
 		T_{1}(\theta,\eta)&\bydef\nabla_x\log|R(\theta)e^{\ii \theta}-R(\eta)e^{\ii \eta}|\cdot\left(\frac{h(\theta)e^{\ii \theta}}{R(\theta)}-\frac{h(\eta)e^{\ii \eta}}{R(\eta)}\right),\\
 		T_{0}(\theta,\eta)&\bydef \log|R(\theta)e^{\ii \theta}-R(\eta)e^{\ii \eta}|
 	\end{align*}
 	and
 	\begin{equation*}
 		\begin{array}{ll}
 			f_{1,1}(\eta)\bydef\partial_\eta\left(\ii R(\eta)e^{\ii \eta}\right)b^2\big(R(\eta)\big),\qquad 
 			&g_{1,1}(\theta)\bydef \partial_\theta \left(R(\theta)e^{\ii \theta}\right),\vspace{0.2cm}\\
 			f_{1,2}(\eta)\bydef\partial_\eta\left(\ii\frac{h(\eta)}{R(\eta)}e^{\ii \eta}\right)b^2\big(R(\eta)\big),
 			\qquad
 			&g_{1,2}(\theta)\bydef\partial_\theta\left(R(\theta)e^{\ii \theta}\right),\vspace{0.2cm}\\
 			f_{1,3}(\eta)\bydef\partial_\eta\left(\ii R(\eta)e^{\ii \eta}\right)b^2\big(R(\eta)\big), 
 			\qquad
 			&g_{1,3}(\theta)\bydef\partial_\theta\left(\frac{h(\theta)}{R(\theta)}e^{\ii \theta}\right),\vspace{0.2cm}\\
 			f_{1,4}(\eta)\bydef 2\frac{h(\eta)}{R(\eta)} \partial_\eta\left(\ii R(\eta)e^{\ii \eta}\right)(bb')\big(R(\eta)\big), 
 			\qquad
 			&g_{1,4}(\theta)\bydef\partial_\theta\left(R(\theta)e^{\ii \theta}\right).
 		\end{array}
 	\end{equation*} 	
Likewise, we compute that   
\begin{align*}
	d_r\big(\mathcal{H}_2(r)\big)[h](\theta)&=\partial_\theta\left(\frac{h(\theta)}{R(\theta)}e^{\ii \theta}\right)\cdot\int_{D_0}\log|R(\theta)e^{\ii \theta}-y|\nabla_y\left(b^2(|y|)\right)dy\\
	&\quad+\partial_\theta\left(R(\theta)e^{\ii \theta}\right)\cdot\int_{D_0}\left(\frac{h(\theta)e^{\ii \theta}}{R(\theta)}\cdot\nabla_x\Big[\log|R(\theta)e^{\ii \theta}-y|\Big]\right)\nabla_y\left(b^2(|y|)\right)dy\\
	&\quad+2\partial_\theta\left(R(\theta)e^{\ii \theta}\right)\cdot\int_0^{2\pi}\log|R(\theta)e^{\ii \theta}-R(\eta)e^{\ii \eta}|(b b')\big(R(\eta)\big)h(\eta)e^{\ii \eta}d\eta.
\end{align*}
Therefore, using that $\nabla_x\big(\log|x-y|\big)=-\nabla_y\big(\log|x-y|\big),$ we get
\begin{equation*}
	\begin{aligned}
		\partial_\theta\left(R(\theta)e^{\ii \theta}\right)\cdot\int_{D_0}&\left(\frac{h(\theta)e^{\ii \theta}}{R(\theta)}\cdot\nabla_x \Big(\log|R(\theta)e^{\ii \theta}-y|\Big)\right)\nabla_y\left(b^2(|y|)\right)dy\\
	&=-\left(\partial_\theta\left(R(\theta)e^{\ii \theta}\right)\otimes\frac{h(\theta)e^{\ii \theta}}{R(\theta)}\right):\int_{D_0}\nabla_y\left(b^2(|y|)\right)\otimes\nabla_y\Big(\log|R(\theta)e^{\ii \theta}-y|\Big)dy.
	\end{aligned}
\end{equation*}  
By further employing Green--Gauss theorem to write that 
	\begin{align*}
		&\partial_\theta\left(R(\theta)e^{\ii \theta}\right)\cdot\int_{D_0}\left(\frac{h(\theta)e^{\ii \theta}}{R(\theta)}\cdot\nabla_x\Big(\log|R(\theta)e^{\ii \theta}-y|\Big)\right)\nabla_y\left(b^2(|y|)\right)dy\\
		& \quad =2\left(\partial_\theta\left(R(\theta)e^{\ii \theta}\right)\otimes\frac{h(\theta)e^{\ii \theta}}{R(\theta)}\right):\int_0^{2\pi}\log|R(\theta)e^{\ii \theta}-R(\eta)e^{\ii \eta}|(bb')\big(R(\eta)\big)e^{\ii \eta}\otimes\partial_\eta\left(\ii R(\eta)e^{\ii \eta}\right)d\eta\\
		&\quad+\left(\partial_\theta\left(R(\theta)e^{\ii \theta}\right)\otimes\frac{h(\theta)e^{\ii \theta}}{R(\theta)}\right):\int_{D_0}\log|R(\theta)e^{\ii \theta}-y|\nabla_y^2\left(b^2(|y|)\right)dy,
	\end{align*}
it then follows  that 
 \begin{equation*}
 	d_r\big(\mathcal{H}_2(r)\big)[h](\theta)=\sum_{i=1}^2g_{2,i}(\theta)\cdot\int_{0}^{2\pi}T_{0}(\theta,\eta)f_{2,i}(\eta)d\eta+\Upsilon_1(\theta)+\Upsilon_2(\theta),
 \end{equation*}
 where, we set
 \begin{equation*}
 	\begin{array}{ll}
 		f_{2,1}(\eta)\bydef2(bb')\big(R(\eta)\big)h(\eta)e^{\ii \eta},\qquad & g_{2,1}(\theta)\bydef\partial_\theta\left(R(\theta)e^{\ii \theta}\right),\vspace{0.2cm}\\
 		f_{2,2}(\eta)\bydef2bb'\big(R(\eta)\big)e^{\ii \eta}\otimes\partial_\eta\left(\ii R(\eta)e^{\ii \eta}\right),\qquad & g_{2,2}(\theta)\bydef\partial_\theta\left(R(\theta) e^{\ii \theta}\right)\otimes\frac{h(\theta)e^{\ii \theta}}{R(\theta)}
 	\end{array}
 \end{equation*}
and 
 \begin{align*}
 	\Upsilon_1(\theta)&\bydef\partial_\theta\left(\frac{h(\theta)}{R(\theta)}e^{\ii \theta}\right)\cdot\int_{D_0}\log|R(\theta)e^{\ii \theta}-y|\nabla_y\left(b^2(|y|)\right)dy,\\
 	\Upsilon_2(\theta)&\bydef\left(\partial_\theta\left(R(\theta)e^{\ii \theta}\right)\otimes\frac{h(\theta)e^{\ii \theta}}{R(\theta)}\right):\int_{D_0}\log|R(\theta)e^{\ii \theta}-y|\nabla_y^2\left(b^2(|y|)\right)dy.
 \end{align*} 
 Now, assuming for a moment that 
 $$\Upsilon_1, \Upsilon_2 \in C^{\alpha}(\mathbb{T})$$
 and  noticing that $g_{j,k}$ and $f_{j,k}$ belong to $C^{\alpha}(\mathbb{T})$  and $L^\infty(\mathbb{T})$, respectively, for all possible values of the integer parameters $(j,k) $ above,  we conclude, by virtue of Corollaries \ref{corollary:kernel:1}, \ref{corollary:kernel:2} and \ref{corollary:kernel:3}, that 
 $$d_r\mathcal{H}\in\mathcal{L}(X_m^\alpha,Y_m^\alpha).$$
   Let us now show that $ \Upsilon_2$ belongs to $C^\alpha(\mathbb{T})$. 
   To that end, we first notice that 
 \begin{equation*}
 	  \sup_{\theta \in \mathbb{T}}\left|\int_{D_0}\log|R(\theta)e^{\ii \theta}-y|\nabla_y^2\left(b^2(|y|)\right)dy    \right| \lesssim  \sup_{x\in B_\sigma} \int_{D_0}\left|\log|x-y|\right| dy    \norm {\nabla^2\left(b^2 \right)}_{L^\infty(D_0)},
 \end{equation*}
   for any ball $B_\sigma$ whose radius satisfies 
   \begin{equation*}
   	\rho\geq \norm { R}_{L^\infty(\mathbb{T})}.
   \end{equation*} 
   Moreover, writing that 
   \begin{equation*}
   	\int_{D_0}\log|R(\theta)e^{\ii \theta}-y|\nabla_y^2\left(b^2(|y|)\right)dy = w\left( R(\theta)e^{\ii \theta}\right),
   \end{equation*}
   where it is seen that $w$ solves Poisson's equation 
   \begin{equation*}
   	\Delta w = \nabla ^2(b^2 )  \mathds{1}_{D_0},
   \end{equation*}
    it then follows that 
   \begin{equation*}
   	\norm {w\left( R(\cdot)e^{\ii \cdot}\right)}_{\dot C^\alpha(\mathbb{T})} \lesssim \norm {R}_{C^1(\mathbb{T})}
   	  \norm { w}_{\dot  C^\alpha(\mathbb{R}^2)} \lesssim \norm {R}_{C^1(\mathbb{T})} \norm { w}_{\dot  W^{1,p_\alpha}(\mathbb{R}^2)} ,
   \end{equation*}   
   where we set 
\begin{equation}\label{def palpha}
	p_\alpha\bydef\frac{2}{1-\alpha}\in(2,\infty). 
\end{equation} 
   Thus, by virtue of Proposition \ref{prop:Elliptic:00} and \eqref{Hb1:eq}, we deduce that 
   \begin{equation*}
   		\norm {w\left( R(\cdot)e^{\ii \cdot}\right)}_{\dot C^\alpha(\mathbb{T})} \lesssim   \norm {R}_{C^1(\mathbb{T})} \norm { \nabla ^2(b^2 ) }_{L^{p_\alpha}(D_0 )} <\infty.
   \end{equation*}  
All in all, combining these bounds with the assumption that $ r\in C^{1+\alpha}(\mathbb{T})$  yields   $ \Upsilon_2\in C^{\alpha}(\mathbb{T})$. In addition to that, we  emphasize that the same preceding  arguments apply to estimate $\Upsilon_1$, entailing     the same conclusion that $ \Upsilon_1\in C^{\alpha}(\mathbb{T})$, too.

 Now, we are left with justifying   the continuity of functional  $$r\mapsto d_r\mathcal{H}(r)$$ from $\mathcal{B}_{m,\varepsilon}^{\alpha}$ into $\mathcal{L}(X_m^\alpha,Y_m^\alpha) $, i.e., we claim that 
 \begin{equation*}
 	\sup_{\norm{h}_{C^{1+\alpha}}<1}\norm{d_r \mathcal{H}(r_1) [h]-d_r \mathcal{H}(r_2) [h]}_{C^{\alpha}}\to0,
 \end{equation*}
 as 
 \begin{equation*}
 	 \norm{r_1-r_2}_{C^{1+\alpha}}\to0.
 \end{equation*}
  For the sake of clarity, we keep the same notations for all the functionals introduced above and they will be further indexed, hereafter, by $r_1$ or $r_2$ to precise the dependence on the functions $r_1$ or $r_2$, respectively.
  
   Below, we will only focus on the analysis of non-classical terms $\Upsilon_1$ and $\Upsilon_2$, given by the expansion of $d_r\mathcal{H}_2$, above, for the treatment of the remaining terms follows by a direct application of Corollaries \ref{corollary:kernel:1}, \ref{corollary:kernel:2} and \ref{corollary:kernel:3}. Note that we can further restrict our proof below to $\Upsilon_2$, for the same analysis applies to $\Upsilon_1$, as well. To that end, we first write that 
 \begin{equation*}
 	\begin{aligned}
 		 \Upsilon_{2,r_1}(\theta)-\Upsilon_{2,r_2}(\theta)&=\big(\Pi_{r_1}(\theta)-\Pi_{r_2}(\theta)\big):J_{r_1}(\theta)+\Pi_{r_2}(\theta):\big(J_{r_1}(\theta)-J_{r_2}(\theta)\big),
 	\end{aligned}
 \end{equation*}
 where, we denote, for any $\theta\in \mathbb{T}$, that 
 \begin{equation*}
 	\Pi_r(\theta)\bydef\left(\partial_\theta\left(R(\theta)e^{\ii \theta}\right)\otimes\frac{h(\theta)e^{\ii \theta}}{R(\theta)}\right)
 \end{equation*}
 and 
 \begin{equation*}
 	J_r(\theta)\bydef\int_{D_{0,r}}\log|R(\theta)e^{\ii \theta}-y|\nabla_y^2\left(b^2(|y|)\right)dy.
 \end{equation*}
  Therefore, it is readily seen that $J_{r_1}\in C^{\alpha}(\mathbb{T})$ and that
 \begin{equation*}
 	\norm{\Pi_{r_1}-\Pi_{r_2}}_{C^{\alpha}}\lesssim\norm{h}_{C^{\alpha}}\norm{r_1-r_2}_{C^{1+\alpha}}.
 \end{equation*}
Now, we write, for any $\theta\in \mathbb{T}$, that 
 \begin{align*}
 	\big(J_{r_1}-J_{r_2}\big)(\theta)&=\int_{D_{0,r_1}}\left(\log|R_1(\theta)e^{\ii \theta}-y|-\log|R_2(\theta)e^{\ii \theta}-y|\right)\nabla_y^2\left(b^2(|y|)\right)dy\\
 	&\quad+\left(\int_{D_{0,r_1}}-\int_{D_{0,r_2}}\right)\log|R_2(\theta)e^{\ii \theta}-y|\nabla_y^2\left(b^2(|y|)\right)dy\\
 	&\bydef\Delta_1(\theta)+\Delta_2(\theta).
 \end{align*}
 Thus, by denoting  
 $$x_s(\theta)\bydef sR_1(\theta)+(1-s)R_2(\theta),$$
 and employing Taylor's expansion, we obtain, for any $\theta\in \mathbb{T}$, that
\begin{align*}
	\Delta_1(\theta)&=\big(R_1(\theta)-R_2(\theta)\big)e^{\ii \theta}\cdot\int_0^1\int_{D_{0,r_1}}\nabla_x\log|x_s(\theta)e^{\ii \theta}-y|\nabla_y^2\left(b^2(|y|)\right)dyds\\
	&=-2\pi\big(R_1(\theta)-R_2(\theta)\big)e^{\ii \theta}\cdot\int_0^1\Big(\nabla\Delta^{-1}\big(\mathds{1}_{D_{r_1}}\nabla^2(b^2)\big)\Big)\big(x_s(\theta)e^{\ii \theta}\big)ds.
\end{align*}
In the sequel of the proof, we will be using the bound 
\begin{equation*} 
	\sup_{\theta\in\mathbb{T}}R_k(\theta)\leq  2a   ,\qquad k\in\{1,2\},
\end{equation*}
which has been shown in \eqref{bnd inf}.
Thus, it follows, by the composition law
\begin{equation*}
	\begin{aligned}
		&\norm { \Big(\nabla\Delta^{-1} \big(\mathds{1}_{D_{r_1}}\nabla^2(b^2)\big)\Big)\big(x_s(\cdot)e^{\ii \cdot}\big)}_{C^\alpha(\mathbb{T})} 
		\\
		& \qquad \qquad \qquad \lesssim  \left(1+\norm{R_1}_{C^1(\mathbb{T})}+\norm{R_2}_{C^1(\mathbb{T})}\right)\norm{\nabla\Delta^{-1}(\mathds{1}_{D_{r_1}}\nabla^2(b^2))}_{C^{\alpha} (B_{2a} )},
	\end{aligned}
\end{equation*}
that 
\begin{align*}
	\norm{\Delta_1}_{C^\alpha(\mathbb{T})}&\lesssim\norm{r_1-r_2}_{C^\alpha(\mathbb{T})}\left(1+\norm{R_1}_{C^1(\mathbb{T})}+\norm{R_2}_{C^1(\mathbb{T})}\right)\norm{\nabla\Delta^{-1}(\mathds{1}_{D_{r_1}}\nabla^2(b^2))}_{C^{\alpha}(B_{2a})}.
\end{align*}
Now, as the function 
\begin{equation*}
     v \bydef  \Delta^{-1}(\mathds{1}_{D_{r_1}}\nabla^2(b^2))
\end{equation*}
is a solution of Poisson's equation
\begin{equation*}
	\Delta v= \mathds{1}_{D_{r_1}}\nabla^2(b^2),
\end{equation*}
the bounds
\begin{equation*}
	\norm {\nabla v }_{L^\infty(\mathbb{R}^2)} \lesssim \norm { \nabla ^2 (b^2)}_{L^{p_\alpha}( D_{r_1})}
\end{equation*}
and
\begin{equation*}
	\norm {\nabla v }_{\dot{C}^\alpha(\mathbb{R}^2)} 
	\lesssim 
	\norm { v }_{\dot W^{2,p_\alpha}(\mathbb{R}^2)}
	 \lesssim \norm { \nabla ^2 (b^2)}_{L^{p_\alpha}( D_{r_1})}
\end{equation*} 
follow by a direct application of the Lemma \ref{lemma:lipschitz} and Proposition \ref{prop:Elliptic:00}, thanks to the assumptions \eqref{Hb1:eq} and \eqref{Hb3:eq} on the depth function $b$, where $p_\alpha$ is defined in \eqref{def palpha}. 
Hence, we deduce that  
\begin{align*}
	\norm{\Delta_1}_{C^\alpha(\mathbb{T})} 	&\lesssim\norm{r_1-r_2}_{C^\alpha(\mathbb{T})}\norm{\nabla^2(b^2)}_{L^{p_\alpha}(D_{r_1})},
\end{align*}  
which takes care of the estimate on $\Delta_1$.
As for $\Delta_2$, we first perform a change of variables in polar coordinates to rewrite it, for any $\theta\in\mathbb{T}$, as 
\begin{align*}
	\Delta_2(\theta)&=\int_0^{2\pi}\int_0^1\log|R_2(\theta)e^{\ii \theta}-sR_1(\eta)e^{\ii \eta}|u_1(s,\eta)dsd\eta\\
    &\quad-\int_0^{2\pi}\int_0^1\log|R_2(\theta)e^{\ii \theta}-sR_2(\eta)e^{\ii \eta}|u_2(s,\eta)dsd\eta,
\end{align*}
where we denote, for $k\in\{1,2\}$, that
$$u_{r_k}(s,\eta)\bydef  s R^2_k(\eta) \left(  \nabla_y ^2\Big(b^2\big(|y|\big)\Big)\right)\Big|_{y=sR_k(\eta)  }.$$
Therefore, it follows that 
\begin{align*}
	\Delta_2(\theta)&=\int_0^{2\pi}\int_0^1\left(\log|R_2(\theta)e^{\ii \theta}-sR_1(\eta)e^{\ii \eta}|-\log|R_2(\theta)e^{\ii \theta}-sR_2(\eta)e^{\ii \eta}|\right)u_{r_2}(s,\eta)dsd\eta\\
    &\quad+\int_0^{2\pi}\int_0^1\log|R_1(\theta)e^{\ii \theta}-sR_1(\eta)e^{\ii \eta}|\left(u_{r_1}-u_{r_2}\right)(s,\eta)dsd\eta.
\end{align*}
Hence, by virtue of Corollaries \ref{corollary:kernel:1} and \ref{corollary:kernel:2}, we infer, for all $\beta\in (0,\alpha)$, that 
\begin{equation*}
	\norm{\Delta_2}_{C^\alpha}\lesssim\norm{r_1-r_2}_{C^{1}}^\beta\norm{u_{r_2}}_{L^\infty([0,1]\times \mathbb{R}^2)}+\norm{u_{r_1}-u_{r_2}}_{L^\infty([0,1]\times \bar B_{2a})}.
\end{equation*} 
Thus, as the function $u_{r_1} - u_{r_2}$ is continuous, whence its maximum on the closed set $[0,1]\times \bar B_{2a}$ is attained, we deduce that 
\begin{equation*}
	\norm{\Delta_2}_{C^\alpha}\lesssim\norm{r_1-r_2}_{C^{1}}^\beta\norm{u_{r_2}}_{L^\infty(\mathbb{R}^2)}+|(u_{r_1}-u_{r_2})(s_0,\eta_0)|,
\end{equation*}
for some $s_0,\eta_0\in [0,1]\times \bar B_{2a}$. Hence, due to the preceding point-wise bound and assumption \eqref{Hb1:eq}, it is readily seen that 
  \begin{equation*}
	\lim_{\norm{r_1-r_2}_{C^{1+\alpha}}\to0}\norm{\Delta_2}_{C^\alpha} =0.
\end{equation*} 
All in all, combining all the preceding bounds, we finally deduce that 
 \begin{equation*}
 \lim_{\norm{r_1-r_2}_{C^{1+\alpha}}\to0} \left( 	\sup_{\norm{h}_{C^{1+\alpha}}<1}\norm{\Upsilon_{2,r_1}-\Upsilon_{2,r_2}}_{C^{\alpha}}\right) =0,
 \end{equation*}
 which concludes the proof of the desired regularity properties for $\mathcal{H}$.

 \subsubsection*{Regularity of $\mathcal{K}$} The desired regularity of that term stems from the high-regularity of the perturbation $\phi$, which is the crux of the new decomposition in Lemma \ref{lemma:expansion1}. This is a generic property and can be observed in other settings. We proceed first by writing, for any $\theta\in\mathbb{T}$, that 
 \begin{equation*}
 	\mathcal{K}(r)(\theta)=\partial_{\theta}\left(R(\theta)e^{\ii \theta} \right)\cdot\nabla\phi_{R}\left(R(\theta)e^{\ii \theta}\right).
 \end{equation*}
 Therefore, by virtue of the fact that $\phi_R\in W^{3,p}_\loc(\mathbb{R}^2)$, for all $p\in (1,\infty)$, it is then readily seen that 
 $$\theta\mapsto\mathcal{K}(r)(\theta)\in C^{\alpha}(\mathbb{T}),$$
 as soon as $r\in C^{1+\alpha}(\mathbb{T})$. Moreover, we compute, for any $h\in C^{1+\alpha}(\mathbb{T})$, that 
 \begin{align*}
 	d_r\big(\mathcal{K}(r)\big)[h](\theta)&=\partial_{\theta}\left(\frac{h(\theta)}{R(\theta)}e^{\ii \theta}\right)\cdot\nabla\phi_{R}\left(R(\theta)e^{\ii \theta}\right)\\
 	&\quad+\left(\partial_{\theta}\left(R(\theta)e^{\ii \theta}\right)\otimes\frac{h(\theta)}{R(\theta)}e^{\ii \theta}\right):\nabla^2\phi_{R}\left(R(\theta)e^{\ii \theta}\right)
 	\\
 	& \quad +    \partial_{\theta}\left(\frac{h(\theta)}{R(\theta)}e^{\ii \theta}\right)\cdot\nabla  \mathcal{E}_{R}\left(R(\theta)e^{\ii \theta}\right),
 \end{align*}
 where the last term above comes from differentiating  the expression \eqref{phi_D:expression} with respect to $R$ (inside the integral) and is given by 
 \begin{equation} \label{E:DEF}
 	\mathcal{E}_R(x)\bydef   \int_{\mathbb{T}} \int_0^{ \infty} K_b(x,\rho e^{\ii \theta})\frac{ b'}{b^2}(\rho )\partial_\rho \left(  \int_{0}^{2\pi } \log|\rho e^{\ii \theta}- R(\eta)e^{\ii \eta}| b\big (R(\eta)\big )  d\eta  \right) \rho d\rho d\theta,
 \end{equation}
 for all $x\in \mathbb{R}^2$. In particular, by virtue of Lemma \ref{lemma:A:1} and the   assumption \eqref{Hb3:eq} on the function $b$, one sees that 
 \begin{equation*}
 	(r,\theta)\mapsto \frac{ b'}{b^2}(\rho )\partial_\rho \left(  \int_{0}^{2\pi } \log|\rho e^{\ii \theta}- R(\eta)e^{\ii \eta}| b\big (R(\eta)\big )  d\eta  \right) \in L^p(\mathbb{R}^+\times \mathbb{T}),
 \end{equation*} 
 for all $p\in [1,\infty)$. Thus, by means of Proposition \ref{prop:Elliptic:00} (see also Remark \ref{rem:ell:bounds}), we deduce that 
 \begin{equation*}
 	\nabla \mathcal{E}_R\in W^{1,p_\alpha}_{\loc } (\mathbb{R}^2)\hookrightarrow C_{\loc }^{\alpha}(\mathbb{R}^2)
 \end{equation*}
 with $p_{\alpha}$ being defined in \eqref{def palpha}.

 Hence, by a similar argument, relying on the $W^{2,p_\alpha}_{\loc }$-regularity of   $\nabla \phi_{R}$,  which is  again a consequence of   Remark \ref{rem:reg:remainder},    together with the symmetry properties established in Lemma \ref{lemma:symmetry-lin}, we deduce that 
 $$d_r\mathcal{K}\in\mathcal{L}(X_m^\alpha,Y_m^\alpha).$$ 
 
 We are now left with proof of continuity of $d_r\mathcal{K}$  with respect to the $r$-variable.  To that end, we first write by means of a direct computation  that 
 \begin{equation*}
 	\begin{aligned}
 		\norm{d_r\mathcal{K} (r_1)[h]-d_r\mathcal{K}(r_2)[h]}_{C^{\alpha}}&\leq\norm{R_1}_{C^{1}}\norm{(\nabla\phi_{R_1}, \nabla \mathcal{E}_{R_1})}_{C ^{\alpha}(B_{2a})}\norm{\partial_{\theta}\left(\frac{h}{R_1}e^{\ii \cdot}\right)-\partial_{\theta}\left(\frac{h}{R_2}e^{\ii \cdot}\right)}_{C^{\alpha}}
 		\\
 		&+\norm{\nabla\phi_{R_1}\big(R_1e^{\ii \cdot}\big)-\nabla\phi_{R_2}\big(R_2e^{\ii \cdot}\big)}_{C^{\alpha}}\norm{\partial_{\theta}\left(\frac{h}{R_2}e^{\ii \cdot}\right)}_{C^{\alpha}}
 		\\
 		&+\norm{\partial_{\theta}\left(R_1e^{\ii \cdot}\right)\otimes\frac{h}{R_1}-\partial_{\theta}\left(R_2e^{\ii \cdot}\right)\otimes\frac{h}{R_2}}_{C^{\alpha}}\norm{R_1}_{C^1}\norm{\nabla^2\phi_{R_1}}_{C ^\alpha(B_{2a})}
 		\\
 		&+\norm{\nabla\mathcal{E}_{R_1}\big(R_1e^{\ii \cdot}\big)-\nabla\mathcal{E}_{R_2}\big(R_2e^{\ii \cdot}\big)}_{C^{\alpha}}\norm{\partial_{\theta}\left(\frac{h}{R_2}e^{\ii \cdot}\right)}_{C^{\alpha}}
 		\\
 		&+\norm{\partial_{\theta}\left(R_2e^{\ii \cdot}\right)\otimes\frac{h}{R_2}}_{C^{\alpha}} \norm{\nabla^2\phi_{R_1}\left(R_1e^{\ii \cdot}\right)-\nabla^2\phi_{R_2}\left(R_2e^{\ii \cdot}\right)}_{C^{\alpha}}.
 	\end{aligned}
 \end{equation*} 
 The most challenging terms in this collection of bounds are  
 \begin{equation}\label{last:ES:A}
 	\begin{aligned}
 		\norm{\nabla\mathcal{E}_{R_1}\big(R_1e^{\ii \cdot}\big)-\nabla\mathcal{E}_{R_2}\big(R_2e^{\ii \cdot}\big)}_{C^{\alpha}}
 		& \lesssim \norm {R_1}_{C^1} \norm{\nabla\mathcal{E}_{R_1} -\nabla\mathcal{E}_{R_2} }_{C^{\alpha}(B_{2a})}
 		\\
 		& \quad  + \norm{\nabla\mathcal{E}_{R_2}\big(R_1e^{\ii \cdot}\big)-\nabla\mathcal{E}_{R_2}\big(R_2e^{\ii \cdot}\big)}_{C^{\alpha}}
 	\end{aligned}
 \end{equation}
 and 
  \begin{equation}\label{last:ES:B}
 	 \begin{aligned}
  \norm{\nabla^2\phi_{R_1}\left(R_1e^{\ii \cdot}\right)-\nabla^2\phi_{R_2}\left(R_2e^{\ii \cdot}\right)}_{C^{\alpha}} 
 	 	& \lesssim  \norm {R_1}_{C^1}\norm{\nabla^2\phi_{R_1} -\nabla^2\phi_{R_2} }_{C^{\alpha}(B_{2a})} 
 	 	\\
 	 	& \quad + \norm{\nabla^2\phi_{R_2}\left(R_1e^{\ii \cdot}\right)-\nabla^2\phi_{R_2}\left(R_2e^{\ii \cdot}\right)}_{C^{\alpha}} ,
 	 \end{aligned}
 \end{equation} 
 whereas the remaining ones can be estimated in a similar, or even easier, way. Notably, we already have prepared the ground to estimate them due to the analysis laid out in the previous lemmas. In particular,  noticing that 
 \begin{equation*}
 	\nabla\mathcal{E}_{R_2} , \nabla^2\phi_{R_2} \in W^{1,p_\kappa}_{\loc }(\mathbb{R}^2) \hookrightarrow C^{\kappa}_{\loc }(\mathbb{R}^2),
 \end{equation*}
 for $\kappa\in (0,1)$ being defined as 
 \begin{equation*}
 	\kappa\bydef1-\frac{2}{q_\kappa}>1-\frac{2}{p_\alpha}=\alpha,
 \end{equation*}
 it then follows by a direct application of Lemma \ref{lemma:composition} that 
 \begin{equation*}
 	 \norm{\nabla^2\phi_{R_2}\left(R_1e^{\ii \cdot}\right)-\nabla^2\phi_{R_2}\left(R_2e^{\ii \cdot}\right)}_{C^{\alpha}}  \lesssim \norm {R_1-R_2}_{C^1}^{\kappa-\alpha}
 \end{equation*}
 and
 \begin{equation*}
 	 \norm{\nabla\mathcal{E} _{R_2}\left(R_1e^{\ii \cdot}\right)-\nabla \mathcal{E}_{R_2}\left(R_2e^{\ii \cdot}\right)}_{C^{\alpha}}  \lesssim \norm {R_1-R_2}_{C^1}^{\kappa-\alpha}.
 \end{equation*}
 This takes care of the last terms in the right-hand sides of \eqref{last:ES:A} and \eqref{last:ES:B}. On the other hand, the first terms in the right-hand sides of these inequalities will be controlled by using a different argument. More precisely, owing to the definition \eqref{E:DEF}, employing Sobolev embedding and the elliptic-gain of regularity (see Proposition \ref{prop:Elliptic:00} and Remark \ref{rem:ell:bounds}), we infer that  
 \begin{equation*}
 	\begin{aligned}
 		\norm{\nabla\mathcal{E}_{R_1} -\nabla\mathcal{E}_{R_1} }_{C^{\alpha}(B_{2a})}
 		& \lesssim \norm {\mathcal{E}_{R_1} -\mathcal{E}_{R_1}}_{\dot W^{2,p_\alpha}(\mathbb{R}^2)}
 		\\
 		& \lesssim \norm {\frac{ b'}{b^2}(\rho )\partial_\rho ( \mathcal{A}_1-\mathcal{A}_2)(\rho,\theta) }_{L^p(\rho d\rho d\theta ; (0,R_\infty) \times \mathbb{T})},
 	\end{aligned}
 \end{equation*}
 where $R_\infty$ is the constant introduced in assumption \eqref{Hb3:eq} and with 
 \begin{equation*}
 	\mathcal{A}_i(\rho, \theta) \bydef   \int_{\mathbb{T}} \log|\rho e^{\ii \theta}- R_i(\eta)e^{\ii \eta}| b\big (R(\eta)\big )d\eta .
 \end{equation*}
 Thus, by a direct application of Lemma \ref{lemma:A:2}, it is readily seen that 
 \begin{equation*}
 	\norm{\nabla\mathcal{E}_{R_1} -\nabla\mathcal{E}_{R_1} }_{C^{\alpha}(B_{2a})}
 	 \lesssim 
 	 \norm {R_1-R_2}_{L^\infty}^\beta,
 \end{equation*}
 for any  $\beta\in [0, \frac{1}{p_\alpha})$. 
 
 At last, it only remains to estimate the first term in \eqref{last:ES:B}. To that end, we notice that the representation \eqref{phi_D:expression} entails that 
 \begin{equation*} 
 	\phi_{R_1}- \phi_{R_2}= \phi_{\widetilde{R}} (x), \end{equation*}
 where $\phi_{\widetilde{R}}$  is given, for any $x\in \mathbb{R}^2$, by
 \begin{equation*}
 	\phi_{\widetilde{R}} (x)\bydef \frac{1}{2\pi } \int_{\mathbb{R}^2} K_b(x,z)\frac{\nabla b}{b^2}(z)\cdot\nabla_z \int_{\mathbb{R}^2} \log|z-y| b(y) \left(  \mathds{1}_{\{y\in D_1\}}-\mathds{1}_{\{y\in D_2\}} \right) dydz ,
 \end{equation*}
 which represents the  remainder   term in the decomposition given in Lemma \ref{lemma:expansion1} of the solution of  \eqref{EL:P1} with the new source term 
 $$f = \mathds{1}_{D_1}-\mathds{1}_{D_2}.$$  
 Therefore, by virtue of Sobolev embedding and the elliptic-gain of regularity (see in particular Remark \ref{rem:reg:remainder}) we deduce the bound 
\begin{equation*}
	\norm{\nabla^2\phi_{R_1} -\nabla^2\phi_{R_2} }_{C^{\alpha}(B_{2a})} \lesssim  \norm{ \phi_{\widetilde{R}} }_{ \dot W^{3,p_\alpha}(\mathbb{R}^2)} \lesssim \norm {\mathds{1}_{D_1}-\mathds{1}_{D_2}}_{L^{p_\alpha}(\mathbb{R}^2)},
\end{equation*}
 which in turn implies that 
 \begin{equation*}
 	\lim_{\norm {R_1-R_2}_{L^\infty}\to 0}\norm{\nabla^2\phi_{R_1} -\nabla^2\phi_{R_2} }_{C^{\alpha}(B_{2a})}  = 0.
 \end{equation*}
 All in all, combining the previous estimates concludes  the justification of the desired regularity properties for $\mathcal{K}$ and completes the proof of the proposition.  
 \end{proof}

 \subsection{Proof of Theorem \ref{thm:1}}
 
We are now in position to employ all the preceding  results above and from the previous sections to finally establish all the prerequisites to applying   Crandall-Rabinowitz's Theorem \ref{Crandall-Rabinowitz theorem} and, subsequently, deduce Theorem \ref{thm:1}. This is achieved by the following  proposition.
\begin{prop}\label{prop.last} Let $\alpha\in (0,1)$, $a>0$, $M(a,b)$ be defined through \eqref{def M ab} and $\varepsilon>0$ be determined by Proposition \ref{proposition regularity of the functional}.     Then the following assertions hold for all integers  $m >M(a,b)$:
\begin{enumerate}
\item The linearized  operator $$d_r \mathcal{F}(\Omega,0):  X_m ^\alpha \to Y_m ^\alpha $$ has a 
non trivial kernel  if and only if
\begin{equation*}
	\Omega\in\left\{\Omega_{nm} \bydef  Q(a,0)-\Lambda_{nm}(a,a),\quad n\in\mathbb{N}^*\right\},
\end{equation*} 
where $Q(a,0)$ and $\big(\Lambda_{n}(a,a)\big)_{n\in\mathbb{N}^*}$ are defined by \eqref{def Q} and \eqref{def Lambdan}, respectively.
\item Moreover, $\ker \big(d_r\mathcal{F}(\Omega_{m} ,0)\big)$  is a one-dimensional vector space   generated by $$\theta\mapsto\cos(m\theta).$$
\item Furthermore, the range of $d_r\mathcal{F}(\Omega _{m} ,0)$ is closed and of co-dimension one.
\item As for the transversality condition, we have that  
 $$\partial_\Omega d_r\mathcal{F}(\Omega _{m} ,0)\big[\cos(m\cdot)\big]\notin \textnormal{Range}\big(d_r\mathcal{F}(\Omega_{m} ,0)\big).$$
\end{enumerate}
\end{prop} 

\begin{proof}  
\textit{(1)} Let $h$ be a function in $X_m^\alpha$ given, for all $\theta\in\mathbb{T}$, by the expansion 
\begin{equation*}
	h(\theta)=\sum_{n=1}^{\infty}h_{n} \cos(nm\theta), 
\end{equation*}
for some scalars $h_n\in\mathbb{R}$.
We also recall, in view of Lemma \ref{lemma:linearized-OP}, that 
$$d_r\mathcal{F}(\Omega,0)[h](\theta)=-\sum_{n=1}^{\infty}nmq_{nm}(\Omega,a)h_{n}\sin(nm\theta),$$
where, for any $n\in\mathbb{N}^*$, we define
$$q_{nm}(\Omega,a)\bydef\Omega  -\Omega_{nm}.$$
Observe that it is readily seen that 
$$q_{nm}(\Omega,a)=0\qquad \text{if and only if}\qquad\Omega=\Omega_{nm} .$$
Hence, it follows that the kernel of $d_r \mathcal{F}(\Omega_{nm} ,0)$ is non trivial.\\
\textit{(2)} From the previous point, we know that the kernel of $d_r \mathcal{F}(\Omega_{m} ,0)$ is non trivial. To show that it is one-dimensional, it is sufficient to prove, for any integer $n\geq 2$, that 
$$ \Omega_{m} \neq\Omega_{nm} .$$
This fact is a consequence of the strict monotonicity of the sequence $\big(\Lambda_{n}(a,a)\big)_{n\in\mathbb{N}^*}$, previously proved in Proposition \ref{lemma:eigenfunction}-\textit{(2)}. Consequently, we deduce that 
$$\dim\Big(\ker\big(d_r \mathcal{F}(\Omega_{m} ,0)\big)\Big)=1$$
and that an obvious generator of the kernel is the function 
$$\theta\mapsto\cos(m\theta).$$
\textit{(3)} Now, we show that the range of $d_r\mathcal{F}(\Omega_{m}(a),0)$ coincides with the space
 \begin{equation*}
 	Z_{m}^\alpha\bydef\left\{g\in C^{\alpha}(\mathbb{T}): g(\theta)= \sum_{n=2}^{\infty}g_n\sin(nm\theta),\ g_n\in\mathbb{R},\ \theta\in\mathbb{T} \right\},
 \end{equation*}
which is clearly a closed sub-space of $Y_m^\alpha$ with co-dimension one. This will take care of the third claim in the statement of the Proposition \ref{prop.last}, above. To see that, we emphasize first that  
$$\textnormal{Range}\big(d_r\mathcal{F}(\Omega _{m} ,0)\big)\subset Z_{m}^\alpha.$$
Hence, it only remains to justify the converse of the preceding inclusion. To that end, we consider an element $g\in Z_{m}^\alpha$ and taking the following form  
\begin{equation*}
	g(\theta)=\sum_{n=2}^{\infty}g_n \sin(nm\theta),\quad \text{for all} \quad \theta\in\mathbb{T}.
\end{equation*}
 Then, observe that the strict monotonicity of $\big(\Lambda_{nm}(a,a)\big)_{n\in\mathbb{N}^*}$ in Proposition \ref{lemma:eigenfunction}-\textit{(4)} implies that 
 \begin{equation}\label{supq}
 	 q_{nm}(\Omega_{m} ,a)= \Lambda_{nm}(a,a)-\Lambda_{m}(a,a) <\Lambda_{2m}(a,a)-\Lambda_{m}(a,a) <0,
 \end{equation}
for all $n\geq 2.$ Next, we introduce the function $h$ by setting, for all $\theta\in \mathbb{T}$, that 
\begin{equation*}
	h(\theta)=\sum_{n=2}^{\infty}h_n \cos(nm\theta),\qquad h_n\bydef 
		\frac{-g_n}{nmq_{nm}(\Omega_{m} ,a)}, \quad  \textnormal{for all} \quad  n\geq2. 
\end{equation*}
  Clearly, we formally have $d_r\mathcal{F}(\Omega_{m}(a),0)[h]=g$ and to make this construction rigorous it only remains to show that $h\in C^{1+\alpha}(\mathbb{T})$. To that end,  notice first  that using \eqref{supq} and applying Cauchy-Schwarz and Bessel inequalities, we get
  \begin{align*}
  	\norm{h}_{L^{\infty}}&\lesssim\sum_{n=2}^{\infty}\frac{|g_n|}{n} \lesssim\|g\|_{L^2} \lesssim\|g\|_{C^{\alpha}}.
  \end{align*}
  Hence, we are left with proving that $h'\in C^{\alpha}(\mathbb{T}).$ To prove that, we first notice that 
  $$h'(\theta)=\sum_{n=2}^{\infty}\frac{g_n}{q_{nm}(\Omega_{m}(a),a)}\sin(nm\theta).$$
  Then, let us define the sequence $\big(\widetilde{q}_n(a)\big)_{n\in \mathbb{N}\setminus\{0,1\}}$ by
\begin{equation*}
	\widetilde{q}_{n}(a)\bydef   \frac{1}{q_{nm}(\Omega_{m}(a),a)}-\frac{1}{\Lambda_{m}(a,a)}=\frac{\Lambda_{nm}(a,a)}{\Lambda_{m} (a,a)[\Lambda_{nm}(a,a)-\Lambda_{m}(a,a)]}. 
\end{equation*} 
Observe that the preceding sequence is well-defined thanks to \eqref{supq} and Proposition \ref{lemma:eigenfunction}-\textit{(3)} which implies, since $m>M(a,b),$    that
$$\Lambda_{m}(a,a)>0.$$
Then, noting, due to \eqref{exp Lambdan} and \eqref{bnd fn_prop}, that 
$$\widetilde{q}_{n}(a)\underset{n\to\infty}{=}O\left(\frac{1}{n}\right),$$
it then follows that 
\begin{equation*}
	\theta\mapsto H(\theta)\bydef\sum_{n=2}^{\infty}\widetilde{q}_n(a)\sin(nm\theta)\in L^2(\mathbb{T})\subset L^1(\mathbb{T}).
\end{equation*}
Therefore, writing   
\begin{align*}
	h'(\theta)&=\sum_{n=2}^{\infty}\widetilde{q}_n(a)g_n\sin(nm\theta)+\frac{1}{\Lambda_{m}(a,a)}\sum_{n=2}^{\infty}g_n \sin(nm\theta)\\
	&=(H\ast g)(\theta)+\frac{1}{\Lambda_{m}(a,a)}g(\theta),
\end{align*}
then using the convolution property $L^1\ast C^{\alpha}\rightarrow C^{\alpha}$ and the fact that $g\in C^{\alpha} $, we conclude that the right-hand side above belongs to $C^{\alpha}$.\\
\textit{(4)} At last, the transversality condition follows from the simple fact that 
\begin{equation*}
	 \partial_\Omega d_r\mathcal{F}(\Omega _{m},0)\big[\cos(m\theta)\big]=-m\sin(m\theta),
\end{equation*}
for all $   \theta\in\mathbb{T} $. It is then obvious that the function in the right-hand side above does not belong to $Z_{m}^\alpha$. This concludes the proof of Proposition \ref{prop.last}.
\end{proof}

 \begin{proof}[Proof of Theorem \ref{thm:1}]
 	The proof of Theorem \ref{thm:1} is achieved by employing the results of Proposition \ref{prop.last} and applying Crandall--Rabinowitz Theorem \ref{Crandall-Rabinowitz theorem}. 
 \end{proof}


\section{Doubly-connected time-periodic solutions}\label{sec DC}
\subsection{The contour dynamics equation and linearization} 
The scope of this section is to prove  Theorem \ref{thm:2}. Throughout the rest of the paper, we  fix two real numbers $0<a_2<a_1$ and we look for a  solution   of \eqref{Lake1}
in the form  
$$\omega (t,\cdot)=b\mathds{1}_{D_{1,t}\setminus{\overline{D}_{2,t}}}.$$
for $t\geq 0$, where   the domain $D_{k,t}$, for   $k\in\{1,2\},$ is close to the disc $\{x\in\mathbb{R}^2: |x|<a_k\}$ in a sense to be precised in a moment. To that end, considering a parametrization  
$$z_k(t,\cdot):  \mathbb{T}\rightarrow\partial D_{k,t}$$   of each   boundary $\partial D_{k,t}$,  it follows by the same arguments leading to \eqref{EQ:z} that the dynamics of interfaces are given by the system
$$\partial_{t}z_{k}(t,\theta)\cdot\mathbf{n}\big(t,z_k(t,\theta)\big)+\frac{1}{b\big(z_{k}(t,\theta)\big)}\partial_{\theta}\Big(\psi_{b}\big(t,z_{k}(t,\theta)\big)\Big)=0, \quad \text{for}\quad k\in\{1,2\},$$
where the stream function $\psi_b$ is defined in \eqref{def psib}.
As we are interested in constructing V-states that are close to annuli, we now consider the following ansatz, for $k \in \{1,2\}$ 
$$  z_k(t,\theta)\bydef e^{\mathrm{i}\Omega t}z_k(\theta),\qquad z_k(\theta)\bydef R_{k}(\theta)e^{\mathrm{i}\theta},\qquad R_k(\theta)\bydef\sqrt{a_k^2+2r_k(\theta)},$$
which describes a rigid rotation in polar coordinates, where $(z_1(\theta),z_2(\theta))_{\theta\in \mathbb{T}}$ is a parametrization of the initial patch
\begin{equation*}
	\mathds{1}_{D_{1}\setminus{\overline{D}_{2}}}\bydef \mathds{1}_{D_{1,0}\setminus{\overline{D}_{2,0}}} .
\end{equation*}
Therefore, the equations of motions write the following system
\begin{equation}\label{system lake stationary}
	\begin{cases}
		\Omega\,r_1'(\theta)+\frac{1}{b\big(R_1(\theta)\big)}\partial_{\theta}\Big(\psi_{b}\big(z_1(\theta)\big)\Big)=0,\\
		\Omega\,r_2'(\theta)+\frac{1}{b\big(R_2(\theta)\big)}\partial_{\theta}\Big(\psi_{b}\big(z_2(\theta)\big)\Big)=0.
	\end{cases}
\end{equation}
 In view of the integral representation \eqref{integral rep psib}, we now compute, using polar coordinates, that 
 \begin{align*}
	\psi_b\big(z_k(\theta)\big)&=\int_{D_1}K_b\big(R_k(\theta)e^{\mathrm{i}\theta},y\big)b(y)dy-\int_{D_2}K_b\big(R_k(\theta)e^{\mathrm{i}\theta},y\big)b(y)dy\\
	&=\int_{0}^{2\pi}\int_{0}^{R_1(\eta)}K_b\big(R_k(\theta)e^{\mathrm{i}\theta},\rho e^{\mathrm{i}\eta}\big)b(\rho)\rho d\rho d\eta\\
	&\quad-\int_{0}^{2\pi}\int_{0}^{R_2(\eta)}K_b\big(R_k(\theta)e^{\mathrm{i}\theta},\rho e^{\mathrm{i}\eta}\big)b(\rho)\rho d\rho d\eta\\
	& = \psi_b^{[1,k]}(\theta)-\psi_{b}^{[2,k]}(\theta),
\end{align*}
where we set, for any $ k,\ell \in\{1,2\}$
\begin{align}\label{PSI:kl:DEF}
	 \psi_b^{[\ell,k]}(\theta)\bydef\int_{0}^{2\pi}\int_{0}^{R_\ell(\eta)}K_b\big(R_k(\theta)e^{\mathrm{i}\theta},\rho e^{\mathrm{i}\eta}\big)b(\rho)\rho d\rho d\eta.
\end{align}
Accordingly, we are led to solve the system
$$\mathcal{G}(\Omega,r_1,r_2)=0,\qquad\mathcal{G}\bydef(\mathcal{G}_1,\mathcal{G}_2),$$
where 
\begin{equation}\label{G:DEF}
	\mathcal{G}_k(\Omega,r_1,r_2)(\theta)\bydef\Omega r_k'(\theta)+\frac{1}{b\big(R_k(\theta)\big)}\partial_{\theta}\left(\psi_b^{[1,k]}(\theta)-\psi_{b}^{[2,k]}(\theta)\right), \quad k\in \{1,2\}.
\end{equation}
The regularity analysis of the contour equations in the doubly-connected case is somewhat similar to the simply-connected case. However, additional new parameters are involved due to the two-dimensional structure which require a careful attention.  

We begin by stating the following lemma which summarizes the regularity properties of the functional $\mathcal{G}$. 
\begin{lem}\label{lemma:reg:DC}
	Let $0<a_2<a_1,$ $m\in\mathbb{N}^*$ and $\alpha\in (0,1).$ There exists $\varepsilon>0$ such that the functional 
	\begin{equation*}
		\mathcal{G}:\mathbb{R}\times \mathcal{B}_{m,\varepsilon}^{\alpha}\times \mathcal{B}_{m,\varepsilon}^{\alpha}\rightarrow Y_{m}^{\alpha} \times Y_{m}^{\alpha}
	\end{equation*}
	is  well-defined and of class $C^1$. Moreover, the partial derivative $\partial_{\Omega}d_{(r_1,r_2)}\mathcal{G} $ exists in the sense that 
	\begin{equation*}
		\partial_{\Omega}d_{(r_1,r_2)}\mathcal{G}:\mathbb{R}\times \mathcal{B}_{m,\varepsilon}^{\alpha}\times \mathcal{B}_{m,\varepsilon}^{\alpha}\rightarrow\mathcal{L}\big(X_{m}^{\alpha}\times X_{m}^{\alpha},Y_{m}^{\alpha}\times Y_{m}^{\alpha}\big)
	\end{equation*} 
	and is continuous.
\end{lem}
We emphasize that the justification  of the preceding lemma can be done by means of the same analysis laid out in the proof of Proposition \ref{proposition regularity of the functional}, above.   Now, we focus our attention to the proof of the next proposition which establishes  important properties of the linearized operator $d_{(r_1,r_2)}\mathcal{G}$ at the equilibrium.

\begin{prop}\label{prop regular2}
	Let $0<a_2<a_1,$ $m\in\mathbb{N}^*$ and $\alpha\in (0,1).$ Further fix  $\varepsilon>0$ as in Lemma \ref{lemma:reg:DC} and let $Q(a_1,a_2)$ and $\big(\Lambda_n(a_1,a_2)\big)_{n\in \mathbb{N}^*}$ be defined    through \eqref{def Q} and \eqref{def Lambdan}, respectively.   If moreover
		$$\Omega\not\in\left\lbrace    Q(a_1,a_2),0\right\rbrace,$$ 
		then, at the equilibrium, the linearized operator $d_{(r_1,r_2)}\mathcal{G}(\Omega,0,0)$ is of Fredholm type with zero index and admits the following Fourier-multiplier structure: For a given $h_1,h_2\in X_m^{\alpha}$  		$$h_1(\theta)=\sum_{n=1}^{\infty}h_n^{(1)}\cos(nm\theta),\qquad h_2(\theta)=\sum_{n=1}^{\infty}h_n^{(2)}\cos(nm\theta), $$
		for some scalars $h_n^{(1)}, h_n^{(2)}\in\mathbb{R} $, it holds that 
		\begin{equation}\label{lin2}
			d_{(r_1,r_2)}\mathcal{G}(\Omega,0,0)\left[h_1,h_2\right](\theta)=-\sum_{n=1}^{\infty}nmM_{nm}(\Omega,a_1,a_2)\begin{pmatrix}
				h_n^{(1)}\\
				h_n^{(2)}
			\end{pmatrix}\sin(nm\theta),
		\end{equation}
		for any $\theta\in \mathbb{T}$, where the matrix $M_n$ is given by
		\begin{equation}\label{Matrix:Mn}
			M_{n}(\Omega,a_1,a_2)\bydef\begin{pmatrix}
			\Omega   -Q(a_1,a_2)+\Lambda_n(a_1,a_1) & -\frac{b(a_2)}{b(a_1)}\Lambda_n(a_1,a_2) \\
			\frac{b(a_1)}{b(a_2)}\Lambda_n(a_1,a_2)& \Omega -\Lambda_n(a_2,a_2)
		\end{pmatrix}.
		\end{equation}

		   \end{prop}
\begin{proof}
	We first need to compute the differential $d_{(r_1,r_2)}\mathcal{G}(\Omega,r_1,r_2)[h_1,h_2]$ at any point. To that end,  considering the representation \eqref{G:DEF},  it is readily seen that the term that requires our attention is the one involving the stream functions $\psi_b^{[k,\ell]}$. Thus, by differentiating its expression,  we find, for any $k\in \{1,2\}$, that
	\begin{align*}
		d_{r_k}\psi_b^{[k,k]}[h_k](\theta)&=\int_{0}^{2\pi}h_{k}(\eta)K_b\big(R_k(\theta)e^{\mathrm{i}\theta},R_{k}(\eta)e^{\mathrm{i}\eta}\big)b\big(R_k(\eta)\big)d\eta\\
		&\quad+h_k(\theta)\frac{e^{\mathrm{i}\theta}}{R_k(\theta)}\cdot\int_{0}^{2\pi}\int_{0}^{R_{k}(\eta)}\nabla_xK_b\big(R_k(\theta)e^{\mathrm{i}\theta},\rho e^{\mathrm{i}\eta}\big)b(\rho)\rho d\rho d\eta
	\end{align*}
	and, for $k,\ell\in\{1,2\}$ with $k\neq\ell,$ that 
	\begin{align*}
		d_{r_k}\psi_b^{[\ell,k]}[h_k](\theta)&=h_k(\theta)\frac{e^{\mathrm{i}\theta}}{R_k(\theta)}\cdot\int_{0}^{2\pi}\int_{0}^{R_{\ell}(\eta)}\nabla_xK_b\big(R_k(\theta)e^{\mathrm{i}\theta},\rho e^{\mathrm{i}\eta}\big)b(\rho)\rho d\rho d\eta\\
		d_{r_\ell}\psi_b^{[\ell,k]}[h_\ell](\theta)&=\int_{0}^{2\pi}h_{\ell}(\eta)K_b\big(R_k(\theta)e^{\mathrm{i}\theta},R_{\ell}(\eta)e^{\mathrm{i}\eta}\big)b\big(R_\ell(\eta)\big)d\eta.
	\end{align*}
	Therefore, by arranging the terms in an appropriate way, we deduce the representation 
	$$d_{(r_1,r_2)}\mathcal{G}(\Omega,r_1,r_2)[h_1,h_2]=\begin{pmatrix}
			T_{1,1}[h_1]+T_{1,2}[h_2]\\
			T_{2,1}[h_1]+T_{2,2}[h_2]
		\end{pmatrix},$$
		where we set
		\begin{align*}
			T_{1,1}[h_1]&\bydef\Omega h_1' +\frac{1}{b\circ R_1}\partial_{\theta}\Big(\big(Q_{1,1}(r_1,r_2)-Q_{1,2}(r_1,r_2)\big)h_1+P_{1,1}(r_1,r_2)[h_1]\Big)+L_1(r_1,r_2)h_1,\\
			T_{2,2}[h_2]&\bydef\Omega h_2' +\frac{1}{b\circ R_2}\partial_{\theta}\Big(\big(Q_{2,1}(r_1,r_2)-Q_{2,2}(r_1,r_2)\big)h_2-P_{2,2}(r_1,r_2)[h_2]\Big)+L_2(r_1,r_2)h_2,\\
			T_{1,2}[h_2]&\bydef-\frac{1}{b\circ R_1}\partial_{\theta}\Big(P_{1,2}(r_1,r_2)[h_2]\Big),\\
			T_{2,1}[h_1]&\bydef\frac{1}{b\circ R_2}\partial_{\theta}\Big(P_{2,1}(r_1,r_2)[h_1]\Big),
		\end{align*}
		and 
		\begin{align*}
			P_{k,\ell}(r_1,r_2)[h_{\ell}](\theta)&\bydef\int_{0}^{2\pi}h_{\ell}(\eta)K_b\big(R_k(\theta)e^{\mathrm{i}\theta},R_{\ell}(\eta)e^{\mathrm{i}\eta}\big)b\big(R_\ell(\eta)\big)d\eta,\\
			Q_{k,\ell}(r_{1},r_2)(\theta)&\bydef\frac{e^{\mathrm{i}\theta}}{R_k(\theta)}\cdot\int_{0}^{2\pi}\int_{0}^{R_{\ell}(\eta)}\nabla_xK_b\big(R_k(\theta)e^{\mathrm{i}\theta},\rho e^{\mathrm{i}\eta}\big)b(\rho)\rho d\rho d\eta,\\
			L_k(r_1,r_2)(\theta)&\bydef-\frac{b'\big(R_k(\theta)\big)}{R_k(\theta)b^2\big(R_k(\theta)\big)}\partial_{\theta}\left(\psi_b^{[1,k]}(\theta)-\psi_b^{[2,k]}(\theta)\right).
		\end{align*}
	We now focus our attention to the justification of the representation \eqref{lin2} of $ d_{(r_1,r_2)}\mathcal{G}$ at the equilibrium $(\Omega,0,0)$. To that end, we proceed   in a similar way to the proof of Lemma \ref{lemma:linearized-OP} for the simply-connected case. First, we e emphasize   that 
	$$L_{k}(0,0)=0,$$
	for any $k\in \{1,2\}$. This is similar to \eqref{L=0} and can be shown by employing   the symmetry property of the kernel $K_b$, which in turn implies \eqref{diff0}.
		Moreover, again, along the same lines as in the proof of Lemma \ref{lemma:linearized-OP} for the simply-connected case, we can easily check that 
	\begin{align*}
		Q_{k,\ell}(0,0)(\theta)&=\frac{e^{\mathrm{i}\theta}}{a_k}\cdot\int_{0}^{2\pi}\int_{0}^{a_\ell}\nabla_xK_b\big(a_ke^{\mathrm{i}\theta},\rho e^{\mathrm{i}\eta}\big)b(\rho)\rho d\rho d\eta\\
		&=\frac{1}{a_k}\int_{0}^{2\pi}\int_{0}^{a_\ell}\partial_{a_k}\Big(K_b\big(a_ke^{\mathrm{i}\theta},\rho e^{\mathrm{i}\eta}\big)\Big)b(\rho)\rho d\rho d\eta\\
		&=\frac{1}{a_k}\int_{0}^{2\pi}\int_{0}^{a_\ell}\partial_{a_k}\Big(K_b\big(a_k,\rho e^{\mathrm{i}(\eta-\theta)}\big)\Big)b(\rho)\rho d\rho d\eta\\
		&=\frac{1}{a_k}\int_{0}^{2\pi}\int_{0}^{a_\ell}\partial_{x_1}K_b\big(a_k,\rho e^{\mathrm{i}\eta}\big)b(\rho)\rho d\rho d\eta  \\
		&=-\frac{b(a_k)}{a_k^2}  \int_0^{\min \{a_k,a_\ell \}}  \tau b(\tau) d\tau ,
	\end{align*}
	for any $k,\ell\in \{1,2\}$ and $\theta\in \mathbb{T}$, where we have used Lemma \ref{velocity_expl_comp} to compute the integral in the last line. At last, we also compute that 
	\begin{align*}
		P_{k,\ell}(0,0)[h_\ell](\theta)&=b(a_\ell)\int_{0}^{2\pi}h_{\ell}(\eta)K_b\big(a_ke^{\mathrm{i}\theta},a_\ell e^{\mathrm{i}\eta}\big)d\eta\\
		&=b(a_\ell)\int_{0}^{2\pi}h_{\ell}(\eta)K_b\big(a_k,a_\ell e^{\mathrm{i}(\eta-\theta)}\big)d\eta\\
		&=b(a_\ell)\sum_{n=1}^{\infty}h_{n}^{(\ell)}\cos(nm\theta)\int_{0}^{2\pi}K_b\big(a_k,a_\ell e^{\mathrm{i}\eta}\big)\cos(nm\eta)d\eta\\
		&=b(a_\ell)\sum_{n=1}^{\infty}h_{n}^{(\ell)}\Lambda_{nm}(a_k,a_\ell)\cos(nm\theta),
	\end{align*}
	for any $k,\ell\in \{1,2\}$ and $\theta\in \mathbb{T}$. All in all,    gathering  the foregoing identities leads to the desired Fourier representation \eqref{lin2}. 
	
	Now we turn to the proof of the Fredholm property. To that end, we introduce the classical $2\pi$-periodic Hilbert transform defined by
$$\mathcal{H}f(\theta)\bydef\frac{1}{2\pi}\int_{0}^{2\pi}f(\eta)\cot\left(\tfrac{\theta-\eta}{2}\right)d\eta$$
and we emphasize that 
$$ \mathcal{H}\cos(n\theta)=\sin(n\theta),  $$
for all  $n\in\mathbb{N}^*.$
Therefore, using the decomposition \eqref{exp Lambdan}, we   write that 
$$d_{(r_1,r_2)}\mathcal{G}(\Omega,0,0)=\mathscr{I}+\mathscr{H}+\mathscr{R},$$
where
\begin{align*}
	\mathscr{I}&\bydef\begin{pmatrix}
		\big (\Omega  -Q(a_1,a_2)\big )\partial_{\theta} & 0\\
		0 &  \Omega  \partial_{\theta}
	\end{pmatrix},  \qquad\qquad
\mathscr{H}\bydef\begin{pmatrix}
	-\tfrac{b(a_1)}{2}\mathcal{H} & 0\\
	0 & \tfrac{b(a_2)}{2}\mathcal{H}
\end{pmatrix}
\end{align*}
and
$$\mathscr{R}[h_1,h_2](\theta)=\sum_{n=1}^{\infty}\mathscr{R}_{nm}(a_1,a_2)\begin{pmatrix}
	h_n^{(1)}\\
	h_n^{(2)}
\end{pmatrix}\sin(nm\theta),$$  \begin{equation*}
	\mathscr{R}_{n}  = -n \begin{pmatrix}
			 f_n(a_1,a_1) & -\frac{b(a_2)}{b(a_1)}\Lambda_n(a_1,a_2) \\
			\frac{b(a_1)}{b(a_2)}\Lambda_n(a_1,a_2)&   -f_n(a_2,a_2)
		\end{pmatrix},
\end{equation*}
where $f_n$ is defined in Proposition \ref{lemma:eigenfunction}. In particular, in view of the asymptotic analysis performed therein, one sees that  
\begin{equation*}
	\mathscr{R}_{n} = O\left (\frac 1{n^2}\right). 
	\end{equation*}
	Moreover, it  is readily seen that if 
\begin{equation}\label{constraint OMG}
	\Omega\not\in\left\lbrace  Q(a_1,a_2),0\right\rbrace
\end{equation}
then the operator $$\mathscr{I}:X_m^{\alpha}\times X_m^{\alpha}\rightarrow Y_m^{\alpha}\times Y_{m}^{\alpha}$$ is an isomorphism. Moreover, it is proved in \cite[Proposition 2.1-(iii)]{GHR23} that the Hilbert transform $\mathcal{H}$ as a mapping from $X_m^{\alpha}$ into $   Y_m^{\alpha}$ is a compact operator. Consequently, the operator 
$$\mathscr{H}:X_m^{\alpha}\times X_m^{\alpha}\rightarrow Y_m^{\alpha}\times Y_m^{\alpha}$$ is also compact.

 The remainder term $\mathscr{R}$  is actually better in terms of regularity. Indeed, the decay of its coefficients implies that  $\mathscr{R}$ maps $X_m^{\alpha}\times X_m^{\alpha}$ into $ Y_m^{2+\alpha}\times Y_m^{2+\alpha}.$
Hence,  due to   the compact embedding
$$C^{2+\alpha}\hookrightarrow   C^{\alpha} ,$$ 
we deduce that $$\mathscr{R}:X_m^{\alpha}\times X_m^{\alpha}\rightarrow Y_m^{\alpha}\times Y_m^{\alpha}$$ is a compact operator, as well. All in all, under the condition \eqref{constraint OMG}, the linearized operator $d_{(r_1,r_2)}\mathcal{G}(\Omega,0,0)$ is a compact perturbation of an isomorphism, namely a Fredholm type operator with zero index, see \cite[Corollary 5.9]{CR21}. This concludes the proof of the proposition. 
\end{proof}

The next step consists in the  study  of  singularities of the matrix $M_{n}$, previously introduced in     Proposition \ref{prop regular2}. This is crucial in order to show the one-dimensional kernel property of the linearized operator $d_{(r_1,r_2)}\mathcal{G}$, which is a requirement in applying  Crandall-Rabinowitz's Theorem.

\begin{prop}\label{prop spec}  Let $0<a_2<a_1$, there exists $N(a_1,a_2,b)\in\mathbb{N}^*$ such that, for any $n\geqslant N(a_1,a_2,b),$ there are two distinct angular velocities 
	\begin{align*}
		\Omega_n^{\pm}  \bydef \ 
		& \frac{Q(a_1,a_2)}{2}+\frac{\Lambda_n(a_2,a_2)-\Lambda_n(a_1,a_1)}{2}\\
		&  \pm\frac{1}{2}\sqrt{\left( Q(a_1,a_2)-\Lambda_n(a_1,a_1)-\Lambda_n(a_2,a_2)\right)^2-4\Lambda_n^2(a_1,a_2)}.
	\end{align*}
	for which the matrix $M_n\big(\Omega_n^{\pm} ,a_1,a_2\big)$, introduced in \eqref{Matrix:Mn}, is singular. 
	
	Moreover, the sequences $ \big(\Omega_{n}^{\pm} \big)_{n\geq N(a_1,a_2,b)}$ are strictly monotone   satisfying
	\begin{equation*}
		\lim_{n\to\infty} \Omega_{n}^{+}  =   Q(a_1,a_2) ,  \qquad \lim_{n\to\infty}  \Omega_{n}^{-}  =0
	\end{equation*} 
	and  
		$$ \Omega_{p}^{-} \neq\Omega_{q}^{+} ,$$
		for all $p,q\geqslant N(a_1,a_2,b)$.
\end{prop}

\begin{proof} Owing to \eqref{Matrix:Mn}, we first compute that
	$$P_n(\Omega)\bydef \det\big(M_n(\Omega,a_1,a_2)\big)=\Omega^2+\beta_n(a_1,a_2)\Omega+\gamma_n(a_1,a_2),$$
	where 
	\begin{equation*}
	\beta_n(a_1,a_2)\bydef  - Q(a_1,a_2) +\Lambda_n(a_1,a_1)-\Lambda_n(a_2,a_2),
	\end{equation*}
	and
	\begin{equation*}
		\gamma_n(a_1,a_2)\bydef-\Lambda_n(a_2,a_2)\Big( \Lambda_n(a_1,a_1)- Q(a_1,a_2)\Big)  +\Lambda_{n}^2(a_1,a_2).
	\end{equation*}
The discriminant of the preceding second order polynomial is given by
\begin{equation}\label{discrimant}
	\begin{aligned} \Delta_n &\bydef\beta^2_n(a_1,a_2)-4\gamma_n(a_1,a_2)
	    \\
		&=\left(Q(a_1,a_2)-  \Lambda_n(a_1,a_1)-\Lambda_n(a_2,a_2)  \right)^2  -4\Lambda_n^2(a_1,a_2).
	\end{aligned}
\end{equation} 
	Hence, by virtue of the results of Proposition \ref{lemma:eigenfunction}, it is readily seen that  
		$$\Delta_{n} \underset{n\to\infty}{\longrightarrow} \big (Q(a_1,a_2) \big )^2 > 0.$$ 
	Thus, there exists $N= N(a_1,a_2,b)\in\mathbb{N}^* $ such that 
	\begin{equation}\label{sign dircriminant}
		  \Delta_{n} >0, \quad \text{for all}\quad  n\geqslant N.
	\end{equation}
	 Therefore, the polynomial $ P_n$ admits two distinct roots $\Omega_n^\pm $ being as given as in the statement of the proposition. Note that, up to increasing the value of $N$ if needed,  we can assume that 
	\begin{equation}\label{CD:lake-Euler}
		Q(a_1,a_2) > \Lambda_n(a_1,a_1)+\Lambda_n(a_2,a_2),
	\end{equation} 
	for all $n\geq N$,  which  allows us to write that  
		\begin{align*}
		\Omega_{n}^{+} &= Q(a_1,a_2)- \Lambda_n(a_1,a_1) +  \mathtt{r}_{n} 
	\end{align*}
	and 
	\begin{align*}
		\Omega_{n}^{-} &= \Lambda_n(a_2,a_2) -  \mathtt{r}_{n} ,
	\end{align*}
	where we set 
	\begin{align*}
		\mathtt{r}_{n} &\bydef\frac{1}{2}
		\big ( 
		Q(a_1,a_2)-  \Lambda_n(a_1,a_1)-\Lambda_n(a_2,a_2)
		\big )\\
		&\quad\times\left(\sqrt{1-\frac{4\Lambda_n^2(a_1,a_2)}{\left(Q(a_1,a_2)-  \Lambda_n(a_1,a_1)-\Lambda_n(a_2,a_2)\right)^2}}-1\right).
	\end{align*}
	Moreover,  observe,  due to \eqref{exp Lambdan}, \eqref{bnd fn_prop}, that
		$$\mathtt{r}_{n} =O \left(\Lambda_n^2(a_1,a_2)\right)= O \left(\frac{1}{n^3}\right),$$
		whence,
		\begin{equation*}
			\Omega_{n+1}^{+} -\Omega_{n}^{+} =\frac{b(a_1)}{2n(n+1)}+O \left(\frac{1}{n^3}\right)
		\end{equation*}
		and 
		\begin{equation*}
			\Omega_{n+1}^{-} -\Omega_{n}^{-} =-\frac{b(a_2)}{2n(n+1)}+O \left(\frac{1}{n^3}\right).
		\end{equation*} 
				All in all, we deduce that the sequence $(\Omega_n^{+} )_{n\geq N} $ is strictly increasing whereas the sequence  $(\Omega_n^{-} )_{n\geq N} $ is strictly decreasing. Hence,  combining  this with
			\begin{equation*}
				\Omega_n^- < \Omega_n^+ ,
			\end{equation*}	  		
			 for all $  n \geq N$,   yields that 			\begin{equation*}
				\Omega_p^-< \Omega_q^- < \Omega_q^+ < \Omega_p^+,
			\end{equation*}
			for all $p>q\geq  N,$ thereby showing that there is no spectral collisions. This completes the proof of the proposition.

			 \end{proof}
			 
			 \subsection{Proof of Theorem \ref{thm:2}}
Now, we prove the last proposition that contains the final prerequisites that we need to apply  Crandall-Rabinowitz's Theorem to, eventually, conclude the proof of Theorem \ref{thm:2}.
\begin{prop}\label{prop.last2}
	Let $\alpha\in(0,1)$,  $0<a_2<a_1,$ and  $  N(a_1,a_2,b) \in \mathbb{N}^*$  be as  in Proposition \ref{prop spec}.  
	 Then, for all  $m\geq N(a_1,a_2,b)$, the following properties hold true:
	\begin{enumerate}
		\item The kernel of the linearized operator
		 $$ d_{(r_1,r_2)}\mathcal{G}\big(\Omega_{m}^{\pm} ,0,0\big) : X_m^\alpha \times X_m^\alpha \to Y_m^\alpha \times Y_m^\alpha $$ is one-dimensional and generated by 
		$$r_0:\theta\mapsto\begin{pmatrix}
			\Omega_{m}^{\pm}  -\Lambda_m(a_2,a_2)\vspace{2mm}\\
			-\frac{b(a_1)}{b(a_2)}\Lambda_{m}(a_1,a_2)
		\end{pmatrix}\cos(m\theta).$$
		\item The range of the linearized operator $ d_{(r_1,r_2)}\mathcal{G}\big(\Omega_{m}^{\pm} ,0,0\big) $ is closed and of co-dimension one.
		\item The transversality condition	is satisfied, that is	$$\partial_{\Omega}d_{(r_1,r_2)}\mathcal{G}\big(\Omega_{m}^{\pm} ,0,0\big)[r_0]\not\in\textnormal{Range}\Big(d_{(r_1,r_2)}\mathcal{G}\big(\Omega_{m}^{\pm} ,0,0\big)\Big).$$
	\end{enumerate} 
\end{prop}

\begin{proof} 
	\textit{(1)} According to the analysis laid out in Proposition \ref{prop spec}, we deduce from the simplicity of the spectrum  that 
	$$\det\Big(M_{m}\big(\Omega_{m}^{\pm} ,a_1,a_2\big)\Big)=0, $$
	and, for all integer $n\geq 2$,
	\begin{equation*}
		\det\Big(M_{mn}\big(\Omega_{m}^{\pm} ,a_1,a_2\big)\Big)\neq0.
	\end{equation*}
	Consequently, in view of the representation \eqref{lin2} of the linearized operator $ d_{(r_1,r_2)}\mathcal{G}$ at the value  $\big(\Omega_{m}^{\pm} ,0,0\big)$, one   deduces that the kernel of this operator is one-dimensional generated by the vector $r_0$ as  defined in the statement of the proposition.

	\textit{(2)} From the strict  monotonicity and limit properties  of the sequences $\big(\Omega_{n}^{\pm} \big)_{n>N(a_!,a_2,b)}$, as $n\to \infty$,  obtained in Proposition \ref{prop spec}, we find that 
	$$\Omega_{m}^{\pm} \not\in\left\lbrace  Q(a_1,a_2) ,0\right\rbrace.$$
	Therefore, by virtue of Proposition \ref{prop regular2}, this implies the Fredholm property for the operator $d_{(r_1,r_2)}\mathcal{G}\big(\Omega_{m}^{\pm} ,0,0\big).$ Hence, combining this with  the one-dimensional property of its kernel, established before, we   deduce that the range of the this operator is closed and of co-dimension one.
	
	Now, let us be more precise about the range of the linearized operator. To that end, we endow the vector space $Y_{m}^{\alpha}\times Y_{m}^{\alpha}$ with the pre-Hilbertian structure given by the following $L^2$-scalar product 
	$$
\begin{aligned}
\Bigg(\left(\sum_{n=1}^{\infty}a_n\sin(nm\theta),\sum_{n=1}^{\infty}c_n\sin(nm\theta)\right) &  \left| 
\left(\sum_{n=1}^{\infty}b_n\sin(nm\theta),\sum_{n=1}^{\infty}d_n\sin(nm\theta)\right)\right.\Bigg)
\\
& \bydef\sum_{n=1}^{\infty}a_nb_n+c_nd_n.
\end{aligned}	
	$$
	Then, observe that
	$$u_{0,m}^{\pm} \bydef\begin{pmatrix}
		\Omega_{m}^{\pm}  -\Lambda_m(a_2,a_2)\vspace{2mm}\\
		\frac{b(a_2)}{b(a_1)}\Lambda_{m}(a_1,a_2)
	\end{pmatrix}\in \ker\Big(M_m^{\top}\big(\Omega^{\pm}_{m} ,a_1,a_2\big)\Big).$$
The notation $M^{\top}$ refers to the transpose of the matrix $M$.
	Hence, defining 
	$$y_0:\theta\mapsto u_{0,m}^{\pm} \sin(m\theta),$$  we then write in view of 
  \eqref{lin2}, for any $(h_1,h_2)\in X_{m}^{\alpha}\times X_m^{\alpha}$ given by
 $$h_{1}(\theta)=\sum_{n=1}^{\infty}h_n^{(1)}\cos(nm\theta),\qquad h_{2}(\theta)=\sum_{n=1}^{\infty}h_n^{(2)}\cos(nm\theta),$$
 for some $ h_n^{(1)},h_n^{(2)} \in \mathbb{R}$, that
\begin{align*}
	\left(d_{(r_1,r_2)}\mathcal{G}\big(\Omega_{m}^{\pm} ,0,0\big)[h_1,h_2]\big|y_0\right)&=-m\left\langle M_m\big(\Omega_m^{\pm} ,a_1,a_2\big)\begin{pmatrix}
		h_1^{(1)}\\
		h_1^{(2)}
	\end{pmatrix},u_{0,m}^{\pm} \right\rangle_{\mathbb{R}^2}\\
	&=-m\left\langle\begin{pmatrix}
		h_1^{(1)}\\
		h_1^{(2)}
	\end{pmatrix},M_m^{\top}\big(\Omega_m^{\pm} ,a_1,a_2\big)u_{0,m}^{\pm} \right\rangle_{\mathbb{R}^2}\\
&=0,
\end{align*}
where, we utilize the notation $\langle\cdot,\cdot\rangle_{\mathbb{R}^2}$ for the Euclidean scalar product on $\mathbb{R}^2$.
This proves that 
	$$\textnormal{Range}\Big(d_{(r_1,r_2)}\mathcal{G}\big(\Omega_{m}^{\pm} ,0,0\big)\Big)\subset\mathtt{span}_{\alpha}^{\perp_{(\cdot|\cdot)}}(y_0),$$
	where
	$$\mathtt{span}_{\alpha}^{\perp_{(\cdot|\cdot)}}(y_0)\bydef\Big\{y\in Y_{m}^{\alpha}\times Y_{m}^{\alpha}:\; (y|y_0)=0\Big\}.$$

	Now, since $\mathtt{span}(y_0)$ has a finite dimension which is equal to one, then the orthogonal supplementary theorem yields that 
	$$Y_m^{\alpha}\times Y_m^{\alpha}=\mathtt{span}(y_0)\oplus\mathtt{span}_{\alpha}^{\perp_{(\cdot|\cdot)}}(y_0).$$
	This proves that $\mathtt{span}_{\alpha}^{\perp_{(\cdot|\cdot)}}(y_0)$ is of codimension one in $Y_{m}^{\alpha}\times Y_m^{\alpha}.$ The range being also of codimension one in this space, then applying \cite[Lemma B.1]{R23} implies that 
	\begin{equation}\label{description range}
		\textnormal{Range}\Big(d_{(r_1,r_2)}\mathcal{G}\big(\Omega_{m}^{\pm} ,0,0\big)\Big)=\mathtt{span}_{\alpha}^{\perp_{(\cdot|\cdot)}}(y_0).
	\end{equation}  
	\textit{(3)} According to \eqref{description range}, notice that   the transversality property holds once we show that the quantity
	\begin{align*}
		\Big(\partial_{\Omega}d_{(r_1,r_2)}\mathcal{G} &\big(\Omega_{m}^{\pm} ,0,0\big)[r_0]|y_0\Big) =\left(\Omega_{m}^{\pm}  -\Lambda_m(a_2,a_2)\right)^2-\Lambda_{m}^{2}(a_1,a_2)
	\end{align*}
	does not vanish.
	To that end, exploiting the fact that
	$$\det\Big(M_{m}\big(\Omega_{m}^{\pm} ,a_1,a_2\big)\Big)=0,$$
	which allows us to write 
	\begin{equation*}
		\begin{aligned}
			-\Lambda_{m}^{2}(a_1,a_2)  
			&= \left(\Omega_{m}^{\pm}  -\Lambda_m(a_2,a_2)\right)  \left(\Omega_{m}^{\pm} -Q(a_1,a_2) +\Lambda_m(a_1,a_1)\right),
		\end{aligned}
	\end{equation*}
	we then obtain that 
	\begin{equation*}
		\begin{aligned}
			\Big(\partial_{\Omega}&d_{(r_1,r_2)}\mathcal{G}\big(\Omega_{m}^{\pm} ,0,0\big)[r_0]|y_0\Big)\\
		&=\left(\Omega_{m}^{\pm}  -\Lambda_m(a_2,a_2)\right) \left(2\Omega_{m}^{\pm} - Q(a_1,a_2) +\Lambda_m(a_1,a_1)-\Lambda_m(a_2,a_2)\right).
		\end{aligned}
	\end{equation*} 
	Therefore,   we deduce that  the identity
	 \begin{equation*}
	 	\Big(\partial_{\Omega}d_{(r_1,r_2)}\mathcal{G}\big(\Omega_{m}^{\pm} ,0,0\big)[r_0]|y_0\Big)= 0
	 \end{equation*}
	 holds if and only if 
	 \begin{equation*}
	 	2\Omega_{m}^{\pm}  -Q(a_1,a_2) +\Lambda_m(a_1,a_1)-\Lambda_m(a_2,a_2)=0
	 \end{equation*}
	 or 
	   \begin{equation*}
	   	\Omega_{m}^{\pm}  =\Lambda_m(a_2,a_2) .
	   \end{equation*}
	 Now, by virtue of the expressions of $\Omega_m^\pm$ given in \eqref{prop spec}, the first possibility above yields  $\Delta_m(a_1,a_2)=0$, which is excluded by \eqref{sign dircriminant} as long as $m\geq N(a_1,a_2,b).$ As for the second possibility,  it is readily seen that the same expressions of $\Omega_m^\pm$ give that
	 \begin{equation*}
	 	\Lambda_m(a_1,a_1) + \Lambda_m(a_2,a_2)- Q(a_1,a_2)   = \pm \sqrt{\Delta_m(a_1,a_2)},
	 \end{equation*}
	 which, upon raising to the power-two, yields that 
	 \begin{equation*}
	 	\Lambda_m(a_1,a_2)=0.
	 \end{equation*}
	 Notably,   this identity is violated  by the third result in Proposition \ref{lemma:eigenfunction} as soon as  $m$ is large enough in the sense that 
	 \begin{equation*}
	 	m >    M( b),
	 \end{equation*}
	 where $M(b)$ is as defined in Proposition \ref{lemma:eigenfunction}.	Note that our lower bound $N(a_1,a_2,b)$ can be chosen greater than $M(b)$ to ensure the last condition. All in all, we conclude that 
	\begin{equation*}
		\partial_{\Omega}d_{(r_1,r_2)}\mathcal{G}\big(\Omega_{m}^{\pm} ,0,0\big)[r_0] \notin \mathtt{span}_{\alpha}^{\perp_{(\cdot|\cdot)}}(y_0) = \textnormal{Range}\Big(d_{(r_1,r_2)}\mathcal{G}\big(\Omega_{m}^{\pm} ,0,0\big)\Big),
	\end{equation*}
	thereby completing the proof of the lemma. 
	\end{proof}
\begin{proof}[Proof of Theorem \ref{thm:2}]
	The proof of Theorem \ref{thm:2} is achieved by employing the results of Proposition \ref{prop.last2} and applying Crandall--Rabinowitz Theorem \ref{Crandall-Rabinowitz theorem}. 
\end{proof} 
   

\appendix 
\section{Crandall--Rabinowitz's Theorem}
Here, we recall the celebrated Crandall--Rabinowitz theorem which was utilized to prove our main theorems.

\begin{thm}[Crandall-Rabinowitz \cite{CR71}]\label{Crandall-Rabinowitz theorem}
	Let $X$ and $Y$ be two Banach spaces.  Consider $V \subset X$ to be a neighborhood of $0$ and a function $F$ such that
	$$F:\mathbb{R}\times V\rightarrow Y.$$
	Assume   that $F$ enjoys the following properties.
	\begin{enumerate}
		\item Existence of trivial branch:
		$$ F(\Omega,0)=0, \quad \text{for all} \quad \Omega\in\mathbb{R}.$$
		\item Regularity: $F$ is       regular in the sense that $\partial_\Omega F $, $d_x F$ and $\partial_{\Omega}d_xF$ exist and are continuous.
		\item Fredholm property:       
		The kernel   $\ker\big(d_{x}F(0,0)\big)$ is  of dimension one, i.e., there is $x_0\in V$ such that
		$$\ker\big(d_{x}F(0,0)\big)=\mathtt{span}(x_{0}),$$  
		    and the $\textnormal{Range}\big(d_{x}F(0,0)\big)$ is closed and of a co-dimension one. 
		\item Transversality:  		$$\partial_{\Omega}d_xF(0,0)[x_{0}]\not\in \textnormal{Range}\big(d_{x}F(0,0)\big).$$
	\end{enumerate}
	Then, denoting $\chi$ any complement of $\ker\big(d_{x}F(0,0)\big)$ in $X$, there exist a neighborhood $U$ of $(0,0)$, an interval $(-a,a)$, for some $a>0$ and continuous functions
	$$\psi:(-a,a)\rightarrow \mathbb{R}\quad\textnormal{and}\quad\phi:(-a,a)\rightarrow\chi,$$
	such that
	$$\psi(0)=0,\qquad\phi(0)=0$$ 
	and 
	$$\Big\{(\Omega,x)\in U: F(\Omega,x)=0\Big\}=\Big\{\big(\psi(s),sx_{0}+s\phi(s)\big): \;|s|<a\Big\}\cup\Big\{(\Omega,0)\in U\Big\}.$$
\end{thm}

\section{Back to the elliptic problem}\label{appendix:elliptic}
This section  serves as a appendix that completes, in a self-contained and rigorous way, the results presented in Section \ref{Sec:elliptic}. We believe that some of the techniques that we present hereafter provide simple (constructive) approaches to justify some results from elliptic theory in a specific case of an elliptic equation with variable coefficient.

For the record, we   stress out that several arguments in this appendix draw insight from \cite{DS20} on the same elliptic problem but on bounded domains, as well the analysis from \cite{CH95} for the case $b\equiv 1$ and in the whole plane $\mathbb{R}^2$. Moreover, we will grab a good portion of notions from the nice notes \cite{OS12} and employ them to prove the existence and uniqueness of weak solutions of \eqref{EL:P1}.

Once again, the overall results presented in this direction are not brand-new (except for the relevant decomposition of solutions from Lemma \ref{lemma:expansion1} which is, to the best of our knowledge, the main novelty in the elliptic analysis we perform in this paper), but the way they are justified here is considerably self-contained and rigorous.


{\color{olive}

}

\subsection{Elliptic regularity}\label{section:A:regularity}

We begin with providing a full justification of Proposition \ref{prop:Elliptic:00} which establishes the primary  elliptic regularities  enjoyed by the solutions of \eqref{EL:P1}.

\begin{proof}[Proof of Proposition \ref{prop:Elliptic:00}]
We begin with splitting  the solution of \eqref{EL:P1} as $\psi_b =\psi_{b,\rad}+\bar\psi_{b}$, where each part of this decomposition is governed by its own elliptic equation
\begin{equation}\label{EL:EQ:BBB} 
		\displaystyle-\div\left(\frac{1}{b}\nabla\psi_{b,\rad }\right)=f_\rad 
\end{equation}
and 
\begin{equation} \label{EL:EQ:BBB222} 
		\displaystyle-\div\left(\frac{1}{b}\nabla\bar \psi_{b}\right)=\bar f,
\end{equation}
where,  $f_\rad $ refers to the radially symmetric part of $f$ which is defined by
\begin{equation}\label{RAD:F}
\begin{aligned}
		f_\rad  (r)\bydef \frac{1}{2\pi} \int_0^{2\pi} f\big (r(\cos \theta, \sin \theta)\big) d\theta,  
			\end{aligned}
	\end{equation}
	where, $r\bydef |x|$, for any $x\in \mathbb{R}^2$, whereas  $\bar f\bydef f-f_\rad$ denotes the remaining part of the source term which satisfies
	\begin{equation}\label{int:f:bar}
		\int_{\mathbb{R}^2} \bar f(x) dx = 0.
	\end{equation}
 Now, observe  that, from their definitions, $f_\rad$ and $\bar{f}$   enjoy the bounds
	\begin{equation*} 
	\norm {f_\rad}_{L^p(\mathbb{R}^2)} \leq  \norm {f}_{L^p(\mathbb{R}^2)} 
\end{equation*}
and 
\begin{equation*}
	\norm {\bar{f}}_{L^p(\mathbb{R}^2)} \leq  2 \norm {f}_{L^p(\mathbb{R}^2)}. 
\end{equation*} 
Noticing, moreover, that $f_\rad$ and $b$  are radial functions, by construction and due to \eqref{Hb3:eq} respectively,  one can easily solve \eqref{EL:EQ:BBB} using radial coordinates to show that the function 
\begin{equation}\label{psi:ODE:sol}
	\psi_{b,\rad } (r)\bydef - \int_0^{r}\frac{b(s)}{s}\int_0^s \tau f_\rad  (\tau)d\tau ds, \quad \text{for all} \quad r\bydef |x|\geq 0 
	\end{equation}
is an exact solution of \eqref{EL:EQ:BBB}.  In order to show that, we only need to recast \eqref{EL:EQ:BBB} in the radial coordinates, which writes as 
\begin{equation*}
	- \frac{d}{dr}\left( \frac{r\frac{d}{dr} \psi_{b,\rad }(r)}{b(r)} \right) = rf_\rad (r),
\end{equation*}
for all $r>0$. It follows then that the unique solution of this second order ODE with boundary conditions
\begin{equation*}
	\psi_{b,\rad}(0)= \frac{d}{dr}\psi_{b,\rad }(0)=0
\end{equation*}
 is then given by \eqref{psi:ODE:sol}. We are not going to show that these boundary conditions are satisfied as this will not serve the purpose of our results; but one can show that by employing  the ideas that we develop below.
 
 Let us now take care of the bounds on $\psi_{b,\rad }$ at first  as they will be obtained by direct estimates on its explicit expression. To that end, computing, for all $r>0,$  that 
\begin{equation*}
	 \frac{d}{dr} \psi_{b,\rad }(r)=  -  \frac{b(r)}{r}\int_0^r \tau f_\rad (\tau)d\tau  , 
\end{equation*}
and, whence
\begin{equation*} 
	- \frac{d^2  }{dr^2}\psi_{b,\rad}(r)= \frac{b'(r)}{r}\int_0^r \tau f_\rad (\tau)d\tau ds - \frac{b(r)}{r^2}\int_0^r \tau f_\rad (\tau)d\tau ds + b(r) f_\rad(r) .
\end{equation*}
Therefore, it is readily seen that 
\begin{equation*}
	\begin{aligned}
		\norm {\nabla^2 \psi_{b,\rad } }_{L^m(\mathbb{R}^2)}
		& \lesssim  \norm {b'}_{L^\infty(\mathbb{R}^2)} \left(  \int_0^{R_\infty} \left( \frac{1}{r}\int_0^r \tau  | f_\rad (\tau)|d\tau\right)^m rdr \right)^{\frac{1}{m}}
		\\
		&\quad +  \norm{b}_{L^\infty(\mathbb{R}^2)} \left(  \int_0^\infty \left( \frac{1}{r^2}\int_0^r \tau  |f_\rad (\tau)|d\tau\right)^m rdr \right)^{\frac{1}{m}}
	\\
	 & \quad + \norm{b}_{L^\infty(\mathbb{R}^2)} \norm {f}_{L^m(\mathbb{R}^2)},
	\end{aligned}
\end{equation*}
for all $m\in (1,p]$, where we have used \eqref{Hb3:eq} to localize the first integral   in the right-hand side on a bounded set. Thus, writing, by rescaling, that
\begin{equation*} 
	\frac{1}{r}\int_0^r \tau f_\rad (\tau)d\tau  =  \int_0^1 \sigma r f_\rad(\sigma r) d\sigma,
\end{equation*} 
yields  
\begin{equation*}
	\begin{aligned}
		\norm {\nabla^2 \psi_{b,\rad } }_{L^m(\mathbb{R}^2)}
		& \lesssim  \norm {b}_{W^{1,\infty }(\mathbb{R}^2)} \left(   \int_0^1  \norm{  f  (\sigma  \cdot)}_{L^m(\mathbb{R}^2)}\sigma d\sigma  + \norm {f}_{L^m(\mathbb{R}^2)} \right)	
		\\
		& \lesssim  \norm {b}_{W^{1,\infty } (\mathbb{R}^2)}  \norm {f}_{L^m(\mathbb{R}^2)} \left(   \int_0^1 \sigma  ^{1-\frac{2}{m}}d\sigma  + 1 \right)	
		\\
		&\lesssim  \norm {b}_{W^{1,\infty } (\mathbb{R}^2)}  \norm {f}_{L^m(\mathbb{R}^2)} .
	\end{aligned}
\end{equation*}

  We turn our attention now to the control of $\bar \psi_{b}$. To that end, we first notice that,  the assumptions on $f$ combined with the property \eqref{int:f:bar} allows us to employ Lemma \ref{lemma:duality} below and deduce that  $\bar f \in \dot H^{-1} (\mathbb{R}^2)$ with
  \begin{equation*}
  	 \norm {\bar f}_{\dot H^{-1} (\mathbb{R}^2)} \lesssim \norm f_{L^1((1+|x|)dx;\mathbb{R}^2) \cap L^2(\mathbb{R}^2)}.
  \end{equation*}
  Thus,  performing an $L^2$ energy estimate for \eqref{EL:EQ:BBB222} leads to the bound
\begin{equation*}
	\begin{aligned}
		\int_{\mathbb{R}^2} \frac{1}{b(x)} |\nabla \bar \psi_{b}(x)|^2dx 
		&= \left| \int_{\mathbb{R}^2} \bar f(x)\bar \psi_b(x)dx \right|
		\\
		&\lesssim    \norm {\bar f}_{\dot H^{-1} (\mathbb{R}^2)}   \norm {\bar \psi_{b}}_{\dot H^{1} (\mathbb{R}^2)}  , 
	\end{aligned}
\end{equation*}
whence
\begin{equation*}
	\norm { \nabla \bar \psi_{b}}_{L^2(\mathbb{R}^2)} \lesssim_b \norm f_{L^1((1+|x|)dx;\mathbb{R}^2) \cap L^2(\mathbb{R}^2)}.
\end{equation*}

We now establish the remaining bounds on $\bar\psi_{b}$ by a bootstrap argument. To see that, we expand the derivatives in the left-hand side of \eqref{EL:EQ:BBB222} to obtain that
\begin{equation}\label{ELQ:000}
	-\Delta \bar \psi_{b} = b\bar f - \frac{\nabla b}{b} \cdot \nabla \bar \psi_{b}.
\end{equation}
Then, observing that the right-hand side above belongs to $L^2(\mathbb{R}^2)$, it then follows, by virtue of classical elliptic regularity, that 
\begin{equation*}
	\norm {\nabla^2 \bar \psi_{b}}_ { L^2(\mathbb{R}^2)} \lesssim _b \norm {f}_{L^2(\mathbb{R}^2)} + \norm {\nabla \bar  \psi_{b}}_{ L^2(\mathbb{R}^2)} \lesssim  _b\norm {f}_{L^1((1+|x|)dx;\mathbb{R}^2) \cap L^2(\mathbb{R}^2)}.
\end{equation*}
This allows us to deduce, by interpolation argument, that 
\begin{equation*}
	\norm {\nabla \bar \psi_{b}}_{L^n(\mathbb{R}^2)} \lesssim \norm {\nabla \bar \psi_{b}}_{L^2(\mathbb{R}^2)} + \norm {\nabla^2 \bar \psi_{b}}_{L^2 (\mathbb{R}^2)} \lesssim _b \norm {f}_{L^1((1+|x|)dx;\mathbb{R}^2) \cap L^2(\mathbb{R}^2)} ,
\end{equation*}
for all $n\in [2,\infty)$. Accordingly, this new bound combined with the assumptions on $b$ imply that the right-hand side of \eqref{ELQ:000} belongs, after all, to $ L^{m}(\mathbb{R}^2)$, for all $m\in (1,p]$, with 
\begin{equation*}
	\norm { b\bar f - \frac{\nabla b}{b} \cdot \nabla \bar \psi_{b}}_{L^{m}(\mathbb{R}^2)} \lesssim_b \norm{f}_{L^1((1+|x|)dx;\mathbb{R}^2) \cap L^{m}(\mathbb{R}^2)},
\end{equation*} 
where we have employed  H\"olders' inequalities.
At last, by means of classical elliptic regularity of Poisson's equation \eqref{ELQ:000}, we conclude that 
\begin{equation}\label{D2psi}
		\norm{\nabla ^2 \bar \psi_{b}}_{ L^m(\mathbb{R}^2)} \lesssim_b \norm f_{L^1((1+|x|)dx;\mathbb{R}^2) \cap L^m(\mathbb{R}^2) }, 
	\end{equation}
	for all $m\in (1,p]$, thereby completing the proof of the proposition.
		\end{proof}

From the preceding proof, and by virtue of direct Sobolev embeddings, we can observe that  
	\begin{equation*}
	\norm{\nabla \psi_{b,\rad }}_{ L^\kappa (\mathbb{R}^2)} \lesssim_b \norm f_{L^1\cap L^2(\mathbb{R}^2)} , \qquad \norm{\nabla \bar \psi_{b}}_{ L^\kappa(\mathbb{R}^2)} \lesssim_b \norm f_{L^1((1+|x|)dx;\mathbb{R}^2) \cap L^2(\mathbb{R}^2)} 
	\end{equation*}  
	for all $\kappa\in (2,\infty)$, and
	\begin{equation}\label{bound:bar:psi:2}
		\norm{\nabla \bar \psi_{b}}_{ L^2   (\mathbb{R}^2)} \lesssim_b \norm f_{L^1((1+|x|)dx;\mathbb{R}^2) \cap L^2(\mathbb{R}^2)}. 
	\end{equation}
The validity of the last bound for $\psi_{b,\rad}$ requires that $f$ has mean-free. A  complementary  discussion on that is deferred to  Section \ref{section:compatibility}, below.

Our next concern is to show that any solution of \eqref{EL:P1} is Lipschitz-continuous and bounded. This corresponds to the endpoint case $\kappa=\infty$ in the previous bounds, which have been useful in the spectral analysis the time-periodic solutions constructed in previous sections.
 \begin{lem}\label{lemma:lipschitz}
	Let $b$ be a $C^1$ function satisfying assumptions \eqref{Hb2:eq}, \eqref{Hb3:eq} and \eqref{Hb4:eq} and consider a   source term $f$ belonging to $ L^1\big (\big(1+|x|\big) dx; \mathbb{R}^2\big ) \cap L^p(\mathbb{R}^2)$, for some  $p\in (2,\infty)$. Then, any solution $\psi_b$  of the elliptic equation \eqref{EL:P1}  is Lipschitz-continuous and enjoys the bound 
		\begin{equation*}
		\norm{\nabla \psi_b}_{L^\infty(\mathbb{R}^2)}\lesssim_b \norm f_{L^1((1+|x|)dx;\mathbb{R}^2) \cap L^p(\mathbb{R}^2)}.
		\end{equation*}
\end{lem}

\begin{proof}
	We begin with decomposing the solution as in the proof of Proposition \ref{prop:Elliptic:00} by writing 
	$$\psi_b=\psi_{b,\rad}+\bar \psi_b,$$
	 where $\psi_{b,\rad} $ is defined in \eqref{psi:ODE:sol} as the solution of \eqref{EL:P1} with source term $f_{\rad }$, introduced in  \eqref{RAD:F}, whereas  $\bar \psi_{b }$ solves the elliptic equation \eqref{EL:P1}  with source term  $\bar f \bydef  f-f_\rad $. Therefore, from its explicit expression \eqref{psi:ODE:sol}, it is readily seen that $\psi_{b,\rad }$ enjoys the bound
	\begin{equation*}
		|\nabla \psi_{b,\rad} (r)|\leq \frac{1}{r}\int_0^r |\tau f_\rad (\tau) |d\tau ,
	\end{equation*} 
	for all $r>0$. Hence, it follows, by employing H\"older's inequality, that  
	\begin{equation*}
		\begin{aligned}
			|\nabla \psi_{b,\rad} (r)|
			&\leq  \frac{1}{r} \left( \int_0^r \tau  d\tau\right)  ^\frac{1}{2} \left(  \int_0^r | f_\rad (\tau) |^2 \tau d\tau\right)   		^\frac{1}{2}	\\
			& \lesssim  \norm {f}_{L^2(\mathbb{R}^2)},
		\end{aligned}
	\end{equation*} 
	for all $r>0$, whereby the boundedness of $\nabla \psi_{b,\rad }$ follows. Next, we take care of $\bar \psi_b$. To that end, we recall from \eqref{D2psi} and \eqref{bound:bar:psi:2}   that
		\begin{equation}\label{bound:AAA:AAA}
		   \norm { \nabla^2 \bar \psi_b}_{L^p} + \norm {\nabla \bar \psi_b}_{H^1 (\mathbb{R}^2)} \lesssim_b \norm f_{L^1((1+|x|)dx;\mathbb{R}^2) \cap L^p(\mathbb{R}^2)}.
	\end{equation}
	Thus, as $\nabla \bar \psi_b$ belongs to $L^2(\mathbb{R}^2)$, writing  that 
	\begin{equation*}
		\nabla \bar \psi_b = S_0 (\nabla \bar \psi_b) + (\id - S_0) (\nabla \bar \psi_b),
	\end{equation*}
	one then sees, for $p>2$, that 
	\begin{equation*}
		\begin{aligned}
			\norm {\nabla \bar \psi_b}_{B^0_{\infty,1}} 
			&\lesssim  \norm {\nabla \bar \psi_b}_{ H^1} + \sum_{j\geq 0} 2^{ j(-1+\frac{2}{p})} \norm { \nabla^2 \Delta_j \bar \psi_b}_{L^p}
			\\
			& \lesssim  \norm {\nabla \bar \psi_b}_{ H^1} +     \norm { \nabla^2 \bar \psi_b}_{L^p}.
		\end{aligned}
	\end{equation*}
	Here, $S_0$ and $\{\Delta_j\}_{j\geq 0}$ denote the classical dyadic blocks from Littlewood-Paley theory. 
	Therefore, by employing the bounds \eqref{bound:AAA:AAA}, we deduce that 
	 \begin{equation*}
		\norm {\nabla \bar \psi_b}_{B^0_{\infty,1}} 
		\lesssim_b \norm f_{L^1((1+|x|)dx;\mathbb{R}^2) \cap L^p(\mathbb{R}^2)},
	\end{equation*}
	thereby concluding the proof of the Lipschitz estimate of $\psi_b$ by combining the preceding bound with the control on $\psi_{b,\rad}$. Moreover, it is obvious that $\nabla \psi_{b,\rad}$ is continuous, whereas the continuity of $\nabla \bar \psi_b$   follows from the properties of functions belonging to the Besov space $B^0_{\infty,1}(\mathbb{R}^2) $. This completes the proof of the lemma. 
	\end{proof}

	Observe that the preceding lemma ensures in particular that any solution of \eqref{EL:P1} with a   source term belonging to $L^1((1+|x|)dx;\mathbb{R}^2) \cap L^p(\mathbb{R}^2)$, for some $p>2$, is locally bounded. This is indeed a consequence of   Taylor's expansion
	\begin{equation*}
		|\psi_b (x)-\psi_b (y)| \leq |x-y| \norm {\nabla \psi_b}_{L^\infty(\mathbb{R}^2)}, 
	\end{equation*}
	for any $x,y\in \mathbb{R}^2$, which implies the continuity of $\psi_b$, whence its local boundedness, too.  Let us now derive more precise bounds on  $\psi_b$, which turn out to solely  require  an $L^2(\mathbb{R}^2)$ control of a compactly supported source term.

	\begin{lem}\label{lemma:boundedness:psi}
		Let $b$ be a $C^1$ function satisfying assumptions \eqref{Hb2:eq}, \eqref{Hb3:eq} and \eqref{Hb4:eq}. Further consider a compactly supported function $f$  belonging to $L^2(\mathbb{R}^2)$. Then, any solution $\psi_b$  of the elliptic equation \eqref{EL:P1}  enjoys the bound 
	\begin{equation*}
		\norm {\psi_b - \psi_{b,rad} - c}_{L^\infty(\mathbb{R}^2)} \lesssim_b  \norm{f}_{L^2(\mathbb{R}^2)},
	\end{equation*}
	for some constant $c\in \mathbb{R}$, where $\psi_{b,\rad}$ is given by \eqref{psi:ODE:sol} and satisfies the bound 
	\begin{equation*}
		\norm {\psi_{b,\rad}}_{L^\infty (B_A)} \lesssim_b  A\norm{f}_{L^2(\mathbb{R}^2)},
	\end{equation*}
	for any ball $B_A$ of finite radius $A>0$.
	\end{lem}

	\begin{proof} We begin with the control of $\psi_{b,\rad} $ which directly follows from its   explicit expression. More precisely, we obtain from \eqref{psi:ODE:sol}  that 
		\begin{equation*}
		\begin{aligned}
			|\psi_{b,\rad }(r)|
			&\leq  r  \sup_{0<\sigma\leq r} \left( \frac{1}{\sigma}\int_0^\sigma |\tau f_\rad (\tau) |d\tau\right) ,
		\end{aligned}
	\end{equation*}
	for all $r>0$. 
	Therefore, the boundedness of $\psi_{b,\rad }$ follows by applying  H\"older's inequality, as is shown in the proof of the preceding lemma, to find that	\begin{equation*}
		\begin{aligned}
			|\psi_{b,\rad }(r)|
			&\lesssim   r \norm {f}_{L^2(\mathbb{R}^2)} ,
		\end{aligned}
	\end{equation*}
	for all $r>0$.  Now, we turn our attention to the control of $\bar \psi_{b} \bydef \psi_b - \psi_{b,\rad }$ by first recalling   that it is governed by the equation  
	\begin{equation*}
		-\div\left(\frac{1}{b}\nabla \bar \psi_b \right)= \bar f,
	\end{equation*} 
	where $\bar f $ is mean-free. More precisely, this source term satisfies that
	\begin{equation}\label{bar_f:int}
		\int_0^{2\pi} \bar f\big(r(\cos \theta, \sin \theta) \big) d\theta = 0,
	\end{equation} 
	for all $r>0$. In light of this, it follows that the function 
	\begin{equation*}
		r \mapsto \bar\phi(r) \bydef  \frac{1}{2\pi }\int_0^{2\pi }\bar \psi_b \big (r(\cos \theta , \sin  \theta)\big ) d\theta
	\end{equation*}
	is a radial solution of the same elliptic equation, i.e.
	\begin{equation*}
		-\div\left(\frac{1}{b}\nabla \bar \phi  \right)= 0,
	\end{equation*} 
	whence is constant. The rigorous justification of that is a consequence of the analysis laid out in  a subsequent section (see in particular Theorem \ref{thm:EL:1}, below). One deduces then that 
	\begin{equation}\label{phi=0}
		 \frac{1}{2\pi }\int_0^{2\pi } \partial_r \bar \psi_b \big (r(\cos \theta , \sin  \theta)\big ) d\theta =  \frac{d}{dr} \bar \phi \equiv 0.
	\end{equation}
	Therefore, one sees, by expanding the equation governing $\bar \psi_b$ that this latter obeys the Poisson equation
	\begin{equation*}
		- \Delta \bar \psi_b = b  \bar f  -    \frac{b'}{b  }  \partial_r   \bar \psi_b \bydef  \bar F ,
	\end{equation*}	
	where the new source term $\bar  F$ is mean-free. More precisely, one can easily check, by virtue of \eqref{bar_f:int} and \eqref{phi=0}, that 
	\begin{equation*}
		 \int_0^{2\pi }\bar F \big (r(\cos \theta , \sin  \theta)\big ) d\theta =  0,
	\end{equation*}
	for all $r> 0$. Thus, we emphasize that  Poisson's equation 
	\begin{equation*}
		- \Delta u= \bar F
	\end{equation*}
	admits a  unique solution $u$ in $\dot H^1$. In other words, all  solutions  with finite square-integrable gradient coincide up to constants. This will be justified in a subsequent section (see Theorem \ref{thm:EL:1}, again, for $b\equiv 1$). As a consequence, owing to the fundamental representation of solutions of Poisson's equation, we deduce that there is $c\in \mathbb{R}$ such that 
	\begin{equation*}
		 \begin{aligned}
		 	\bar \psi_b (x) - c 
		 	&=- \frac{1}{2\pi } \int_{\mathbb{R}^2} \log |x-y| \bar F(y) dy ,
		 \end{aligned}
	\end{equation*}
	for all $x\in \mathbb{R}^2$. We now focus  our attention on the boundedness of the function in right-hand side. To that end,   we first notice that $\bar F$ is compactly supported, thanks to the assumptions on the source term $f$ and the coefficient $b$. Thus, it holds that 	\begin{equation*}
		\bar \psi_b (x) - c = \frac{1}{2\pi } \int_{ |y|\leq A} \log |x-y| \bar F(y) dy ,
	\end{equation*}		
	 for any $A >0 $ with $ \supp f \subset B_A$ and for all $x\in \mathbb{R}^2$. Hence, by further employing the fact that $\bar F$ is mean-free, we obtain that 
	\begin{equation*}
		 \begin{aligned}
		 	\Big| \bar \psi_b (x) - c\Big|  
		 	&=  \frac{1}{2\pi }\left| \mathds{1}_{|x|\leq A+1} \int_{|y|\leq A} \log |x-y| \bar F(y) dy  + \mathds{1}_{|x|> A+1} \int_{|y|\leq A} \log \left(\frac{|x-y| }{|x|} \right)\bar F(y) dy  \right|
		 	\\
		 	& \lesssim_A \norm {\bar F}_{L^2(\mathbb{R}^2)},
		 \end{aligned}
	\end{equation*}
	for all $x\in \mathbb{R}^2$, where we utilized the elementary facts that
	\begin{equation*}
		(x,y) \mapsto \log |x-y| \mathds{1}_{|x|\leq A+1} \mathds{1}_{|y|\leq A} \in L^\infty  \big(\mathbb{R}^2_x;L^2  (\mathbb{R}_y^2)\big)
	\end{equation*}
	and 
	\begin{equation*}
		(x,y) \mapsto \log |x-y|  \mathds{1}_{|x|>  A+1}\mathds{1}_{|y|\leq A}  \in L^\infty  (\mathbb{R}^2\times \mathbb{R}^2).
	\end{equation*}
	It then remains to control $\bar F$ in $L^2(\mathbb{R}^2)$ which can be established by employing H\"older's inequality to infer that 
	\begin{equation*}
		\norm { \bar F}_{L^2(\mathbb{R}^2)} \lesssim_b   \norm {f}_{L^2(\mathbb{R}^2)} + \norm {\nabla \bar \psi _b}_{L^2(\mathbb{R}^2)}.
	\end{equation*}
	Consequently, by a further use of \eqref{bound:bar:psi:2}, we deduce that 
	\begin{equation*}
			\norm { \bar \psi_b  - c}_{L^\infty(\mathbb{R}^2)}\lesssim_b   \norm {f}_{L^2(\mathbb{R}^2)},
	\end{equation*}
	which is the expected control of $ \bar \psi_b$, thereby completing   the proof of the lemma.
	\end{proof}

\subsection{Existence and uniqueness of solutions}\label{section:A:existence}
In this section, we show the existence and uniqueness, in a sense that will be precised later, of solutions to the elliptic problem \eqref{EL:P1} in the whole domain $\mathbb{R}^2$. Here, we only discuss the  case of a variable coefficient $b$ obeying certain conditions that are directly connected to our interest in the present paper. We refer to \cite{OS12} for a nice and self-contained notes where the content of the first paragraph of this section is further presented in more details and full generality, along with   several examples and discussions on relevant topics.

\subsubsection{Functional framework and introduction to the equations}\label{sec:Sobolev}
To get started, allow us   to clarify a few things about the functional space $\dot {H}^1(\mathbb{R}^2)$ that we were using in the previous few sections as well as the appropriate far--boundary conditions that we should supplement \eqref{EL:P1} with as $|x|\to \infty$. 
Henceforth, we follow the notations from \cite{OS12} by  introducing, for any measurable function $u$ on $\mathbb{R}^2$, its equivalence class defined by
\begin{equation*}
	[u]\bydef \{ u+c : c\in \mathbb{R} \}.
\end{equation*}
Denoting by $\mathcal{D}(\mathbb{R}^2)$ the space of smooth and compactly supported functions on $\mathbb{R}^2$, we further introduce its equivalence class
\begin{equation*}
	\dot{\mathcal{D}}(\mathbb{R}^2) \bydef \big\{ [u] : u \in \mathcal{D}(\mathbb{R}^2) \big\}.
\end{equation*}
Therefore, the Sobolev space $\dot{H}^1(\mathbb{R}^2)$ is defined as
\begin{equation*}
	\dot {H}^1(\mathbb{R}^2) \bydef \big\{[u]: u\in H^1_\loc(\mathbb{R}^2)  \quad \text{and} \quad \nabla u \in L^2(\mathbb{R}^2) \big\},
\end{equation*}
 which can be shown to coincide with the space corresponding to the completion of $ \dot{\mathcal{D}}(\mathbb{R}^2)$ by the norm $\norm {[\cdot]}_{\dot H^1(\mathbb{R}^2)}$ that we now define  as 
\begin{equation*}
	\norm {[u]}_{\dot H^1(\mathbb{R}^2)} \bydef \norm {\nabla u}_{L^2(\mathbb{R}^2)},
\end{equation*}
for any $u \in \dot H^1(\mathbb{R}^2)$.  The justification of the fact that the preceding mapping is a norm on $\dot H^1(\mathbb{R}^2)$ is straightforward and omitted here. Note in passing that $\dot H^1(\mathbb{R}^2) $ is a Hilbert space when endowed with the inner product $$\langle \cdot,\cdot\rangle_{\dot H^1}\bydef \langle \nabla \cdot,\nabla \cdot\rangle_{L^2} .$$   
At last, the topological dual of $\dot {H}^1(\mathbb{R}^2)$ in the usual $L^2$-duality sense is denoted by $\dot {H}^{-1}(\mathbb{R}^2)$ which can be  characterized as
\begin{equation*}
	\dot{H}^{-1}(\mathbb{R}^2) = \big\{ \div F : F\in L^2(\mathbb{R}^2;\mathbb{R}^2) \big\},
\end{equation*}
where  $ \div F$ is to be understood in the distributional sense. In a similar fashion, one can define the homogeneous Sobolev space $\dot W^{1,p}(\mathbb{R}^2)$  and its corresponding dual space $\dot{W}^{-1,p}(\mathbb{R}^2)$, for any $p\in (1,\infty)$. 
One fundamental example that is relevant to the scope of this section is given in the following lemma.
\begin{lem}\label{lemma:duality} 
	Let $f$ be a mean-free function in $ L^1(\mathbb{R}^2)\cap L^2(\mathbb{R}^2)$. Assume further that $f$ has an integrable first moment, i.e., that
	\begin{equation*}
		\int_{\mathbb{R}^2} |xf(x)|dx <\infty.
	\end{equation*} 
	Then, $f$ belongs to $\dot{H}^{-1}(\mathbb{R}^2)$. More precisely, there exists $C>0$ such that
	\begin{equation*}
		\left| \int_{\mathbb{R}^2} f(x) u(x)dx\right| \leq C \norm {\nabla u}_{L^2(\mathbb{R}^2)}\norm f_{L^1((1+|x|)dx;\mathbb{R}^2) \cap L^2(\mathbb{R}^2)},
	\end{equation*}
	for all $u\in \dot{H}^1(\mathbb{R}^2)$.
\end{lem} 
The proof uses Fourier analysis and is detailed in \cite{OS12}. Note in passing that, as a direct consequence of the previous lemma, if $f$ is  compactly supported, mean-free function and belongs to $L^2(\mathbb{R}^2)$, then it lies in $ \dot{H}^{-1}(\mathbb{R}^2)$. In particular, it holds that 
\begin{equation*}
	\left|\int_{\mathbb{R}^2}  f u (x) dx\right| \lesssim \norm {f}_{L^2(\mathbb{R}^2)} \norm {\nabla u}_{\dot{H}^1(\mathbb{R}^2)},
\end{equation*}
for any $u\in \dot H^1(\mathbb{R}^2)$.  

Now, we want  to discuss the existence of weak solutions (we will precise the sense of that in a moment) of the elliptic problem
\begin{equation}\label{EL:AAA1} 
		\left\{
		\begin{aligned}
			\displaystyle-\div\left(\frac{1}{b}\nabla u \right) 
			& =f, \quad \text{for} \quad x\in \mathbb{R}^2,
			 \\
			u &\sim 0, 
			\quad \text{as} \quad |x|\to \infty,
		\end{aligned} 
		\right.
\end{equation}  
where the   far--boundary conditions are understood in the sense of the equivalence class defined at the beginning of this section, $b$ is a fixed non-negative function and the source term $f$ is a given function whose properties will be made precise, later on. 

\subsubsection{Compatibility of the far--boundary conditions and the source term}\label{section:compatibility}
Before we step over the key elements to showing the existence of weak solutions of \eqref{EL:AAA1}, allow us first to highlight a very important observation that connects the far--boundary conditions in \eqref{EL:AAA1} with admissible properties (or let us say necessary assumptions) of the source term $f$ so that the first and the second lines in that system are, in some sense, compatible. 

To see that, we now assume that the source term is a radial function and we look for a radial solution of \eqref{EL:AAA1}. As this is already done in the foregoing sections, we just recall that, in this case, the solution of the first line in \eqref{EL:AAA1} is, up to constants, given by its explicit expression
\begin{equation*} 
	u_{\text{rad}}(r)= - \int_0^r \frac{b(\sigma)}{\sigma} \int_0^\sigma \tau f(\tau) d\tau d\sigma,
\end{equation*} 
for all $r >0$. We are now going to show that, if $f$ has a non-vanishing mean (which is also assumed to be a positive function here for simplicity, though the hypothesis on its sign   can completely be removed), then the far--boundary condition in \eqref{EL:AAA1} as $r\to \infty$ are strongly violated by the solution  $u_{\text{rad}}$. To that end, under the preceding assumptions on   $f$, it is readily seen that there is $A>0$ such that, for any $\sigma\geq A$, it holds that
\begin{equation*}
	 \int_0^\sigma \tau f(\tau) d\tau  > \frac{ m }{4\pi} >0,
\end{equation*} 
where we denoted 
\begin{equation*}
	m\bydef \int_{\mathbb{R}^2} f(x)dx = 2\pi \int_0^\infty \tau f(\tau) d\tau.
\end{equation*}
it then follows that 
\begin{equation*}
	\begin{aligned}
		\left| u_{\text{rad}}(r) \right|
		  & \geq   \int_A^r \frac{b(\sigma)}{\sigma} \int_0^\sigma \tau f(\tau) d\tau d\sigma 
		  \geq \left( \frac{m}{4\pi } \inf_{\sigma>0}b(\sigma) \right) \int_A^r \frac{d\sigma}{\sigma}\geq \left( \frac{m}{4\pi } \inf_{\sigma>0}b(\sigma) \right)\log r ,
	\end{aligned}
\end{equation*}
	for all $r>A$. This clearly leads to the unboundedness of $u_{\text{rad}}$ at infinity, thereby violating the far--boundary conditions in \eqref{EL:AAA1}. Note that the preceding $log$-growth is optimal as, it is by a direct estimate, that one can show, moreover, that
	\begin{equation*}
		|u_\rad (r)| \leq O(1) + \left( \frac{m}{2\pi } \sup_{\sigma>0}b(\sigma) \right) \log r ,
	\end{equation*}
	for all $r>0$.
	This particular example of solutions shows that the finite far--boundary conditions are not compatible with source terms with non-vanishing mean. 
	This is not the end of the story for \eqref{EL:AAA1} as we will show in the next section that this problem admits a unique weak solution in an affine space which is made of shifting $\dot{H}^1(\mathbb{R}^2)$ by the typical ``bad'' solution $u_{\text{rad}}$. 
	
	At the end, allow us to conclude this section by mentioning that the consideration of the space $ \dot{H}^1(\mathbb{R}^2) + u_{\text{rad}}$ is strongly motivated by the seminal work of Chemin \cite{CH95} (see in particular Lemma 1.3.1 in \cite{CH95}) where he sheds light on a similar observation that permits to uniquely solve \eqref{EL:AAA1} in the whole domain $\mathbb{R}^2$ and in the case $b\equiv 1$. More precisely, it is shown in \cite{CH95} that a divergence free velocity field is uniquely derived from a given two-dimensional vorticity $\omega$ and a ``stationary'' vector field $\sigma$  in  the affine space $L^2(\mathbb{R}^2; \mathbb{R}^2 ) + \sigma$. In the next section, we provide a generalization of  that in the case of variable coefficient $b$.

\subsubsection{Existence and uniqueness of weak solutions} Here, we shall establish  a well-posedness result for \eqref{EL:AAA1}. To that end,  allow us first to recall the notion of weak solutions that we are concerned with. We say that $u$ is a weak solution to \eqref{EL:AAA1}, for a fixed coefficient $b>0$ and a given source term $f$ in $L^2(\mathbb{R}^2)$, if it belongs to $\dot{H}^1(\mathbb{R}^2)$ and satisfies the following weak reformulation  
\begin{equation*} 
	\int_{\mathbb{R}^2} \frac{1}{b(x)} \nabla u(x)\cdot \nabla v(x)dx = \int_{\mathbb{R}^2} f(x) v(x)dx ,
\end{equation*} 
for every test function $v\in \mathcal{D}( \mathbb{R}^2)$.

The central result of this section is now stated in the next theorem which builds on \cite[Theorem 3.1]{OS12} and whose elements of proof rely on the ingredients from the previous sections.   
\begin{thm}\label{thm:EL:1}
	Let $b$ be a non-negative, radially symmetric function in $C^1(\mathbb{R}^2)$ and  $f$ be a function in $L^1(\mathbb{R}^2)\cap L^2(\mathbb{R}^2)$.  Assume further that the first moment of $f$ is integrable, as well. Then, there is a unique weak solution $u$ of the elliptic problem \eqref{EL:AAA1} with source term $f$ in the affine space $ \dot{H}^1(\mathbb{R}^2) + u_{\text{rad}}$, where $u_{\text{rad}}$ is given by 
	\begin{equation*}
		u_{\text{rad}} (r)\bydef  - \int_0^r \frac{b(\sigma)}{\sigma} \int_0^\sigma \tau f_{\text{rad}}(\tau) d\tau d\sigma,
	\end{equation*}
	for all $r\geq 0$, and $f_{\text{rad}}$ stands for the radially symmetric part of $f$, which is defined by 
	\begin{equation*}
		f_{\text{rad}}(r) \bydef \frac{1}{2\pi} \int_0^{2\pi} f\big (r(\cos \theta, \sin \theta)\big ) d\theta,
	\end{equation*} 
	for all $r\geq 0$. Moreover, it holds that 
	\begin{equation*}
		\norm {\nabla (u - u_{\text{rad}})}_{L^2(\mathbb{R}^2)}\lesssim_b \norm f_{L^1((1+|x|)dx;\mathbb{R}^2) \cap L^2(\mathbb{R}^2)}.
			\end{equation*} 
\end{thm}

\begin{proof}Observe that $u_{\text{rad}}$ belongs to $C^1_{\text{loc}}(\mathbb{R}^2)\cap \dot {H}^2(\mathbb{R}^2)$, which can be justified along the same lines as  is shown in the proofs of Proposition \ref{prop:Elliptic:00} and Lemma \ref{lemma:boundedness:psi} above, and solves \eqref{EL:AAA1} with source term $f_\rad$. Thus, it is readily seen that $u_{\text{rad}}$ is also a solution of the variational problem
\begin{equation*}
	\int_{\mathbb{R}^2} \frac{1}{b(x)} \nabla u_{\text{rad}}(x)\cdot \nabla v(x)dx = \int_{\mathbb{R}^2} f_\rad(x) v(x)dx ,
\end{equation*} 
for every test function $v\in \mathcal{D}( \mathbb{R}^2)$. Therefore, we now need to show the existence of a unique solution $\bar u \bydef u- u_{\text{rad}}$ in $\dot{H}^1(\mathbb{R}^2)$ of the equation
	\begin{equation}\label{FF:2}
	\int_{\mathbb{R}^2} \frac{1}{b(x)} \nabla \bar u(x)\cdot \nabla v(x)dx = \int_{\mathbb{R}^2} \bar f(x) v(x)dx ,
\end{equation} 
for all test function $v\in \mathcal{D}( \mathbb{R}^2)$, where we set $\bar f\bydef f-f_\rad $. It is straightforward to check that $\bar f$ is a mean-free function that satisfies all the assumptions in Lemma \ref{lemma:duality}; whence, it follows that 
\begin{equation*}
		\left| \int_{\mathbb{R}^2} \bar{f}(x) v(x)dx\right| \lesssim \norm f_{L^1((1+|x|)dx;\mathbb{R}^2) \cap L^2(\mathbb{R}^2)} \norm {\nabla v}_{L^2(\mathbb{R}^2)} ,
	\end{equation*}
	for any test function $v\in \mathcal{D}( \mathbb{R}^2)$. That is to say that 
	\begin{equation*}
		v \mapsto \int_{\mathbb{R}^2}\bar{f}(x) v(x)dx
	\end{equation*}
	belongs to $\dot{H}^{-1}$. On the other hand, it is readily seen that the bilinear operator 
	\begin{equation*}
		 (v_1,v_2) \mapsto \int_{\mathbb{R}^2} \frac{1}{b(x)} \nabla v_1(x)\cdot \nabla v_2(x)dx 
	\end{equation*}
	defines a scalar product on the Hilbert space $\dot H^1(\mathbb{R}^2)$ which is equivalent to $ \langle \cdot,\cdot\rangle_{\dot H^1}$, due to the assumptions on the function $b$. At last, the existence of a unique solution $w$ of \eqref{FF:2}  in $\dot H^1(\mathbb{R}^2)$ is achieved by a direct application of Reisz's representation theorem, thereby concluding the proof of the theorem.
\end{proof}

\subsection{A notion of self-adjointness}\label{section:A:selfAD}
Here, we establish a property of solutions of \eqref{EL:P1} that is, in some sense, a soft version of the self-adjointness of the elliptic operator \begin{equation*}
	  -\div \left( \frac{1}{b} \nabla \cdot \right).
\end{equation*} 
The details of that are given in the next lemma.
 \begin{lem}\label{lemma:SA} 
	 For any given test functions $\omega_1$ and $\omega_2$   in $\mathcal{D}(\mathbb{R}^2)$ with  
\begin{equation}\label{mean-omega_i}
	\int_{\mathbb {R}^2} \omega_1 (x)dx = \int_{\mathbb {R}^2} \omega_2 (x)dx= 1,
\end{equation}
	  there is a couple of solutions of \eqref{EL:AAA1} with source terms $\omega_1$ and $\omega_2$, denoted by $g_1$  and  $g_2$, respectively, such that 
	 \begin{equation}\label{self:AD}
	 	\int_{\mathbb{R}^2} g_1 \omega_2(x)dx = \int_{\mathbb{R}^2} \omega_1 g_2 (x)dx.
	 \end{equation} 
	 In addition to that, any couple of solutions $(h_1,h_2)$ of \eqref{EL:P1} with respective source terms $ (\omega_1,\omega_2)$ is also a solution of \eqref{self:AD} if and only if 
	 \begin{equation*}
	 	\lim_{|x|\to \infty }(h_1-h_2)(x)=0.
	 \end{equation*}
	 \end{lem}

\begin{rem}
	We emphasize that the preceding lemma can be generalized to hold for any test functions $\omega_1$ and $\omega_2$ with non-vanishing means that are not necessarily identical. This will hold up to a suitable rescaling argument. As this is not going to serve in our paper, we only restrict our selves to the case of test functions with unitary means. 
 \end{rem}
 \begin{rem}
 	Note that, under the non-vanishing assumptions on the means of the test functions \eqref{mean-omega_i}, Lemma \ref{lemma:SA} above provides us with a characterization of all admissible solutions $(g_1,g_2)$ of \eqref{EL:P1} and \eqref{self:AD}. Notice, furthermore, that such characterization is clearly not valid if one of the test functions (or both of them) is mean-free. Indeed, in such situations, \eqref{self:AD} is invariant by translating solutions by any constants.
 \end{rem}
 
\begin{proof} We begin with introducing the unique couple of  solutions that satisfies the vanishing far-boundary solutions by setting
\begin{equation}\label{gi:def}
	g_i \bydef g_{i,rad}+\bar g_i,
\end{equation}
where,  for a fixed  real number  $A>0$  such that 
$$\supp \omega_1, \supp \omega_2 \subset B_A ,$$
 we explicitly define $ g_{i,rad}$,    for all $r\geq 0$, by
\begin{equation*}
	 g_{i,rad}(r) \bydef - \int_A^r \frac{b(\sigma)}{\sigma} \int_0^\sigma \tau \omega_{i,\rad}(\tau) d\tau,
\end{equation*} 
 with 
\begin{equation*}
	 \omega_{i,\text{rad}}(r) \bydef  \frac{1}{2\pi} \int_0^{2\pi } \omega_i\big(r( \cos \theta,\sin \theta)\big) d\theta,
\end{equation*}
and  $ \bar g_i$ is defined as the unique solution of \eqref{EL:P1} with vanishing far-boundary condition
\begin{equation*}
	\lim_{|x|\to \infty} \bar g_i(x) = 0 ,
\end{equation*}
and source term
\begin{equation*}
	 f=\bar \omega_{i}  \bydef \omega_i -  \omega_{i,\text{rad}} .
\end{equation*} Note that the far-boundary conditions are given in classical sense since $ \bar g_i $ is continuous, thanks to Lemma \ref{lemma:lipschitz}. In this sense, the solution $g_i$ of \eqref{EL:P1} is uniquely defined, for any $i\in \{1,2\}$.

We now claim that $(g_1,g_2)$ is a solution    of \eqref{self:AD}. To show that, we write,  by definition of $g_1$ and $g_2$,   that
	\begin{equation*}
		\begin{aligned}
			\int_{\mathbb{R}^2} g_1 \omega_2(x)dx
			& =\int_{\mathbb{R}^2}\bar g_1   \bar \omega_{2 }(x)dx 
			+ 		 \int_{\mathbb{R}^2} g_{1,rad} \omega_{2,rad}(x)dx
			\\
			&\quad 
			 + \int_{\mathbb{R}^2} g_{1,rad} \bar \omega_{2 }(x)dx 
			+\int_{\mathbb{R}^2} \bar g_1  \omega_{2,rad}(x)dx  
			\\ 
			 &\bydef \sum_{j=1}^4 \mathcal{K}_j,
		\end{aligned} 
	\end{equation*}
	where each term in the preceding sum will be dealt with separately.
	\subsubsection*{About $\mathcal{K}_1$} The first term $ \mathcal
	{K}_1$ is the easiest one and can be dealt with by simple integrations by parts
	\begin{equation*}
		\begin{aligned}
			\mathcal{K}_1 
			& = - \int_{\mathbb{R}^2}  \bar g_1    \div \left( \frac{1}{b} \nabla \bar g_2  \right) (x)dx
			\\
			& =  \int_{\mathbb{R}^2}    \frac{1}{b} \nabla \bar g_1   \cdot \nabla \bar g_2  (x)dx
			\\
			& = - \int_{\mathbb{R}^2}  \div \left( \frac{1}{b} \nabla \bar g_1 \right)  \bar g_2 (x)dx
			\\
			& = \int_{\mathbb{R}^2} \bar \omega_{1}\bar g_2 (x)   dx,
		\end{aligned}
	\end{equation*} 
	where all the foregoing computations are justified because both $ \bar g_1 $ and $\bar g_2 $ belong to $\dot H^1(\mathbb{R}^2)$. 
	
	\subsubsection*{About $\mathcal{K}_2$} We turn our attention now to deal with $\mathcal{K}_2$ where we are going to employ  the support properties of the test functions $\omega_1$ and $\omega_2$ as well as the explicit expression of the solutions  $ g_{1,\rad}$ and $ g_{2,\rad}$ in the case where the source terms are radially symmetric. To see that, we write by   direct computations that  
	\begin{equation*}
		\begin{aligned}
			\mathcal{K}_2 =  - 2\pi  \int_{0}^A  \left(\int_A^r \frac{b(\sigma)}{\sigma} \int_0^\sigma \tau \omega_{1,\text{rad}}(\tau) d\tau d\sigma \right) r\omega_{2,rad}(r)dr.
		\end{aligned}
	\end{equation*}
	Thus, by means of two integration by parts, which are fully justified due to the fact that
	 $$ \frac{d}{dr}  g_{i,\rad} \in L^2_{\text{loc} }(\mathbb{R}^2), $$
	for all $i\in \{1,2\}$, we infer that 
	\begin{equation*}
		\begin{aligned}
			\mathcal{K}_2 
			& =  - 2\pi  \int_{0}^A  \int_A^r \frac{b(\sigma)}{\sigma}\int_0^\sigma \tau \omega_{1,\text{rad}}(\tau) d\tau  \frac{d}{dr} \left( \int_0^r \sigma \omega_{2,rad}(\sigma ) d\sigma  \right) dr
			\\
			& =    2\pi  \int_{0}^A   \frac{b(r)}{r}\int_0^r \tau \omega_{1,\text{rad}}(\tau) d\tau   \int_0^r \sigma \omega_{2,rad}(\sigma ) d\sigma  dr
			\\
			& =    2\pi  \int_{0}^A  \int_0^r \tau \omega_{1,\text{rad}}(\tau) d\tau      \frac{d}{dr}\left( \int_A^r \frac{b(\sigma)}{\sigma}\int_0^\sigma \tau  \omega_{2,rad}(\tau ) d\tau d\sigma  \right)  dr
			\\
			&=  - 2\pi  \int_{0}^A r \omega_{1,\text{rad}}(r)   \int_A^r \frac{b(\sigma)}{\sigma}\int_0^\sigma \tau  \omega_{2,rad}(\tau ) d\tau d\sigma    dr.
		\end{aligned}
	\end{equation*}
	Thus, we deduce that 
	\begin{equation*}
		\mathcal{K}_2=\int_{\mathbb{R}^2} \omega_{1,rad} \cdot  g_{2,\rad}(x)dx 	. 
	\end{equation*}

	\subsubsection*{About $\mathcal{K}_3$ and $\mathcal{K}_4$} We finally turn our attention  to compute the terms $\mathcal{K}_3$ and $\mathcal{K}_4$. to that end, observing from the definition of $\bar  \omega_{1}$ and $\bar  \omega_{2}$ that 	\begin{equation*}
		\int_0^{2\pi} \bar  \omega_{i} \big(r(\cos \theta, \sin \theta)\big) d\theta =0,
	\end{equation*}
	for all $r\geq0$, yields that 
	\begin{equation*}
		\mathcal{K}_3 = \int_0^A  g_{1,\rad}(r) \left( \int_0^{2\pi }\bar \omega_{2 }\big(r(\cos \theta, \sin \theta)\big) d\theta \right) rdr =0.
	\end{equation*}
	 	Moreover, noticing that 
	\begin{equation*}
		-\div\left(\frac{1}{b} \nabla \right) = - \partial_r \left( \frac{1}{b} \partial_r  \right) + \frac{1}{b}\frac{\partial_r}{r} + \frac{1}{b}\frac{\partial_{\theta\theta}}{r^2},
	\end{equation*}
	it is readily seen, since $ \bar g_i$ solves the equations
	\begin{equation*} 
		\left\{
		\begin{aligned}
			\displaystyle-\div\left(\frac{1}{b}\nabla\bar g_i  \right) 
			& =\bar  \omega_{i },  
			 \\
			\bar g_i (x)  &= 0, 
			\quad \text{as} \quad |x|\to \infty,
		\end{aligned} 
		\right.
\end{equation*}   
	for all $i\in \{1,2\}$, that  
	\begin{equation*}
		  \int_{0}^{2\pi} \bar g_i\big(r(\cos \theta, \sin \theta)\big) d\theta =0,
	\end{equation*}
		for all $r\geq 0$. Therefore, one sees  that 
		\begin{equation*}
			\mathcal{K}_3 =  \int_0^A    \left( \int_0^{2\pi } \bar g_1\big(r(\cos \theta, \sin \theta)\big) d\theta \right)   \omega_{2,rad} (r)rdr =0   .
		\end{equation*}
		All in all, gathering the previous identities yields the validity of \eqref{self:AD} with the particular choice of solutions $(g_1,g_2)$ introduced in \eqref{gi:def}.
		
		Now, let $(h_1,h_2)$ be another couple of solutions of \eqref{self:AD}. Thus, as $ h_i$ and $g_i$ are also solutions to the elliptic equation \eqref{EL:P1} with the same source term $\omega_i$, for all $i\in \{1,2\}$, then, by uniqueness from Theorem \ref{thm:EL:1},  there exist two real numbers $c_1$ and $c_2$ such that 
		\begin{equation*}
			(h_1,h_2)= \left( g_1 + c_1, g_2+c_2 \right).
		\end{equation*}
		Therefore, by plugging this identity in \eqref{self:AD}, we infer that $c_1=c_2$. Hence, we deduce that 
		\begin{equation*}
			h_1-h_2= g_1-g_2,
		\end{equation*}
		which yields that
		\begin{equation}\label{h1-h2:A}
			\lim_{|x|\to \infty}(h_1-h_2)(x)= \lim_{|x|\to \infty}(g_1-g_2)(x).
		\end{equation}
		Let us now show that $g_1-g_2$ satisfies the vanishing far-boundary conditions. Note that, by definition of $g_1$ and $g_2$ (see \eqref{gi:def}), this is reduced to showing that 
		\begin{equation}\label{sb1-sb2:A}
			\lim_{r\to \infty} \big(  g_{1,rad} -  g_{2,rad}  \big)(r) =0.
		\end{equation}
		To see that, we further notice that for large values of $r$, in particular for any  $r>A$, it holds that 
		\begin{equation*}
			g_{i,rad}(r)= \left(\int_A^r \frac{b(\sigma)}{\sigma} d\sigma \right) \left( \int_0^\infty \tau \omega_{i,\text{rad}}(\tau)d\tau\right)= \frac{1}{2\pi }\int_A^r \frac{b(\sigma)}{\sigma} d\sigma , 
		\end{equation*}
		for all $i\in \{1,2\}$, where we utilized the assumption that $\omega_1$ and $\omega_2$ are supported in the ball $B_A$ and with unitary means. Consequently, we deduce the validity of \eqref{sb1-sb2:A} which yields that \eqref{h1-h2:A} holds, as well.
		
		Conversely, we now assume that  $ h_1$ and $h_2$ are two solutions of the elliptic equation  \eqref{EL:P1} with respective source terms $\omega_1$ and $\omega_2$ such that 
		\begin{equation*}
			\lim_ {|x|\to \infty}(h_1-h_2)(x)=  0.
		\end{equation*}
		Therefore, again, by uniqueness of solutions to \eqref{EL:P1}, there exist two real numbers $c_1$ and $c_2$ such that 
		\begin{equation*}
			(h_1,h_2)= \left( g_1 + c_1, g_2+c_2\right).
		\end{equation*}
		Thus,  it follows  that
		\begin{equation*}
			c_1-c_2 = \lim _{|x|\to \infty} \Big( (h_1-h_2) - (g_1-g_2)\Big) (x)=0.
		\end{equation*}
		Hence, one deduces that 
		\begin{equation*}
			(h_1,h_2)= \left( g_1 + c, g_2+c\right),
		\end{equation*}
		for some constant $c\in \mathbb{R}$.  At last, by virtue of the assumption  \eqref{mean-omega_i},
		 one readily sees that 
		$(h_1,h_2)$ solves \eqref{self:AD} as soon as $(g_1,g_2)$ does so. This completes the proof of the lemma. 
\end{proof}

\subsection{Singular kernel and proof of Proposition \ref{prop.kernel}}\label{section:A:kernel}
Here, we shall provide a full justification of Proposition \ref{prop.kernel} which played a central role in the proof of our main theorems.   
\begin{proof}[Proof of Proposition \ref{prop.kernel}]
	The three bounds in the statement of the proposition are consequences of Proposition \ref{prop:Elliptic:00} and Lemma \ref{lemma:boundedness:psi}. Indeed, by a direct application of the estimates therein, we obtain the controls, for all $y\in \mathbb{R}^2$
	\begin{equation*}
	 \norm {\nabla ^2 S_b(\cdot ,y)}_{L^m(\mathbb{R}^2)} \lesssim_b \sqrt{b(y)} \norm {\log|\cdot -y|\Delta \left(\frac{1}{\sqrt{b }}\right)}_{L^m(\mathbb{R}^2)},
\end{equation*}
for all $m\in (1,p]$, and 
\begin{equation*}
	 \norm {\nabla  S_b(\cdot ,y)}_{L^n(\mathbb{R}^2)} \lesssim_b \sqrt{b(y)} \norm {\log|\cdot -y|\Delta \left(\frac{1}{\sqrt{b }}\right)}_{L^n(\mathbb{R}^2)},
\end{equation*}
for all $n\in (2,p]$, as well as
 \begin{equation*}
	  \norm {   S_b(\cdot ,y)}_{L^ \infty  (B_\rho)} \lesssim_b \sqrt{b(y)} \norm {\log|\cdot -y|\Delta \left(\frac{1}{\sqrt{b }}\right)}_{L^2 (\mathbb{R}^2)},
\end{equation*}
for any ball $B_\rho$ or radius $\rho>0$.
Therefore, by further employing assumption \eqref{Hb3:eq}, it follows that 
\begin{equation*}
	 \norm {\nabla ^2 S_b(\cdot ,y)}_{L^m(\mathbb{R}^2)} \lesssim_b \norm {b}_{C^2(\mathbb{R}^2)}^2 \norm {\log|\cdot -y| }_{L^m(B_{R_\infty})},
\end{equation*}
for all $m\in (1,p]$, and 
\begin{equation*}
	 \norm {\nabla  S_b(\cdot ,y)}_{L^n(\mathbb{R}^2)} \lesssim_b \norm {b}_{C^2(\mathbb{R}^2)}^2 \norm {\log|\cdot -y| }_{L^n(B_{R_\infty})},
\end{equation*}
for all $n\in (2,p]$, as well as
 \begin{equation*}
	  \norm {   S_b(\cdot ,y)}_{L^ \infty  (B_\rho)} \lesssim_b \norm {b}_{C^2(\mathbb{R}^2)}^2 \norm {\log|\cdot -y| }_{L^2 (B_{R_\infty})}.
\end{equation*} 
 Consequently, the claimed bounds on $S_b$ are obtained by taking the supremum on $y$ over any compact set in $\mathbb{R}^2$. 

Once $S_b$ is constructed, it is then by a direct computation that one can show that the function 
$$ x\mapsto \int_{\mathbb{R}^2} K_b(x,y) f(y)dy,$$
solves the elliptic equation \eqref{EL:P1}, where $K_b$ is as introduced in the statement of the proposition.
We are now left with the proof of the properties of $K_b$. Before doing so, let us notice that $K_b$ is determined up to constants. Thus, we now select a suitable representative and prove that this one satisfies all symmetries in the statement of the proposition. At the end, since all the symmetry properties are stable by shifting the solution by constants, the same properties will eventually remain satisfied by any kernel $K_b$.

The choice of the representative kernel draws insight from Lemma \ref{lemma:SA} and is done by first selecting the unique solution $S_b$ of \eqref{EL:P2}  by setting
\begin{equation}\label{Sb:def:AAA}
	S_b(x,y)= S_{b,rad}
	(|x|,y)+\bar S_{b}(x,y),
\end{equation}
where  $S_{b,rad}$  is explicitly defined by 
\begin{equation*}
	S_{b,rad}(r,y)\bydef - \int_{R_\infty}^r \frac{b(\sigma)}{\sigma} \int_0^\sigma \tau f _{ \text{rad}}(\tau,y) d\tau + \frac{1}{2\pi} \sqrt{b(R_\infty)b(y)} \log  R_\infty   ,
\end{equation*}
where $ R_\infty>0$ is the radius of the ball introduced in  \eqref{Hb3:eq}, and $f _{ \text{rad}}$ denotes the radially symmetric part of $f$ with respect to the $x$ variable, i.e., 
\begin{equation*}
	f _{ \text{rad}}(r,y) \bydef \frac{1}{2\pi }\int_0^{2\pi} f\big(r(\cos\theta,\sin \theta),y\big) d\theta,
\end{equation*}
where \begin{equation*}
	f(x,y) \bydef  \frac{1}{2\pi}\log|x-y|\sqrt{b(y)}\Delta_x\left(\frac{1}{\sqrt{b(x)}}\right) ,
\end{equation*}
 for all $x,y\in \mathbb{R}^2\times \mathbb{R}^2 $, with $y\neq x$, and where $\bar S_b(\cdot,y)$ is the unique solution of the elliptic equation \eqref{EL:P1} with source term $f-f_{\text{rad}}$ and vanishing far-boundary conditions, i.e.
 \begin{equation*}
 	\lim_{|x|\to \infty}\bar S_b(x,y) = 0,
 \end{equation*}
 for all $y\in \mathbb{R}^2$. Note in passing that 
 \begin{equation}\label{int:f}
 	\int_{\mathbb{R}^2} f(x,y) dx=  1 -\frac{\sqrt{b(y)}}{\sqrt{b(R_\infty)}} ,
 \end{equation}
for all $y\in \mathbb{R}^2$. Indeed, to see that, we first observe that  
\begin{equation*}
\begin{aligned}
  \int_{\mathbb{R}^2} f(x,y) dx =  \frac{\sqrt{b(y)}}{2\pi}	\int_{\mathbb{R}^2} \log|x-y|\Delta_x\left(\frac{1}{\sqrt{b(x)} } - \frac{1}{\sqrt{b(R_\infty)} }  \right)dx,
\end{aligned}
\end{equation*} 
for any $y\in \mathbb{R}^2$. Therefore, as the function  
\begin{equation*}
x\mapsto \frac{1}{\sqrt{b(x)} } - \frac{1}{\sqrt{b(R_\infty)} } 
\end{equation*} 
is continuous and compactly supported, thanks to the assumption \eqref{Hb3:eq}, it then follows that it belongs to $L^1\cap L^\infty(\mathbb{R}^2)$. One then sees, by a direct integration by parts, that
\begin{equation*}
	\begin{aligned}
		\int_{\mathbb{R}^2} f(x,y) dx
		& =  \sqrt{b(y)} \left( (-\Delta_y)  \left( -\frac{1}{2\pi}	\int_{\mathbb{R}^2} \log|x-y| \left(\frac{1}{\sqrt{b(x)} } - \frac{1}{\sqrt{b(R_\infty)} }  \right)dx \right)\right)
		\\
		& =   \sqrt{b(y)}    \left(\frac{1}{\sqrt{b(y)} } - \frac{1}{\sqrt{b(R_\infty)} }  \right), 
	\end{aligned}
\end{equation*}
where we utilized the fact that 
$$(x,y) \mapsto -\frac{1}{2\pi} \log|x-y|$$ is the fundamental solution of the Laplace operator $(-\Delta)$ in dimension two of space. This shows the validity of \eqref{int:f} which will come in handy later on. It is to be emphasized in a moment that the correction by the term
\begin{equation*}
	\frac{1}{2\pi} \sqrt{b(R_\infty)b(y)} \log R_\infty 
\end{equation*} 
in the definition of $S_{b,\text{rad}}$ is carefully chosen to compensate the behavior, as $|x|\to \infty$, of the first term in the right-hand side of \eqref{symmetry-kernel:split}. Indeed, in view of the very definition of $S_{b,\text{rad}}$ above, we emphasize    that 
\begin{equation}\label{S:rad:infinity}
	S_{b,\text{rad}}(r,y) = - \frac{b(R_\infty)}{2\pi} \log \left( \frac{r}{R_\infty}\right) + \frac{1}{2\pi} \sqrt{b(R_\infty)b(y)}  \log r  ,
\end{equation}
for all $r\geq R_\infty$ and $y\in \mathbb{R}^2$. To see that, it is enough to notice that
\begin{equation*}
	S_{b,rad}(r,y)= - \frac{b(R_\infty)}{2\pi } \left(  \int_{R_\infty}^r \frac{d\sigma}{\sigma} \right) \int_{\mathbb{R}^2} f (x,y) dx + \frac{1}{2\pi} \sqrt{b(R_\infty)b(y)} \log R_\infty   ,
\end{equation*}
for large values of $r\geq R_\infty$ and all $y\in \mathbb{R}^2$, where we employed the   assumption \eqref{Hb3:eq}, again. Therefore, the expansion \eqref{S:rad:infinity} follows by incorporating \eqref{int:f} in the preceding identity.  

Let us now move on to the proof of the remaining claims in the proposition. To that end, in term of the definition of $S_b$ above, we now select the kernel $K_b$ by setting
     \begin{equation*}
	K_b(x,y) \bydef -\frac{1}{2\pi}\log|x-y|\sqrt{b(x)b(y)} + S_b(x,y),
\end{equation*}
      and we emphasize, again, that it can be shown by a direct computations---which do not rely on the very fine definition of $S_b$ given in \eqref{Sb:def:AAA}, but only the fact that it solves \eqref{EL:P2}---that the function 
      \begin{equation*}
      	x\mapsto \int_{\mathbb{R}^2} K_b(x,y)\omega(y)dy
      \end{equation*}
      is a distributional solution of \eqref{EL:P1} with source term $\omega$, for any  $\omega \in \mathcal{D}(\mathbb{R}^2)$. Moreover, this distributional solution can be shown to satisfy the elliptic regularity estimates of Proposition \ref{prop:Elliptic:00}. We again skip the proof of that here and refer  to \cite{M23}.
      
      Now, we turn our attention  to prove the symmetry properties of $K_b$. To that end, due to the assumption \eqref{Hb4:eq}, observing that the function 
\begin{equation*}
	(x,y) \mapsto -\frac{1}{2\pi}\log|x-y|\sqrt{b(x)b(y)}
\end{equation*}
satisfies all the symmetry properties claimed in the statement of the proposition, then, we  only need to show that $S_b(x,y)$ does not perturb these symmetries. The principle step towards showing that consists in   justifying that the two solutions of \eqref{EL:P1} given by
\begin{equation*}
	\psi_b[\omega_i] \bydef \int_{\mathbb{R}^2} K_b(\cdot,y) \omega_i (y)dy, \quad \text{for} \quad i\in \{1,2\}
\end{equation*} 
are also solutions to equation \eqref{self:AD}. That is to say, we claim that  
\begin{equation}\label{claim:omega:i}
	\int _{\mathbb{R}^2} \omega_2 \psi_b[\omega_1](x)dx = \int _{\mathbb{R}^2} \omega_1 \psi_b[\omega_2](x)dx,
\end{equation}
for any test functions $ \omega_1 $ and $\omega_2$ in $\mathcal{D}(\mathbb{R}^2)$ with  
\begin{equation}\label{omega:i=1}
	\int_{\mathbb{R}^2} \omega_1 (x)dx = \int_{\mathbb{R}^2} \omega_2 (x)dx=1.
\end{equation}
Let us assume the validity of \eqref{claim:omega:i} for a moment and first continue the proof of \eqref{symmetry-kernel-0}. In view of \eqref{claim:omega:i} and the fact that 
\begin{equation*}
	\int_{\mathbb{R}^2\times \mathbb{R}^2} \log|x-y|\sqrt{b(x)b(y)} \omega_1(x) \omega_2(y) dxdy= \int_{\mathbb{R}^2\times \mathbb{R}^2} \log|x-y|\sqrt{b(x)b(y)} \omega_1(y) \omega_2(x)dxdy , 
\end{equation*}
it is readily seen that 
\begin{equation*}
	\int_{\mathbb{R}^2\times \mathbb{R}^2} S_b(x,y) \omega_1(x) \omega_2(y)dxdy = \int_{\mathbb{R}^2\times \mathbb{R}^2}S_b(x,y) \omega_1(y) \omega_2(x)dxdy , 
\end{equation*}
whereby we deduce that 
\begin{equation*}
	\int_{\mathbb{R}^2\times \mathbb{R}^2} \Big(  S_b(x,y)- S_b(y,x)\Big)  \omega_1(x) \omega_2(y)dxdy =0,
\end{equation*}
for all test functions $\omega_1$ and $\omega_2$ in $\mathcal{D}(\mathbb{R}^2)$ with unitary means. Thus, we arrive at the conclusion that 
\begin{equation*}
	S_b(x,y)=S_b(y,x),
\end{equation*}
for all $x,y\in \mathbb{R}^2$, whereby \eqref{symmetry-kernel-0} follows. 

We are now left with the proof of \eqref{claim:omega:i}. To that end, we notice, in view of the second part of Lemma \ref{lemma:SA}, that the couple $( \psi_b[\omega_1], \psi_b[\omega_2])$ of solutions satisfies \eqref{claim:omega:i} if and only if 
\begin{equation*}
	\lim_{|x|\to \infty} \Big(\psi_b[\omega_1] - \psi_b[\omega_2]\Big)(x)=0.
\end{equation*}
Thanks to the definitions of $ \psi_b[\omega_1]$ and $\psi_b[\omega_2] $ as well as  \eqref{omega:i=1}, it is readily seen that the preceding limit is a direct consequence of the  following  two properties:
\begin{equation}\label{claimH1}
	\lim_{|x|\to \infty } \mathcal{A}_1(x)\bydef \lim_{|x|\to \infty } \int _{\mathbb{R}^2} \bar S_b(x,y) h(y)dy = 0
\end{equation} 
and
\begin{equation}\label{claimH2}
	\lim_{|x|\to \infty } \mathcal{A}_2(x)\bydef \lim_{|x|\to \infty } \int _{\mathbb{R}^2} \left( -\frac{1}{2\pi } \log|x-y|\sqrt{b(x)b(y)}  + S_{b,\text{rad}}(x,y)\right) h(y)dy = 0,
\end{equation}
for any test function $h\in \mathcal{D}(\mathbb{R}^2)$ with 
\begin{equation*}
	\int_{\mathbb{R}^2}h(y)dy=0.
\end{equation*}
  The proof of \eqref{claimH1} directly follows by applying the    dominated convergence theorem of Lebesgue, for $\bar S_b(x,y)$ vanishes as $|x|\to \infty$ by construction. As for the justification of \eqref{claimH2}, we employ \eqref{S:rad:infinity} altogether with the assumption that $h$ is mean-free to obtain that 
  \begin{equation*}
  	\mathcal{A}_2(x) = -\frac{\sqrt{b(R_\infty)}}{2\pi } \int _{\mathbb{R}^2}   \log\left( \frac{ |x-y|}{|x|}\right) \sqrt{b(y)}   h(y)dy ,
  \end{equation*}
for all $|x|\geq R_\infty$. Therefore, noticing that the integral in the right-hand side above can be restricted to a compact set, as $h$ is compactly supported,  the dominated convergence theorem of Lebesgue yields \eqref{claimH2}, as well. This leads to the completion of the proof of \eqref{claim:omega:i} and, eventually, a justification of \eqref{symmetry-kernel-0}.

 We now show the stability by rotations and involution of $S_b(x,y)$, which, as previously pointed out, will imply \eqref{symmetry-kernel}. To that end, observing, in view of assumption \eqref{Hb4:eq},  that
\begin{equation*}
	\div_x\left(\frac{1}{b}\nabla_x\cdot\right)=\frac{1}{b}\left(\partial_{rr}+\frac{\partial_r}{r}+\frac{\partial_{\eta\eta}}{r^2}\right)+ \frac{d}{dr}\left(\frac{1}{b}\right)\partial_r,
\end{equation*} 
for any $x=r(\cos \eta, \sin \eta)\in \mathbb{R}^2\setminus \{0\}$, one sees that the preceding operator is invariant by the transformation $ x\mapsto \mathcal{R}_\theta \cdot x$, for any $\theta\in [0,2\pi]$. In addition to that, it is readily seen that the function 
$$(x,y)\mapsto\log|x-y|\sqrt{b(y)}\Delta\left(\frac{1}{\sqrt{b(x)}}\right)$$
is invariant by the transformation
$$(x,y)\rightarrow(\mathcal{R}_\theta\cdot x,\mathcal{R}_\theta\cdot y),  $$
for all $\theta\in[0,2\pi]$. Thus, 
by uniqueness of the solution  $\bar S_b$ of \eqref{EL:P2}, we deduce that 
$$\bar S_b(\mathcal{R}_\theta\cdot x,\mathcal{R}_\theta\cdot y)=\bar S_b(x,y),$$ 
for all $x,y\in \mathbb{R}^2$ and $\theta\in[0,2\pi].$ Likewise, the stability of the solution by involution can be proven along the same lines and by employing the same arguments as before. The details of that are left the reader. This leaves us with a justification of the \eqref{symmetry-kernel} for the part  $S_{b,\text{rad}}$. To that end, as $S_{b,\text{rad}}$ is radially symmetric with respect to the $x$ variable, it is enough to check that it is radially symmetric with respect to the $y$ variable, as well, to conclude the proof. To see that, we only need to observe that 
\begin{equation*}
f_{\text{rad}}(|x|,y)= \int _{0}^{2\pi} 	\log\Big||x|e^{\ii \theta}-|y|e^{\ii \eta}\Big|\sqrt{b(|y|)}\Delta_x\left(\frac{1}{\sqrt{b(|x|)}}\right) d\theta,
\end{equation*}
where we identity $\mathbb{R}^2$ with $\mathbb{C}$, for simplicity, to write $x=|x|e^{\ii \theta}$ and $y=|y|e^{\ii \eta}$, for some $\theta,\eta\in [0,2\pi]$. Therefore, by a change of variables, it is readily seen that 
\begin{equation*}
f_{\text{rad}}(|x|,y)= \sqrt{b(|y|)}\Delta_x\left(\frac{1}{\sqrt{b(|x|)}}\right) \int _{0}^{2\pi} 	\log\Big||x|e^{\ii \theta} -|y| \Big| d\theta,
\end{equation*}
which implies that $ f_{\text{rad}}$ is radial with respect to the $y$ variable, as well. Thus, it follows that $S_{b,\text{rad}}$ is radial in $x$ and $y$ variables and, whence, is stable by rotations and involution. This completes the proof of the proposition.
\end{proof}

\section{A key formula}

For the sake of completeness, we  conclude with a justification of a crucial classical identity that has been used several times in this paper. 
\begin{lem}
For any $n\in \mathbb{N}^*$, $x,y\in (0,\infty)$ and   $\theta\in[0,2\pi]$, it holds that  
 \begin{equation}\label{log:identity}
\frac{1}{2\pi}\int_0^{2\pi}\log|ye^{\ii \theta}-xe^{\ii \eta}|\cos(n\eta)d\eta=-\cos(n\theta)\frac{\mathtt{m}^n(x,y)}{2n}, 
\end{equation} 
where 
$$\mathtt{m}(x,y)\bydef\min\left\{\frac{x}{y},\frac{y}{x}\right\}.$$
\end{lem} 
 \begin{proof}
 First, we recall  the identity (see \cite[identity (4.16)]{HHR23}, for instance)
   $$\frac{1}{2\pi}\int_0^{2\pi}\log|1-xe^{\ii \eta}|\cos(n\eta)d\eta=-\frac{x^n}{2n},$$
   which holds for any $x \in (0,1]$.
Therefore, combining the preceding identity with the parity and periodicity of the trigonometric function $\eta\mapsto\cos(n\eta)$, it follows, for any $x\in [1,\infty)$, that 
$$\frac{1}{2\pi}\int_0^{2\pi}\log|1-xe^{\ii \eta}|\cos(n\eta)d\eta=-\frac{x^{-n}}{2n}\cdot$$
Hence we deduce, for any $x\in [1,\infty)$,  that 
\begin{equation}\label{IDENT:0}
\frac{1}{2\pi}\int_0^{2\pi}\log|1-xe^{\ii \eta}|\cos(n\eta)d\eta=-\frac{ \min^n\{x,x^{-1}\}}{2n}\cdot
\end{equation}
Thus, observing that 
\begin{equation*}
\int_0^{2\pi}\log|ye^{\ii \theta}-xe^{\ii \eta}|\cos(n\eta)d\eta=\log|y|\underbrace{\int_0^{2\pi}\cos(n\eta)d\eta}_{=0}+\int_0^{2\pi}\log\left|1-\frac{x}{y}e^{\ii (\eta-\theta)}\right|\cos(n\eta)d\eta, 
\end{equation*}
and employing the change of variables $\eta\to\eta+\theta$ along with \eqref{IDENT:0} infers, for all $x\in(0,\infty)$ and $\theta\in\mathbb{T}$, that  
\begin{equation*}
\frac{1}{2\pi}\int_0^{2\pi}\log|ye^{\ii \theta}-xe^{\ii \eta}|\cos(n\eta)d\eta=-\cos(n\theta)\frac{\mathtt{m}^n(\frac{x}{y},1)}{2n}+\frac{\sin (n\theta)}{2\pi}\underbrace{\int_0^{2\pi}\log\left|1-\tfrac{x}{y}e^{\ii  \eta}\right|\sin(n\eta)d\eta}_{=0}, 
\end{equation*}
where we have used the fact that the function
$$\eta\mapsto\log\left|1-\tfrac{x}{y}e^{\ii \eta}\right|\sin(n\eta)$$  
is odd and $2\pi$-periodic. In conclusion, \eqref{log:identity} follows and the proof of the lemma is now completed.
\end{proof}

\section*{Acknowledgements} 
T. Hmidi has been supported by Tamkeen under the NYU Abu Dhabi Research Institute grant. E. Roulley has been supported by PRIN 2020XB3EFL, ``Hamiltonain and Dispersive PDEs”.
 \bibliographystyle{plain} 
\bibliography{BIB}
\end{document}